\documentclass[12pt]{amsart}

\usepackage{a4wide}
\usepackage{cite}
\usepackage{tikz}
\usepackage{hyperref}
\usepackage{amsmath,amsthm}
\usepackage{graphicx}
\usepackage{standalone}
\usepackage{mathdots}
\usepackage{makecell}
\usepackage{algorithm,array,ltablex}
\usepackage{subcaption}
\usepackage{tabularx}

\newtheorem{lemma}{Lemma}

\newtheorem{theorem}{Theorem}

\newtheorem{corollary}{Corollary}
\newtheorem{example}{Example}
\newtheorem*{example*}{Example}
\newtheorem{question}{Question}

\newcommand{\seqnum}[1]{\href{https://oeis.org/#1}{#1}}

\newcommand{\gfll}[1]{F_{#1}(x)}
\newcommand{\gfall}{\gfll{A}}
\newcommand{\gff}[2]{F_{#2}^{(#1)}(x)}
\newcommand{\gfa}[1]{\gff{#1}{A}}
\newcommand{\gfnoa}[1]{F^{(#1)}(x)}

\begin{document}
\pagestyle{plain}
\setlength{\parindent}{0 cm}

\title{Counting colored trees}

\author[1]{Stoyan Dimitrov}
\address[Stoyan Dimitrov]{Dartmouth College,
Department of Mathematics,
29 North Main Street,
6188 Kemeny Hall
Hanover, NH 03755, USA}
\email{\textcolor{blue}{\href{mailto:emailtostoyan@gmail.com}{emailtostoyan@gmail.com}}}

\author[2]{Nathan Fox}
\address[Nathan Fox]{Canisius University,
Department of Quantitative Sciences,
2001 Main Street,
Buffalo, NY 14208, USA}
\email{\textcolor{blue}{\href{mailto:fox42@canisius.edu}{fox42@canisius.edu}}}

\author[3]{Kimberly Hadaway}
\address[Kimberly Hadaway]{Iowa State University, Department of Mathematics, 411 Morrill Rd, Ames, IA 50011, USA}
\email{\textcolor{blue}{\href{mailto:kph3@iastate.edu}{kph3@iastate.edu}}}

\author[4]{Ashley Tharp}
\address[Ashley Tharp]{North Carolina School of Science and Mathematics, Department of Mathematics, 1219 Broad St, Durham, NC 27705, USA}
\email{\textcolor{blue}{\href{mailto:ashley.tharp@ncssm.edu}{ashley.tharp@ncssm.edu}}}

\author[5]{Stephan Wagner}
\address[Stephan Wagner]{Institute of Discrete Mathematics, TU Graz, Steyrergasse 30, 8010 Graz, Austria \and Department of Mathematics, Uppsala University, Box 480, 751 06 Uppsala, Sweden}
\email{\textcolor{blue}{\href{mailto:stephan.wagner@tugraz.at}{stephan.wagner@tugraz.at}}}
\thanks{S. Wagner was supported by the Swedish research council (VR), grant 2022-04030.}

\thanks{We wish to thank the American Mathematical Society for organizing the Mathematics Research Community workshop where this work began. This material is based on work supported by the National Science Foundation under Grant Number DMS 1641020.}

\begin{abstract}
We consider the enumeration of plane trees (rooted ordered trees) whose vertices are colored according to a specific coloring rule that prescribes which possible pairs of colors can occur as the colors of a parent vertex and its child. This general construction covers many different examples that have been studied in the literature.
Some general necessary and sufficient conditions for two different coloring rules to result in the same counting sequence are established. 
We also provide exhaustive lists of counting sequences arising from coloring rules with two or three colors, and we find formulas and closed form expressions for many of these sequences. 
The famous Fibonacci, Catalan, Narayana, and Schr\"oder sequences appear in several cases. Some of these coloring rules are extended to families of coloring rules with arbitrarily many colors.
\end{abstract}

\maketitle

\section{Introduction}\label{sec:intro}

The enumeration of different kinds of trees is a classical topic within enumerative combinatorics. In this article, we consider a class of enumeration problems related to \emph{plane trees} (rooted ordered trees), which are a well-known example of a class of combinatorial objects counted by the \emph{Catalan numbers} (\seqnum{A000108})\footnote{Here and in the following, we refer to sequences in the On-Line Encyclopedia of Integer Sequences (OEIS) \cite{OEIS}.}. A plane tree is an unlabeled rooted tree where vertices can have arbitrary degree and the left-to-right order of children matters. For example, Figure~\ref{fig:plane-4} depicts all plane trees on four vertices.
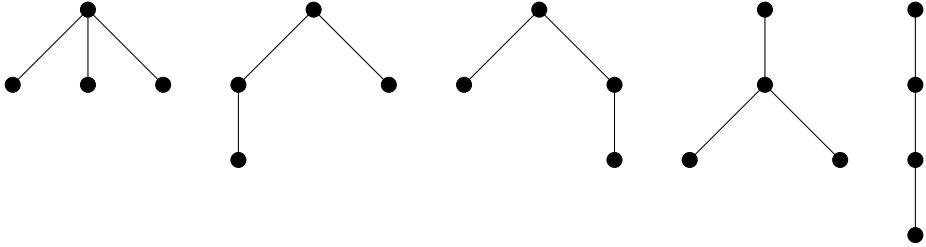
\begin{figure}[htbp]
\begin{tikzpicture}
\tikzset{point/.style = {circle, fill=black, draw=black, inner sep=2pt}}
    \node[point] at (0,0) {};
    \node[point] at (-1,-1) {};
    \node[point] at (0,-1) {};
    \node[point] at (1,-1) {};

\draw (-1,-1)--(0,0)--(1,-1);
\draw (0,0)--(0,-1);

    \node[point] at (3,0) {};
    \node[point] at (2,-1) {};
    \node[point] at (4,-1) {};
    \node[point] at (2,-2) {};

\draw (2,-2)--(2,-1)--(3,0)--(4,-1);

    \node[point] at (6,0) {};
    \node[point] at (5,-1) {};
    \node[point] at (7,-1) {};
    \node[point] at (7,-2) {};

\draw (7,-2)--(7,-1)--(6,0)--(5,-1);

    \node[point] at (9,0) {};
    \node[point] at (9,-1) {};
    \node[point] at (8,-2) {};
    \node[point] at (10,-2) {};

\draw (10,-2)--(9,-1)--(9,0);
\draw (8,-2)--(9,-1);

    \node[point] at (11,0) {};
    \node[point] at (11,-1) {};
    \node[point] at (11,-2) {};
    \node[point] at (11,-3) {};

\draw (11,0)--(11,-3);

\end{tikzpicture}
 \caption{All $4$-vertex plane trees.}\label{fig:plane-4}
\end{figure}
The number of plane trees with $n$ vertices is
$$C_{n-1} = \frac{1}{n} \binom{2n-2}{n-1}.$$
Other classical combinatorial sequences can be interpreted in terms of plane trees, as well. One example are the \emph{Narayana numbers} (\seqnum{A001263})
$$N_{n,k} = \frac{1}{n} \binom{n}{k} \binom{n}{k-1}.$$
Indeed, $N_{n-1,k}$ is the number of plane trees with $k$ leaves and $n$ vertices in total. The \emph{large Schr\"oder numbers} (\seqnum{A006318}) are another such example. They can be expressed as
\begin{equation}\label{eq:schroder}
R_n = \sum_{k=1}^n N_{n,k} 2^k,
\end{equation}
which has a simple (and well-known) combinatorial interpretation: The large Schr\"oder number $R_{n-1}$ counts plane trees with $n$ vertices and bicolored leaves (i.e., each leaf can have one of two possible colors). See for example \cite[Theorem 3.5]{CLS2007Butterfly} or \cite[Example 18]{Barry2016Riordan}. We discuss large Schr\"oder numbers more in Section~\ref{sec:small}. Figure~\ref{fig:Schr-3} shows all such plane trees on three vertices with bicolored leaves.
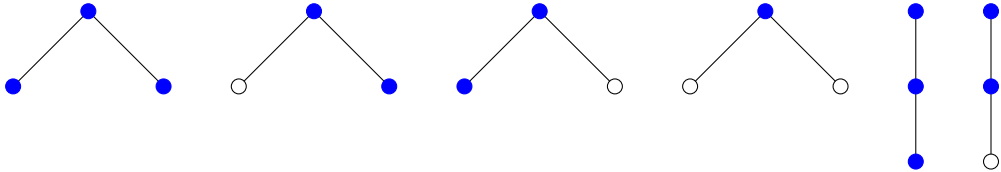
\begin{figure}[htbp]
\begin{tikzpicture}
\tikzset{blpoint/.style = {circle, fill=blue, draw=blue, inner sep=2pt}}
\tikzset{wpoint/.style = {circle, fill=white, draw=black, inner sep=2pt}}
\draw (-1,-1)--(0,0)--(1,-1);

    \node[blpoint] at (0,0) {};
    \node[blpoint] at (-1,-1) {};
    \node[blpoint] at (1,-1) {};

\draw (2,-1)--(3,0)--(4,-1);

    \node[blpoint] at (3,0) {};
    \node[wpoint] at (2,-1) {};
    \node[blpoint] at (4,-1) {};

\draw (5,-1)--(6,0)--(7,-1);

    \node[blpoint] at (6,0) {};
    \node[blpoint] at (5,-1) {};
    \node[wpoint] at (7,-1) {};

\draw (8,-1)--(9,0)--(10,-1);

    \node[blpoint] at (9,0) {};
    \node[wpoint] at (8,-1) {};
    \node[wpoint] at (10,-1) {};

\draw (11,0)--(11,-2);

    \node[blpoint] at (11,0) {};
    \node[blpoint] at (11,-1) {};
    \node[blpoint] at (11,-2) {};

\draw (12,0)--(12,-2);

    \node[blpoint] at (12,0) {};
    \node[blpoint] at (12,-1) {};
    \node[wpoint] at (12,-2) {};

\end{tikzpicture}
\caption{All $3$-vertex plane trees with leaves that can either be blue or white (while all other vertices are blue).}\label{fig:Schr-3}
\end{figure}
This generalizes in a natural way to more colors, see OEIS sequences \seqnum{A047891}, \seqnum{A082298} and \seqnum{A082301} for plane trees with $3$-colored, $4$-colored, and $5$-colored leaves, respectively. 
These counting sequences can also be interpreted in terms of colored/weighted lattice paths, see \cite{Barry2019Generalized,CP2017Identities}.
Another type of colored plane trees that are counted by the Catalan-like sequence \seqnum{A007226} stems from the following result due to Kirschenhofer, Prodinger and Tichy~\cite{KPT1986Fibonacci}:

\begin{theorem}[{see~\cite{KPT1986Fibonacci}}]\label{thm:indsets}
The total number of independent sets over all plane trees with $n$ vertices is
$$\frac{2}{n} \binom{3n-3}{n-1}.$$
\end{theorem}

This can also be interpreted in terms of colored trees: for every pair of a plane tree and an independent set of that tree, let us assign the color white to vertices that belong to the independent set and the color blue to all other vertices. The result is a blue-white colored plane tree in which no two white vertices are adjacent. Thus we have the following equivalent theorem.

\begin{theorem}\label{thm:red-blue}
The number of blue-white colored plane trees with $n$ vertices in which no two white vertices are adjacent is
$$\frac{2}{n} \binom{3n-3}{n-1}.$$
\end{theorem}

A generalization to several colors can be found in \cite{GPW2010Bijections}:

\begin{theorem}[{see~\cite{GPW2010Bijections}}]\label{thm:k-plane}
The number of plane trees with $n$ vertices that are given labels between $1$ and $m$ in such a way that the labels of adjacent vertices never add up to a number greater than $m+1$ is
$$\frac{m}{n} \binom{(m+1)(n-1)}{n-1}.$$
\end{theorem}

The labels can also be interpreted as colors 
$1,2,\ldots,m$ with the restriction that vertices of colors $i$ and $j$ may only be adjacent if $i+j \leq k+1$. See also \cite{OW2024Refined} for a refinement of Theorem~\ref{thm:k-plane}. The trees in Theorem~\ref{thm:k-plane} are called $m$-\emph{plane trees}.

In a recent paper \cite{BC2024Lambda}, an even more general kind of colored plane tree was considered in the context of random matrices: let $\lambda$ be an integer partition, i.e., a nonincreasing finite sequence $\lambda = (\lambda_1,\lambda_2,\ldots,\lambda_m)$ of positive integers. A $\lambda$-\emph{plane tree} is defined as a plane tree in which the vertices are colored using $m$ colors 
$1,2,\ldots,m$,
and a vertex of color $i$ can only have children of color $j$ if $j \leq \lambda_i$. The aforementioned family of $m$-plane trees is obtained as the special case of this construction where $\lambda = (m,m-1,\ldots,2,1)$.

The aim of this paper is to study all these examples of colored plane trees in a unified framework. In general, let us take a set of $m$ colors and an $m \times m$ matrix $A$ with entries $a_{ij} \in \{0,1\}$. We color the vertices of a rooted tree according to the following rule:

\begin{center}
A vertex of color $j$ can only be a child of a vertex of color $i$ if $a_{ij} = 1$.
\end{center}

We are interested in the number of trees $t_A(n)$ with $n$ vertices that are colored according to this rule. All types of trees mentioned so far fit this general scheme. Let us briefly discuss these and other examples.

\begin{enumerate}
\item Of course, the class of plane trees itself can be represented by the $1 \times 1$-matrix consisting of a single $1$. Generally, if each vertex can be colored in $1$ out of $m$ colors, then let $A$ be the $m \times m$-matrix whose entries are all equal to $1$. We have $t_A(n) = m^n C_{n-1}$, since every tree can be colored in $m^n$ ways.
\item The example of plane trees with bicolored leaves (as shown in Figure~\ref{fig:Schr-3}) can be represented by the $2 \times 2$-matrix
$$A = \begin{bmatrix} 1 & 1 \\ 0 & 0 \end{bmatrix}.$$
The first row corresponds to the color blue in the figure: children of a vertex that is colored blue can have any color. The second row corresponds to the color white in the figure. Since there is no possible color for children of a white vertex, these vertices must necessarily be leaves.
Thus counting plane trees with bicolored leaves: internal vertices are necessarily blue, while leaves can be either blue or white. It follows that $t_A(n)$ is the sequence of large Schr\"oder numbers in this example.
\item Theorem~\ref{thm:red-blue} corresponds to the matrix
$$A = \begin{bmatrix} 1 & 1 \\ 1 & 0 \end{bmatrix},$$
where the color blue corresponds to the first row and the color white corresponds to the second row. All possible parent-child combinations are allowed, except for two neighboring white vertices. So we have $t_A(n) = \frac{2}{n} \binom{3n-3}{n-1}$ in this case.
\item The more general example of $m$-plane trees is modeled by the $m \times m$-matrix of the form
$$A_m = \begin{bmatrix} 1 & 1 & \cdots & 1 & 1 \\ 1 & 1 & \cdots & 1 & 0 \\ \vdots & \vdots & \iddots & \vdots & \vdots \\ 1 & 1 & \cdots & 0 & 0 \\ 1 & 0 & \cdots & 0 & 0  \end{bmatrix}.$$
So Theorem~\ref{thm:k-plane} can be expressed as $t_{A_m}(n) = \frac{m}{n} \binom{(m+1)(n-1)}{n-1}$.
\end{enumerate}

Looking at further examples, one quickly notices that the sequence $t_A(n)$ does not uniquely determine the matrix $A$. For instance, the pair of matrices
$$A = \begin{bmatrix} 1 & 0 \\ 0 & 1 \end{bmatrix} \text{ and } B = \begin{bmatrix} 0 & 1 \\ 1 & 0 \end{bmatrix},$$
have the same counting sequence given by twice the Catalan numbers: $t_A(n) = t_B(n) = 2C_{n-1}$. In each case, the color of the root already determines the entire coloring. For the coloring rule encoded by matrix $A$, all vertices have to have the same color, while for the coloring rule encoded by matrix $B$, the colors have to alternate.

There are also more complicated examples, such as the following pair of matrices:
$$A = \begin{bmatrix} 1 & 1 & 0 \\ 1 & 0 & 0 \\ 1 & 1 & 0 \end{bmatrix} \text{ and } B = \begin{bmatrix} 1 & 1 & 0 \\ 0 & 0 & 1 \\ 1 & 1 & 0 \end{bmatrix}.$$
Both produce the counting sequence $3, 5, 17, 72, 341, 1729,\ldots$
with the general formula $t_A(n) = t_B(n) = \frac{7 n - 4}{n (2 n - 1)} \binom{3n-3}{n-1}$. The following general question thus arises:

\begin{question}
Can we characterize pairs of matrices $A$ and $B$ for which the counting sequences $t_A(n)$ and $t_B(n)$ coincide?
\end{question}

We will call two matrices with this property \emph{tree coloring equivalent}. A full characterization of tree coloring equivalence seems out of reach. In this paper, we provide some necessary conditions for matrices to be tree coloring equivalent (see Section~\ref{sec:necessary}). In Section~\ref{sec:constructions}, we discuss some ways to construct tree coloring equivalent pairs of matrices.

Most of the examples mentioned so far lead to ``nice'' formulas for the sequence $t_A(n)$ (i.e., hypergeometric expressions). This raises another general question.

\begin{question}
\label{q:2}
When is there an explicit formula for the counting sequence $t_A(n)$. In particular, when is this sequence hypergeometric?
\end{question}
A sequence $a_n$ is hypergeometric if the quotient $\frac{a_n}{a_{n-1}}$ is a rational function of $n$.

Again, a general answer of Question~\ref{q:2} seems out of reach. There are eight different sequences that arise from $2 \times 2$-matrices, and 72 different sequences that arise from $3 \times 3$-sequences. These are discussed in detail in Section~\ref{sec:small}. A total of $13$ of the sequences for $3\times3$ matrices are nonconstant hypergeometric sequences, and there are also some other counting sequences that are not hypergeometric but still have explicit counting formulas. For example, there are several $3 \times 3$-matrices for which the associated counting sequence is $(2^n+1)C_{n-1}$, e.g. see Entry 29 in Appendix~\ref{app:3by3}. The latter expression can be easily shown to not be hypergeometric. 
\subsection{Other related work and summary of the paper}

Knuth \cite[Section~3]{knuth1968another} considers essentially the same problem, but for oriented trees (arborescences), in which the order of the children for each vertex does not matter, yet they still have a root and all edges are oriented towards it. 
His article establishes counting formulas even when we have a given number of vertices of each color, generalizing Cayley's formula on the total number of different spanning trees. 
A further generalization was obtained by Bernardi and Morales \cite{bernardi2014counting}. 
Another recent paper~\cite{white2024quota} studies quota trees, which are quite general objects with a certain specialization that is quite similar to our setting.

It is also worth mentioning that via the glove bijection, which maps ordered trees to Dyck paths (see Section~\ref{sec:prelim} for a description), the problems we are studying can be considered as problems of counting colored Dyck paths. 
While some papers consider Dyck paths with colored accents or hills \cite{asinowski2008dyck,manesbijections}, the coloring we look at does not seem to be considered.     

The rest of the paper is summarized as follows. In Section~\ref{sec:prelim}, we start by discussing some needed background. In Section~\ref{sec:necessary} and Section~\ref{sec:constructions}, we establish some necessary and sufficient conditions for two coloring rules to have the same counting sequence. 
Next, in Section~\ref{sec:new_from_old}, we explore how certain coloring rules can be built out of simpler rules. 
As previously mentioned, Section~\ref{sec:small} contains an exhaustive analysis of coloring rules with two or three colors.
Then, in Section~\ref{sec:special}, we describe some infinite families of matrices for which there are explicit formulas for the associated counting sequences. 
Finally, Section~\ref{sec:future} outlines some possible directions for further study.

\section{Preliminaries}\label{sec:prelim}

As mentioned previously, plane trees form a classical example of a class of Catalan objects. In several proofs, we refer to the \emph{standard Catalan decomposition} of plane trees. Given a plane tree with $n\geq2$ vertices, decompose it into two pieces:
\begin{itemize}
    \item The tree rooted at the rightmost child of the root $(R)$. 
    \item The rest of the tree (including the root) $(L)$.
\end{itemize}
In such a decomposition of a tree $T$ with $n$ vertices, $L$ has $k$ vertices (with $1\leq k\leq n-1$), and $R$ has the remaining $n-k$ vertices (see Figure~\ref{fig:cat-standard}). 
This decomposition demonstrates that the number of plane trees on $n$ vertices satisfies the same recurrence as the Catalan numbers.
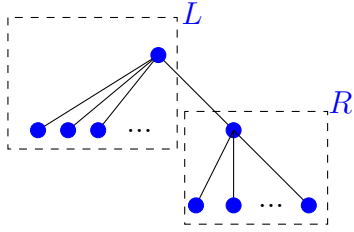
\begin{figure}
    \centering
\begin{tikzpicture}
\tikzset{blpoint/.style = {circle, fill=blue, draw=blue, inner sep=2pt}}
\tikzset{wpoint/.style = {circle, fill=white, draw=black, inner sep=2pt}}
\draw (0,0)--(1,-1);
\draw (0,0)--(-1.6,-1);
\draw (0,0)--(-1.2,-1);
\draw (0,0)--(-0.8,-1);

    \node[blpoint] at (0,0) {};
    \node[blpoint] at (-1.6,-1) {};
    \node[blpoint] at (-1.2,-1) {};
    \node[blpoint] at (-0.8,-1) {};
    \node[] at (-0.25,-1) {...};
    \node[blpoint] at (1,-1) {};

\draw (1,-1)--(0.5,-2);
\draw (1,-1)--(1,-2);
    \node[] at (1.5,-2) {...};
\draw (1,-1)--(2,-2);

    \node[blpoint] at (0.5,-2) {};
    \node[blpoint] at (1,-2) {};
    \node[blpoint] at (2,-2) {};
    
\draw[dashed] (-2,0.5) rectangle (0.25,-1.25);
\node[blue] at (0.45,0.557) {$L$};

\draw[dashed] (0.35,-0.75) rectangle (2.25,-2.25);
\node[blue] at (2.42,-0.625) {$R$};

\end{tikzpicture}

    \caption{The standard Catalan decomposition for plane trees}
    \label{fig:cat-standard}
\end{figure}
Also, in several of our proofs, we use a folklore bijection between ordered trees and Dyck paths, known as the \emph{glove} or the \emph{accordion} bijection \cite{callan}: Execute a pre-order traversal of an ordered tree (i.e., run a depth-first search while prioritizing exploring to the left).
During this process, write an up-step whenever we go further away from the root, and a down-step every time we go closer to the root. 
The result is a 
Dyck path, and it is straightforward to show the map is bijective (see Figure~\ref{fig:glove} below).

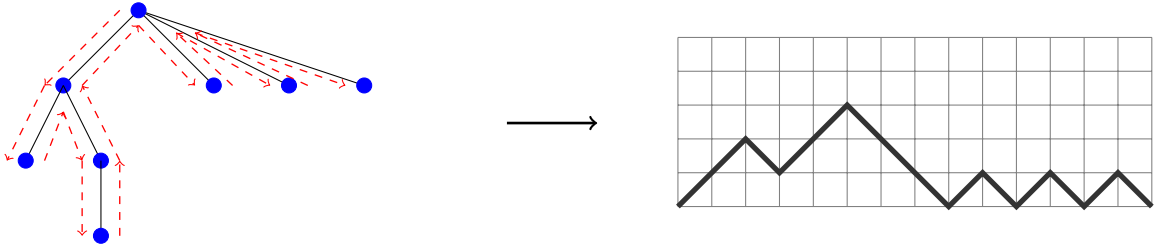
\begin{figure}[ht!]
    \centering
\begin{minipage}{0.4\textwidth}
  \centering
\begin{tikzpicture}
\tikzset{blpoint/.style = {circle, fill=blue, draw=blue, inner sep=2pt}}
\tikzset{wpoint/.style = {circle, fill=white, draw=black, inner sep=2pt}}
\draw (0,0)--(1,-1);
\draw (0,0)--(3,-1);
\draw (0,0)--(2,-1);
\draw (0,0)--(-1,-1);

    \node[blpoint] at (0,0) {};
    \node[blpoint] at (3,-1) {};
    \node[blpoint] at (2,-1) {};
    \node[blpoint] at (1,-1) {};
    \node[blpoint] at (-1,-1) {};

\draw (-1,-1)--(-0.5,-2);
\draw (-1,-1)--(-1.5,-2);

    \node[blpoint] at (-0.5,-2) {};
    \node[blpoint] at (-1.5,-2) {};

    \draw (-0.5,-3)--(-0.5,-2);
    \node[blpoint] at (-0.5,-3) {};

    \draw[->, red, dashed, line width=0.5pt] (-0.25,0) -- (-1.25,-1);
    \draw[->, red, dashed, line width=0.5pt] (-1.25,-1) -- (-1.75,-2);
    \draw[->, red, dashed, line width=0.5pt] (-1.25,-2) -- (-1,-1.35);
    \draw[->, red, dashed, line width=0.5pt] (-1,-1.35) -- (-0.75,-2);
    \draw[->, red, dashed, line width=0.5pt] (-0.75,-2) -- (-0.75,-3);
    \draw[->, red, dashed, line width=0.5pt] (-0.25,-3) -- (-0.25,-2);
    \draw[->, red, dashed, line width=0.5pt] (-0.25,-2) -- (-0.75,-1);
    \draw[->, red, dashed, line width=0.5pt] (-0.75,-1) -- (0,-0.2);
    \draw[->, red, dashed, line width=0.5pt] (0,-0.2) -- (0.75,-1);

    \draw[->, red, dashed, line width=0.5pt] (1.25,-1) -- (0.5,-0.3);
    \draw[->, red, dashed, line width=0.5pt] (0.5,-0.3) -- (1.75,-1);

    \draw[->, red, dashed, line width=0.5pt] (2.25,-1) -- (0.75,-0.3);
    \draw[->, red, dashed, line width=0.5pt] (0.75,-0.3) -- (2.75,-1);
\end{tikzpicture}
\end{minipage}
\hspace{0.5cm}
\begin{minipage}[c]{0.1\textwidth}
  \centering
  \begin{tikzpicture}
    \draw[->, line width=1pt] (0,0) -- (1.2,0);
  \end{tikzpicture}
\end{minipage}
\hspace{0.5cm}
\begin{minipage}{0.4\textwidth}
  \centering
  \begin{tikzpicture}[scale=.45, roundnode/.style={circle,fill=black!60, inner sep=1.5pt, minimum width=4pt}, triangnode/.style={regular polygon,regular polygon sides=3, fill=black!60, inner sep=1.5pt, outer sep = 0pt}]
        \draw[step=1.0, gray!100, thin] (0,0) grid (14, 5);
	\draw[black!80, line width=2pt] (0,0) -- (1,1) -- (2,2) -- (3,1) -- (4,2) -- (5,3) -- (6,2) -- (7,1) -- (8,0) -- (9,1) -- (10,0) -- (11,1) -- (12,0) -- (13,1) -- (14,0);
    \end{tikzpicture}
\end{minipage}
    \caption{Example of the glove bijection for plane trees.}
    \label{fig:glove}
\end{figure}

Let us define some more notation and terminology. Let $\mathcal{T}_{n}$ denotes the set of rooted plane trees with $n$ vertices. We refer to a square zero-one matrix as a \emph{coloring matrix} since such a matrix encodes a coloring rule. 
Letting $A$ be a fixed $m \times m$ 
coloring matrix, for $1\leq i\leq m$, we let 
$$
t_A^{(i)}(n) =  \textit{the number of trees in $\mathcal{T}_{n}$ with root color $i$ and coloring rule given by $A$.}
$$ 

We also use the convention $t^{(i)}_A(0)=0$, for all $i$. 
Immediately, we have $t_A(n)=\sum_{i=1}^mt_A^{(i)}(n)$ for all $n$. 
This leads to the following simple result.
\begin{theorem}\label{prop:t2}
    Given an $m\times m$ coloring matrix $A$, for any $1\leq i\leq m$, $t^{(i)}_A(2)$ is equal to the number of ones in row $i$ of $A$. Consequently, $t_A(2)$ is equal to the number of ones in $A$.
\end{theorem}
\begin{proof}
    The only tree with two vertices has a root and one child of the root. If the root has color $i$, then the options for coloring the child correspond to the ones in row $i$. Considering all possible root colors proves the remainder of the theorem.
\end{proof}

With the previous ideas in mind, we introduce a stronger notion of tree coloring equivalence that we sometimes need: Two $m\times m$ coloring matrices $A$ and $B$ are \emph{strictly tree coloring equivalent} if the sequences $t_A^{(i)}(n)$ and $t_B^{(i)}(n)$ are equal for all $1\leq i\leq m$. Clearly, if $A$ and $B$ are strictly tree coloring equivalent, then they are tree coloring equivalent, yet the converse is not true.

Given the $m\times m$ coloring matrix $A$, define $\gfall=\sum_{n=1}^\infty t_A(n)x^n$ to be the ordinary generating function for the coloring rule determined by $A$. Similarly, for $1\leq i\leq m$, define $\gfa{i}=\sum_{n=1}^\infty t_A^{(i)}(n)x^n$. We have the following observations.
\begin{theorem}\label{prop:rootrec}
    Let $A$ be an $m\times m$ coloring matrix, and let $1\leq i\leq m$ be fixed. For $n\geq2$, we have
    \[
    t_A^{(i)}(n)=\sum_{k=1}^{n-1}\sum_{j=1}^ma_{ij}t_A^{(i)}(k)t_A^{(j)}(n-k),
    \]
    where $a_{ij}$ denotes the $(i,j)$-entry of $A$.
\end{theorem}
\begin{proof}
    Consider the standard Catalan decomposition of a rooted plane tree $T$ with $n\geq 2$ vertices into pieces $L$ with $k$ vertices and $R$ with $n-k$ vertices. 
    If $T$ has root color $i$, then $L$ has root color $i$, and the root color of $R$ is some color $j$ for which $a_{ij}=1$. In fact, $L$ and $R$ can be any plane trees of the appropriate order with those root colors. 
    The recurrence now follows by considering all possibilities of root colors.
\end{proof}

Theorem~\ref{prop:rootrec} leads to a corresponding result about generating functions.
\begin{theorem}\label{prop:gfgeneral}
    Let $A$ be an $m\times m$ coloring matrix, and fix $1\leq i\leq m$. We have
    \[
    \gfa{i}=x+\gfa{i}\sum_{j=1}^ma_{ij}\gfa{j}.
    \]
\end{theorem}
\begin{proof}
    First, note that, for each $1\leq i\leq m$, $t_A^{(i)}(0)=0$ (conventionally) 
    and $t_A^{(i)}(1)=1$, since there is a single way to color the tree with one vertex when the root is forced to have color $i$. Theorem~\ref{prop:rootrec} applies when $n\geq2$, so we have
    \begin{align*}
        \gfa{i}&=\sum_{n=0}^\infty t_A^{(i)}(n)x^n\\
        &=x+\sum_{n=2}^\infty\sum_{k=1}^{n-1}\sum_{j=1}^ma_{ij}t_A^{(i)}(k)t_A^{(j)}(n-k)x^n\\
        &=x+\gfa{i}\sum_{j=1}^ma_{ij}\gfa{j},
    \end{align*}
    as required.
\end{proof}

Based on Theorem~\ref{prop:gfgeneral}, we see that the $m$ generating functions corresponding to the $m$ root colors for a rule defined by an $m\times m$ coloring matrix satisfy a system of polynomial equations. This leads immediately to the following corollary.
\begin{corollary}\label{cor:algebraicgf}
    Let $A$ be an $m\times m$ coloring matrix. The function $\gfall$ is algebraic, as are the functions $\gfa{i}$ for $1\leq i\leq m$.
\end{corollary}

We conclude this section with some extensions of our notation to a somewhat more general setting. For integers $0\leq q<n$ and fixed $m\times m$ coloring matrix $A$, denote by $t_A(n,q)$ the number of $m$-colored $n$-vertex trees where exactly $q$ edges violate the coloring rule specified by $A$, and define $t_A^{(i)}(n,q)$ analogously for $1\leq i\leq m$. Using this notation, we have that $t_A(n,0)=t_A(n)$ and that $t_A^{(i)}(n,0)=t_A^{(i)}(n)$ for each $1\leq i\leq m$. Conventionally, we also define $t_A^{(i)}(n,q)=0$ whenever $q<0$ or $q\geq n$ (and similarly for $t_A(n,q)$).

We also define bivariate analogues of our generating functions: Let
\[
F_A(x,u)=\sum_{n=1}^\infty\sum_{q=0}^{n-1}t_A(n,q)x^nu^q,
\]
and for each $1\leq i\leq m$ let
\[
F_A^{(i)}(x,u)=\sum_{n=1}^\infty\sum_{q=0}^{n-1}t_A^{(i)}(n,q)x^nu^q.
\]

We have analogues of Theorems~\ref{prop:rootrec} and~\ref{prop:gfgeneral} for this generalization.

\begin{theorem}\label{prop:rootrecgeneral}
    Let $A$ be an $m\times m$ coloring matrix, and let $1\leq i\leq m$ be fixed. For $n\geq2$ and $0\leq q<n$, we have
    \[
    t_A^{(i)}(n,q)=\sum_{k=1}^{n-1}\sum_{\ell=0}^qt_A^{(i)}(k,\ell)\sum_{j=1}^m\left(a_{ij}t_A^{(j)}(n-k,q-\ell)+(1-a_{ij})t_A^{(j)}(n-k,q-\ell-1)\right),
    \]
    where $a_{ij}$ denotes the $(i,j)$-entry of $A$.
\end{theorem}
\begin{proof}
    We prove first that
    \begin{align*}
        t_A^{(i)}(n,q)=\sum_{k=1}^{n-1}\sum_{j=1}^m&\left(\sum_{\ell=0}^qa_{ij}t_A^{(i)}(k,\ell)t_A^{(j)}(n-k,q-\ell)\right.\\&+\left.\sum_{\ell=0}^{q-1}(1-a_{ij})t_A^{(i)}(k,\ell)t_A^{(j)}(n-k,q-\ell-1)\right).
    \end{align*}
    Consider the standard Catalan decomposition of a rooted plane tree $T$ with $n\geq 2$ vertices into pieces $L$ with $k$ vertices and $R$ with $n-k$ vertices. 
    If $T$ has root color $i$, then $L$ has root color $i$, and the root color of $R$ is some color $j$. There are two cases to consider. First, if $a_{ij}=1$, then the rightmost edge from the root does not violate the coloring rule. This means that there are a total of $q$ coloring violations split between $L$ and $R$. The first inner summation enumerates all such colorings. On the other hand, if $a_{ij}=0$, then the rightmost edge from the root does violate the coloring. This means that there are a total of $q-1$ coloring violations split between $L$ and $R$. The second inner summation enumerates all such colorings.
    The recurrence now follows by considering all possibilities of root colors.

    Now, we observe that
    \[
    \sum_{\ell=0}^{q-1}(1-a_{ij})t_A^{(i)}(k,\ell)t_A^{(j)}(n-k,q-\ell-1)=\sum_{\ell=0}^{q}(1-a_{ij})t_A^{(i)}(k,\ell)t_A^{(j)}(n-k,q-\ell-1)
    \]
    by the convention $t_A^{(j)}(n-k,-1)=0$ for every value of $n-k$. The desired result then follows by combining the two inner sums and then exchanging the order of summation and factoring out $t_A^{(i)}(k,\ell)$.
\end{proof}

\begin{theorem}\label{prop:gfgeneralgeneral}
    Let $A$ be an $m\times m$ coloring matrix, and fix $1\leq i\leq m$. We have
    \[
    F_A^{(i)}(x,u)=x+F_A^{(i)}(x,u)\sum_{j=1}^m(a_{ij} + (1-a_{ij})u) F_A^{(j)}(x,u).
    \]
\end{theorem}
\begin{proof}
    First, note that $t_A^{(i)}(1,0)=1$, since there is a single way to color the tree with one vertex when the root is forced to have color $i$. Also, note that if $n\in\{0,1\}$ and $(n,q)\neq(1,0)$, $t_A^{(i)}(n,q)=0$ (conventionally).
    Theorem~\ref{prop:rootrecgeneral} applies when $n\geq2$, so we have
    \begin{align*}
        &\phantom{=}F_A^{(i)}(x,u)=\sum_{n=0}^\infty t_A^{(i)}(n,q)x^nu^q\\
        &=x+\sum_{n=2}^\infty\sum_{k=1}^{n-1}\sum_{\ell=0}^qt_A^{(i)}(k,\ell)\sum_{j=1}^m\left(a_{ij}t_A^{(j)}(n-k,q-\ell)+(1-a_{ij})t_A^{(j)}(n-k,q-\ell-1)\right)x^nu^q\\
        &=x+F_A^{(i)}(x,u)\sum_{j=1}^m(a_{ij} + (1-a_{ij})u) F_A^{(j)}(x,u),
    \end{align*}
    as required.
\end{proof}

\section{Tree coloring equivalence}

In this section, we study conditions that let us say something about tree coloring equivalence.

\subsection{Necessary conditions}\label{sec:necessary}

First, we give some necessary conditions for two coloring matrices $A$ and $B$ to be tree coloring equivalent. 
\begin{theorem}\label{prop:equalones}
    Let $A$ and $B$ be two $m\times m$ coloring matrices. If $A$ and $B$ are tree coloring equivalent, then $A$ and $B$ have the same numbers of ones. 
    Also, if $A$ and $B$ are strictly tree coloring equivalent, then corresponding rows of $A$ and $B$ have the same number of ones.
\end{theorem}
\begin{proof}
    If $A$ and $B$ are tree coloring equivalent, then $t_A(2)=t_B(2)$. By Theorem~\ref{prop:t2}, this means that the number of ones in $A$ is the same as the number of ones in $B$, as required. Analogously, if $A$ and $B$ are strictly tree coloring equivalent, then for each $1\leq i\leq m$ we have $t^{(i)}_A(2)=t^{(i)}_B(2)$. By Theorem~\ref{prop:t2}, this means that the number of ones in row $i$ of $A$ is the same as the number of ones in row $i$ of $B$, as required.
\end{proof}

\begin{theorem} \label{prop:3mtns+3lines}
     Let $A$ and $B$ be two $m\times m$ coloring matrices. If $A$ and $B$ are tree coloring equivalent, then the sum of the entries in $A^T A+A^2$ equals the sum of the entries in $B^T B+B^2$.
\end{theorem}

To prove this theorem, we need the following lemmas.

\begin{lemma}\label{lem:3mtns}
Let $A$ be an $m\times m$ coloring matrix. The sum of the entries in $A^T A$ is the number of colored plane trees with three vertices, where the root has exactly two children.
\end{lemma}
\begin{proof}
    Suppose we have a root of color $i$, and that it has two children in colors $j$ and $k$ (not necessarily distinct). 
    Consider the sum
    \[
    \sum_{i=1}^m \sum_{j=1}^m \sum_{k=1}^m a_{ij}a_{ik}.
    \]
    Since $a_{ij}=a^T_{ji}$, 
    we claim that this is exactly the sum of all entries in the matrix $A^T A$, as desired.
\end{proof}
\begin{lemma}\label{lem:all_lines}
Let $A$ be an $m\times m$ coloring matrix. For $n\geq0$, the sum of the entries in $A^n$ is the number of colored rooted paths on $n+1$ vertices.
\end{lemma}
\begin{proof}
    Note that we may interpret $A$ 
    as an adjacency matrix for a directed graph on vertex set $[m]$. 
    As such, the sum of entries in $A^n$ is the number of walks of length $n$ in this directed graph. 
    There is a bijection between such walks and colored rooted paths on $n+1$ vertices: 
    Apply colors down the path corresponding to vertex labels along the walk. 
    In other words, when we sum the entries in $A^n$, we get exactly the number of colorings of a rooted path with $n+1$ vertices.
\end{proof}

We are now ready to prove Theorem \ref{prop:3mtns+3lines}.

\begin{proof}[Proof of Theorem \ref{prop:3mtns+3lines}]
    There are two rooted trees with three vertices: the tree where the root has exactly two children and the path of length $2$. By combining Lemmas~\ref{lem:3mtns} and~\ref{lem:all_lines}, we thereby see that $t_A(3)$ equals the sum of the entries in $A^TA+A^2$. Likewise, $t_B(3)$ equals the sum of the entries in $B^TB+B^2$. In order for $A$ and $B$ to be tree coloring equivalent, we must have $t_A(3)=t_B(3)$. 
    This concludes the proof.
\end{proof}

\subsection{Sufficient conditions}\label{sec:constructions}

Now, we discuss some conditions that guarantee tree coloring equivalence. 
First, we state and prove the obvious fact that permuting colors results in matrices that are tree coloring equivalent.

\begin{theorem}\label{prop:permmat}
If coloring matrices $A$ and $B$ are equal up to permutation of rows and columns, i.e., there is a permutation matrix $P$ such that $P^TAP = B$, then $A$ and $B$ are tree coloring equivalent.
\end{theorem}

\begin{proof}
    Suppose $A$ and $B$ are coloring matrices satisfying $P^TAP = B$ for some permutation matrix $P$, that is, a zero-one matrix with exactly one $1$ in each row and column. 
    By definition of matrix multiplication, $P$ simply permutes rows and columns of $A$ to yield $B$.
    In other words, without loss of generality, for each color $i$ used in $A$, $P$ renames this color to $j$ which is used in $B$.
    In this way, we can construct a bijection from the colors used in $A$ to the colors used in $B$, and hence, these matrices are tree coloring equivalent.
\end{proof}

Whereas the condition in Theorem~\ref{prop:permmat} is sufficient, it is not necessary. In fact, there are many pairs of matrices that are tree coloring equivalent that are not related by Theorem~\ref{prop:permmat}. Section~\ref{sec:intro} includes multiple such examples. 
Keeping Theorem~\ref{prop:permmat} in mind, we say that coloring matrices $A$ and $B$ are \emph{isomorphic} if $A$ and $B$ satisfy $P^TAP = B$ for some permutation matrix $P$. 
Then, recalling our earlier definition of strict tree coloring equivalence, we say that $A$ and $B$ are \emph{strongly tree coloring equivalent} if there is a matrix $A'$ that is isomorphic to $A$ and strictly tree coloring equivalent to $B$. 
Using the permutation given by this isomorphism, we equivalently have that 
$A$ and $B$ are strongly tree coloring equivalent if the multisets 
of sequences $\{t_A^{(1)}(n),t_A^{(2)}(n),\ldots,t_A^{(m)}(n)\}$ and $\{t_B^{(1)}(n),t_B^{(2)}(n),\ldots,t_B^{(m)}(n)\}$ are equal. 
Note that strict tree coloring equivalence implies strong tree coloring equivalence and that strong tree coloring equivalence implies tree coloring equivalence, but neither converse holds.

The following theorem gives a sufficient condition for two matrices to be strictly tree coloring equivalent, and therefore, also a sufficient condition for two matrices to be tree coloring equivalent.
\begin{theorem}\label{prop:EqivColsStrict}
    Let $A$ be an $m\times m$ coloring matrix. Suppose $I\subseteq[m]$ and $J\subseteq[m]$ are two disjoint sets with $|I|=|J|$ such that, for all $n$, $\sum_{i\in I}t_A^{(i)}(n)=\sum_{j\in J}t_A^{(j)}(n)$ and that $\ell$ is a color (possibly in $I$ or $J$), such that $a_{\ell i} = 0$ for all $i\in I$ and $a_{\ell j} = 1$ for all $j\in J$. 
    Let $B$ be the $m\times m$ coloring matrix defined by
    \[
    b_{pq}=\begin{cases}
        1,&\text{$p=\ell$ and $q\in I$}\\
        0,&\text{$p=\ell$ and $q\in J$}\\
        a_{pq},&\text{otherwise}.
    \end{cases}
    \]
    Then, $A$ and $B$ are strictly tree coloring equivalent.
    Furthermore, for all $n\geq1$,
    $\sum_{i\in I}t^{(i)}_A(n)=\sum_{j\in J}t^{(j)}_A(n)=\sum_{i\in I}t^{(i)}_{B}(n)=\sum_{j\in J}t^{(j)}_{B}(n)$.
\end{theorem}
\begin{proof}
Since $\sum_{i\in I}t_A^{(i)}(n)=\sum_{j\in J}t_A^{(j)}(n)$, there is a bijection $\phi$ mapping each tree with $n$ vertices and a root color in $J$ to a tree with $n$ vertices and a root color in $I$. We now use $\phi$ to construct a bijection $\psi$ between the set of $A$-colored trees and the set of $B$-colored trees where $\psi(T)$ has the same number of vertices and same root color as $T$. This suffices to prove that $t^{(k)}_A(n)=t^{(k)}_{B}(n)$ for all $n$ and for any color $1\leq k\leq m$, that is that $A$ and $B$ are strictly tree coloring equivalent.

In what follows, a \emph{subtree} of a tree is defined to mean a child of the root along with all of that child's descendants. Suppose $T$ is an $A$-colored tree with $n$ vertices and root color $k$, where $1\leq k\leq m$.  Define $\psi(T)$ as follows:
\begin{itemize}
    \item If $k\neq\ell$, then $\psi(T)=T$.
    \item If $k=\ell$, then $\psi(T)$ is the tree $T'$ where:
    \begin{itemize}
        \item The root of $T'$ has color $k$.
        \item For each subtree of $T$ whose root color is not in $J$, $T'$ has that same subtree in the same position.
        \item For each subtree $\tilde T$ of $T$ whose root color is in $J$, $T'$ has $\phi(\tilde T)$ in its place.
    \end{itemize}
\end{itemize}
Each operation in computing $\psi$ is either an identity map or an application of $\phi$. Since $\phi$ is a bijection, it follows that $\psi$ is also a bijection. Furthermore, $\psi$ is explicitly constructed to preserve root colors, and the fact that $\phi$ preserves the number of vertices implies that $\psi$ preserves the number of vertices as well. Finally, notice that the roots of the subtrees obtained via $\phi$ have colors in $I$, and the other subtrees have root colors that can follow color $\ell$ in both $A$ and $B$ since the remaining parts of row $\ell$ are identical in $A$ and $B$. Therefore, $\psi(T)$ is $B$-colored.

So, we know that $t^{(k)}_A(n)=t^{(k)}_{B}(n)$ for all $n$ and for any color $1\leq k\leq m$. This immediately implies that $A$ and $B$ are strictly tree coloring equivalent. It also gives that $\sum_{i\in I}t^{(i)}_A(n)=\sum_{i\in I}t^{(i)}_{B}(n)$ and that $\sum_{j\in J}t^{(j)}_A(n)=\sum_{j\in J}t^{(j)}_{B}(n)$ for all $n$. This, along with the assumption that $\sum_{i\in I}t^{(i)}_A(n)=\sum_{j\in J}t^{(j)}_{A}(n)$, completes the proof of the remaining parts of the theorem.
\end{proof}

We obtain an important corollary of Theorem~\ref{prop:EqivColsStrict} when we take $|I|=|J|=1$. Corollary~\ref{cor:EqivCols} below is used
extensively in the proofs for the sequences in Section~\ref{sec:small} and Appendix~\ref{app:3by3}. 
Given a coloring rule $A$, let us call two colors $i$ and $j$ \emph{interchangeable} if
$t^{(i)}_A(n)=t^{(j)}_A(n)$ for all $n\geq1$. Note that if the $i$-th and the $j$-th rows of $A$ are identical, then $i$ and $j$ are interchangeable, since then we can just change the root color from $i$ to $j$ (or vice versa) in any valid coloring. However, this is not a necessary condition for two colors to be interchangeable. Though, as a consequence of Theorem~\ref{prop:equalones}, rows for interchangeable colors must contain the same number of ones.

\begin{corollary}
\label{cor:EqivCols}
Suppose $i$ and $j$ are two interchangeable colors for a coloring matrix $A$, and that $\ell$ is a color (not necessarily different from $i$ and $j$), such that 
$a_{\ell i} = 0$ and $a_{\ell j} = 1$. If we switch the numbers $a_{\ell i}$ and $a_{\ell j}$ in $A$, we obtain another coloring matrix $B$ such that $A$ and $B$ are strictly tree coloring equivalent.
Furthermore, for any $n\geq1$, 
$t^{(k)}_A(n)=t^{(k)}_{B}(n)$ for any color $k$, and 
in particular, $t^{(i)}_A(n)=t^{(j)}_A(n)=t^{(i)}_{B}(n)=t^{(j)}_{B}(n)$.
\end{corollary}

In the following examples, we illustrate how to apply Theorem~\ref{prop:EqivColsStrict} and Corollary~\ref{cor:EqivCols}.
\begin{example}
    Consider the following three matrices:
        \[\begin{bmatrix}
        1 & 1 & 1 \\
        0 & 1 & 0 \\
        0 & 0 & 1
        \end{bmatrix}, 
        \begin{bmatrix}
        1 & 1 & 1 \\
        0 & 1 & 0 \\
        0 & 1 & 0
        \end{bmatrix},
        \begin{bmatrix}
        1 & 1 & 1 \\
        0 & 0 & 1 \\
        0 & 1 & 0
        \end{bmatrix}.\]
    Here, colors $i = 2$ and $j = 3$ are interchangeable (see the second matrix).
    Starting with the first matrix, we interchange the $(3,2)$ entry and the $(3,3)$ entry; that is, we replace $a_{32}$ with $a_{23}$ and vice versa.
    This gives us the second matrix in the list, and by Corollary~\ref{cor:EqivCols}, we know these two matrices have the same counting sequence.
    Now, we inspect the second matrix, and we similarly interchange the $(2,2)$ entry and the $(2,3)$ entry; this gives us the third matrix in the list, which has the same counting sequence as the first two.
    Hence, all three of these matrices have the same counting sequence.
    We direct interested readers to Entry~29 in Appendix~\ref{app:3by3} to explore other matrices which have the same counting sequence as these three, yet this cannot be obtained using this notion of interchangeable colors.
\end{example}

\begin{example}\label{ex:iffcounter}
    Let
    \[
A=\begin{bmatrix}
    1 & 1 & 1 & 0 & 0 & 0\\
    0 & 1 & 0 & 0 & 0 & 0\\
    0 & 0 & 1 & 0 & 0 & 0\\
    0 & 0 & 0 & 1 & 1 & 0\\
    0 & 0 & 0 & 1 & 1 & 0\\
    0 & 0 & 0 & 0 & 0 & 1
\end{bmatrix}
\quad\text{and}\quad
B=\begin{bmatrix}
    0 & 0 & 0 & 1 & 1 & 1\\
    0 & 1 & 0 & 0 & 0 & 0\\
    0 & 0 & 1 & 0 & 0 & 0\\
    0 & 0 & 0 & 1 & 1 & 0\\
    0 & 0 & 0 & 1 & 1 & 0\\
    0 & 0 & 0 & 0 & 0 & 1
\end{bmatrix}.
\]
Future results (Entry~29 in Appendix~\ref{app:3by3} and Theorem~\ref{prop:blockdiag}) give us values for $t^{(i)}_A(n)$ for $1\leq i\leq 6$. In particular, we learn 
that $t^{(1)}_A(n)+t^{(2)}_A(n)+t^{(3)}_A(n)=t^{(4)}_A(n)+t^{(5)}_A(n)+t^{(6)}_A(n)$. 
Noting that $B$ differs from $A$ only in the first row; $A$ and $B$ are strictly tree coloring equivalent by Theorem~\ref{prop:EqivColsStrict}, taking $I=\{4,5,6\}$, $J=\{1,2,3\}$, and $l=1$. 
Since the first rows of these matrices are the only rows with sum $3$, color~$1$ is not interchangeable with any other color. So, it cannot be the case that $A$ and $B$ are related by Corollary~\ref{cor:EqivCols}.
\end{example}

Example~\ref{ex:iffcounter} makes it clear that matrices can be strictly tree coloring equivalent without being obtainable via
repeated applications of Corollary~\ref{cor:EqivCols}. 
However, it turns out that $A$ and $B$ are strictly tree coloring equivalent if and only if they are related by repeated applications of Theorem~\ref{prop:EqivColsStrict}.
\begin{theorem}\label{thm:EqivColsStrictIff}
    Suppose that $A\neq B$ are strictly tree coloring equivalent coloring matrices. Then, $B$ can be obtained from $A$ via one or more applications of Theorem~\ref{prop:EqivColsStrict}.
\end{theorem}
\begin{proof}
    Suppose that $A\neq B$ are $m\times m$ coloring matrices that are strictly tree coloring equivalent. This implies that $\gfa{k}=\gff{k}{B}$ for all $1\leq k\leq m$. So, by Theorem~\ref{prop:gfgeneral}, we have that
    \begin{equation}\label{eq:A}
    \gfa{k}=x+\gfa{k}\sum_{j=1}^ma_{kj}\gfa{j}
    \end{equation}
    and
    \begin{equation}\label{eq:B}
    \gff{k}{B}=x+\gff{k}{B}\sum_{j=1}^mb_{kj}\gff{j}{B}.
    \end{equation}
    Since we know that $\gfa{k}=\gff{k}{B}$, \eqref{eq:B} can be rewritten as
    \begin{equation}\label{eq:C}
    \gfa{k}=x+\gfa{k}\sum_{j=1}^mb_{kj}\gfa{j}.
    \end{equation}
    Subtracting \eqref{eq:C} from \eqref{eq:A} and dividing through by $\gfa{k}$, we obtain that
    \[
    \sum_{j=1}^ma_{kj}\gfa{j}=\sum_{j=1}^mb_{kj}\gfa{j},
    \]
    or
    \[
    \sum_{j=1}^m(a_{kj}-b_{kj})\gfa{j}=0.
    \]
    Since $A\neq B$, there is some $k$ for which not all of the coefficients $a_{kj}-b_{kj}$ are zero. As $t^{(j)}_A(1)=1$ for all $j$, there must be an equal number of $1$s and $-1$s among such coefficients. This implies there are sets $I$ and $J$ with $|I|=|J|$ 
    such that the entries in row $k$ of $A$ match the entries in row $k$ of $B$, except that the entries in the columns of $I$ are $0$ in $A$ and $1$ in $B$ and the entries in the columns of $J$ are $1$ in $A$ and $0$ in $B$. We now claim that $\sum_{i\in I}\gfa{i}=\sum_{j\in J}\gfa{j}$, which would essentially complete the proof.
    
    For $j\in J$, we have $a_{kj}-b_{kj}=1$, and for $i\in I$ we have $a_{kj}-b_{kj}=-1$. For all other entries in row $k$, $A$ and $B$ match. So, 
    \[
    0=\sum_{j=1}^m(a_{kj}-b_{kj})\gfa{j}=\sum_{j\in J}\gfa{j}-\sum_{i\in I}\gfa{i}.
    \]
    Therefore, $\sum_{i\in I}\gfa{i}=\sum_{j\in J}\gfa{j}$, as required.

    The preceding argument shows that $B$ is obtainable from $A$ via a single application of Theorem~\ref{prop:EqivColsStrict} if $A$ and $B$ differ in only a single row. If they differ in multiple rows, one application is needed per row where they differ.
\end{proof}

We conclude this subsection with a couple of additional results. First, the following simple observation on matrices with identical row sums yields many instances of tree coloring equivalent matrices.

\begin{theorem}\label{prop:row-sums}
    If $A$ is an $m \times m$ coloring matrix whose row sums are all equal to $r$, then we have $t_A^{(i)}(n) = r^{n-1} C_{n-1}$ for all $i$, thus $t_A(n) = mr^{n-1}C_{n-1}$. 
    In particular, $A$ has all of its colors interchangeable, and all such matrices are strictly tree coloring equivalent.
\end{theorem}

\begin{proof}
If all row sums of $A$ are equal to $r$, then every color can be followed by precisely $r$ different colors. Thus, once the root color has been selected, there are $r$ possibilities for the color of every subsequent vertex, regardless of the structure of the tree. The statement readily follows.\end{proof}

Next, we have a result about our generalized counting sequences, where we allow a fixed number of coloring violations.
\begin{theorem}\label{thm:strongcomplement}
    Suppose $m\times m$ coloring matrices $A$ and $B$ are tree coloring equivalent. Then, for all $0\leq q<n$ we have $t_A(n,q)=t_B(n,q)$.
\end{theorem}
\begin{proof}
    By Theorem~\ref{prop:gfgeneral}, the generating functions $\gfa{i}$ satisfy
    \begin{equation}\label{eq:gf-eq}
    \gfa{i}= \frac{x}{1-\sum_{j=1}^m a_{ij} \gfa{j}}   
    \end{equation}
    for $1 \leq i \leq m$.
    Two matrices $A$ and $B$ are tree coloring equivalent if and only if
    \[
    \gfall = \sum_{i=1}^m \gfa{i} =  \sum_{i=1}^m \gff{i}{B} = \gfll{B}.\]
    Thus it suffices to show that $\gfall$ determines $t_A(n,q)$ for all $n$.
    
    We shall show that the bivariate generating $F_A(x,u)$ is already determined by $\gfall = F_A(x,0)$. Note first that, by Theorem~\ref{prop:gfgeneralgeneral}, the functional equation~\eqref{eq:gf-eq} extends to
    \begin{align*}
    F_A^{(i)}(x,u) &= \frac{x}{1-\sum_{j=1}^m (a_{ij} + (1-a_{ij})u) F_A^{(j)}(x,u)} \\
    &= \frac{x}{1- u F_A(x,u) - (1-u) \sum_{j=1}^m a_{ij}F_A^{(j)}(x,u)}.
    \end{align*}
    This can be rewritten as
    \[
    \frac{(1-u)F_A^{(i)}(x,u)}{1-u F_A(x,u)} = \frac{\frac{x(1-u)}{(1- u F_A(x,u))^2}}{1 - \sum_{j=1}^m a_{ij} \frac{(1-u)F_A^{(j)}(x,u)}{1-u F_A(x,u)}}.
    \]
    In other words, $\frac{(1-u)F_A^{(1)}(x,u)}{1-uF_A(x,u)},\frac{(1-u)F_A^{(2)}(x,u)}{1-uF_A(x,u)}, \ldots, \frac{(1-u)F_A^{(m)}(x,u)}{1-uF_A(x,u)}$ satisfy the same system of equations as $\gfa{1},\gfa{2},\ldots,\gfa{m}$, with $x$ replaced by $\frac{x(1-u)}{(1- u F_A(x))^2}$. Since this system of equations characterizes them uniquely as formal power series, we have
    \[
    \frac{(1-u)F_A^{(i)}(x,u)}{1-u F_A(x,u)} = F_A^{(i)}\!\left(\frac{x(1-u)}{(1- u F_A(x))^2}\right)
    \]
    for $1 \leq i \leq m$. Summing this equation over all $i$, we find
    \[
    \frac{(1-u)F_A(x,u)}{1-uF_A(x,u)} = F_A \!\left( \frac{x(1-u)}{(1-uF_A(x,u))^2} \right).
    \]
    It is well known that a formal power series has a unique compositional inverse if its constant coefficient is $0$ and its linear coefficient is not $0$. This is the case for $\gfall$, so it has a compositional inverse $F_A^{\langle-1 \rangle}(x)$. Setting $H_A(x,u) = \frac{(1-u)F_A(x,u)}{1-uF_A(x,u)}$ (equivalently, $F_A(x,u) = \frac{H_A(x,u)}{1-u+uH_A(x,u)}$), we have
    \[
    \frac{x}{1-u}(1-u+uH_A(x,u))^2 = \frac{x(1-u)}{(1-uF_A(x,u))^2} = F_A^{\langle-1 \rangle}(H_A(x,u)),
    \]
    or
    \[
    x = (1-u)(1-u+uH_A(x,u))^{-2} F_A^{\langle-1 \rangle}(H_A(x,u)).
    \]
    So $H_A(x,u)$ is the unique compositional inverse of the power series $(1-u)(1-u+ux)^{-2} F_A^{\langle-1 \rangle}(x)$. Thus $\gfall$ determines $H_A(x,u)$, which in turn determines $F_A(x,u)$. This completes the proof.
\end{proof}

Theorem~\ref{thm:strongcomplement} has an important corollary.
\begin{corollary}\label{cor:complement}
    Given $m\times m$ coloring matrices $A$ and $B$, let $\bar{A}$ and $\bar{B}$ be obtained from $A$ and $B$ by replacing all $0$'s with $1$'s and $1$'s with $0$'s. Then, $A$ and $B$ are tree coloring equivalent if and only if $\bar{A}$ and $\bar{B}$ are tree coloring equivalent.
\end{corollary}
\begin{proof}
    First, suppose that $m\times m$ coloring matrices $A$ and $B$ are tree coloring equivalent. By Theorem~\ref{thm:strongcomplement}, we have $t_A(n,n-1)=t_B(n,n-1)$ for all $n\geq1$. But, $t_A(n,n-1)$ (respectively $t_B(n,n-1)$) counts the number of $m$-colored $n$-vertex trees with exactly $n-1$ edges violating the coloring rule $A$ (respectively $B$). A tree on $n$ vertices has $n-1$ edges, so such a tree violates the coloring rule for every edge. The coloring matrices $\bar{A}$ and $\bar{B}$ specify rules where every edge violates the coloring rules $A$ and $B$ respectively. So, we have $t_A(n,n-1)=t_{\bar{A}}(n)$ and $t_B(n,n-1)=t_{\bar{B}}(n)$. Therefore, $t_{\bar{A}}(n)=t_{\bar{B}}(n)$.

    The other direction of the equivalence follows from the fact that $\bar{\bar{A}}=A$ and $\bar{\bar{B}}=B$, so we can run the same argument with $A$ and $B$ replaced by $\bar{A}$ and $\bar{B}$.
\end{proof}

\section{Coloring rules based on block matrices}
\label{sec:new_from_old}
In this section, we describe several block matrix constructions that allow us to obtain sequence counts and/or generating functions for rules with more colors given one or more rules with fewer colors. We begin with the simplest such result, where the block matrix is diagonal.

\begin{theorem}\label{prop:blockdiag}
Given coloring matrices $A_1,A_2,\ldots,A_k$, let
\[
A=\left[\begin{array}{c|c|c|c}A_1&0&0&0\\\hline0&A_2&0&0\\\hline0&0&\ddots&0\\\hline0&0&0&A_k\end{array}\right].
\]
We have
\[
t_A(n)=\sum_{i=1}^kt_{A_i}(n)
\]
for all $n$.
\end{theorem}
\begin{proof}
    Let $A_1,A_2,\ldots,A_k$ and $A$ be as in the statement of the theorem. Consider a tree on $n$ vertices colored according to $A$. Suppose the root has color $j$, and suppose that the nonzero entries in row $j$ of $A$ come from matrix $A_i$. Then, the tree is colored according to $A_i$. This proves that $t_A^{(j)}(n)=t_{A_i}^{(j)}(n)$. Summing over all root colors gives the desired result.
\end{proof}

Next, we examine what happens if we ``blow up'' each $0$ in an existing coloring matrix to a square block of zeroes and each $1$ to a square block of ones.

\begin{theorem}\label{prop:blowup}
    Given a coloring matrix $A$, define the block matrix $\widehat{A}$ where we replace every $1$ in $A$ with the $k$-dimensional all ones matrix, and we replace every $0$ in $A$ with the $k$-dimensional all zeroes matrix.
    We call $\widehat{A}$ the \emph{blowup} of $A$.
    If $A$ and $B$ are (strictly) tree coloring equivalent, then so are $\widehat{A}$ and $\widehat{B}$. Moreover, $t_{\widehat{A}}(n) = k^n t_A(n)$ holds for all $n \geq 1$.
\end{theorem}
\begin{proof}
    Assume that $A$ and $B$ are (strictly) tree coloring equivalent and that each uses $m$ colors.
    When we blow up $A$ and $B$, each color is replaced with a $k \times k$ matrix, so we have $mk$ colors total. 
    Every color in $A$ is replaced with $k$ new colors in $\widehat{A}$, and similarly for $B$ and $\widehat{B}$.
    Since the new colors 
    behave like the old colors, 
    and the only difference is the quantity of colors, we conclude that $\widehat{A}$ and $\widehat{B}$ are (strictly) tree coloring equivalent.
    The formula for $t_{\widehat{A}}(n)$ follows from the observation that every color in a feasible coloring of a tree with $n$ vertices according to $A$ can be replaced by one of precisely $k$ colors to yield a coloring that is feasible according to $\widehat{A}$.
\end{proof}

For example, since $t_{[1]}(n) = C_{n-1}$, it follows that if $A = \begin{bmatrix}
1 & 1 \\
1 & 1 
\end{bmatrix}$, then $t_{A}(n) = 2^{n}C_{n-1}$ (see Entry~8 in Table~1). 
The rest of the constructions in this section are based on the observation that if we know that some colors are only usable for the root, we can often use that information to obtain information about the generating functions for trees with those root colors. Our first result in that vein is a simple observation about rules where some colors can never appear anywhere below vertices of other colors.
\begin{theorem}\label{prop:blockul}
    Let $A$ be an $m\times m$ coloring matrix, $B$ an $m'\times m$ matrix of zeroes and ones, and $C$ an $m'\times m'$ coloring matrix. Define $A'$ to be the block matrix
    \[
    A'=\left[\begin{array}{c|c}A&0\\\hline B&C\end{array}\right].
    \]
    For any $1\leq i\leq m$, we have
    \[
    \gff{i}{A'}=\gff{i}{A}.
    \]
\end{theorem}
\begin{proof}
    Fix a color $1\leq i\leq m$. 
    Notice that no color greater than $m$ can appear in any $A'$ colored tree anywhere below a vertex of color $i$.
    This implies trees colored according to $A'$ with root color $i$ are precisely those colored by $A$ with that root color.
    Thus, $\gff{i}{A'}=\gff{i}{A}$, as required.
\end{proof}
Theorem~\ref{prop:blockul} applies to numerous entries in Table~\ref{tb:2cols} and Appendix~\ref{app:3by3}.
We now present several results that specialize Theorem~\ref{prop:blockul} to situations where we can say something about the counting sequences with root colors greater than $m$.
\begin{theorem}\label{prop:rootonlysingle}
    Let $A$ be an $m\times m$ coloring matrix and $B$ an $m'\times m$ matrix of zeroes and ones. Define $A'$ to be the block matrix
    \[
    A'=\left[\begin{array}{c|c}A&0\\\hline B&0\end{array}\right].
    \]
    For any $1\leq i\leq m$, we have
    \[
    \gff{i}{A'}=\gff{i}{A}.
    \]
    Moreover, for any $m+1\leq i\leq m+m'$,
    we have
    \[
    \gff{i}{A'}=\frac{x}{1-\sum_{j=1}^mb_{i-m,j}\gff{j}{A}}.
    \]
\end{theorem}
\begin{proof}
    First, note that 
    for $1\leq i\leq m$, $\gff{i}{A'}=\gff{i}{A}$ by Theorem~\ref{prop:blockul}.
    Now, let $m+1\leq i\leq m+m'$. Theorem~\ref{prop:gfgeneral} says that
    \[
    \gff{i}{A'}=x+\gff{i}{A'}\sum_{j=1}^{m+m'}a'_{ij}\gff{j}{A'}.
    \]
    We know that
    \[
    a'_{ij}=\begin{cases}
        b_{i-m,j}&j\leq m,\\
        0&j>m,
    \end{cases}
    \]
    and we know that $\gff{j}{A'}=\gff{j}{A}$ for $j\leq m$. 
    We substitute to find
    \[
    \gff{i}{A'}=x+\gff{i}{A'}\sum_{j=1}^{m}b_{i-m,j}\gff{j}{A}.
    \]
    Solving this expression for $\gff{i}{A'}$ yields the desired result.
\end{proof}
Theorem~\ref{prop:rootonlysingle} with $B=\begin{bmatrix}1&0\end{bmatrix}$ applies to Entries~2, 4, 8, 9, 18, 21, 40, and~43 in Appendix~\ref{app:3by3}. These can be obtained in this order, up to isomorphism, by letting $A$ be each of the eight entries of Table~\ref{tb:2cols} in order. Similarly, Theorem~\ref{prop:rootonlysingle} with $B=\begin{bmatrix}0&1\end{bmatrix}$ applies to Entries~2, 3, 7, 9, 14, 13, 42, and~43 in Appendix~\ref{app:3by3}. 
Theorem \ref{prop:rootonlysingle}
also applies with other matrices $B$, but at least one of the upcoming theorems also applies in each such case.

Next, here is a rule for when there are $m'$ colors, each of which can only be used for the root, 
and any other color can follow these $m'$ colors.
\begin{theorem}\label{prop:rootonly}
    Let $A$ be an $m\times m$ coloring matrix, and define $A'$ to be the block matrix
    \[
    A'=\left[\begin{array}{c|c}A&0\\\hline1&0\end{array}\right],
    \]
    where $1$ represents a block matrix of $m'$ rows of ones and $0$ represents a block of zeroes. 
    We have
    \[
    \gfll{A'}=\gfall+\frac{m'x}{1-\gfall}.
    \]
\end{theorem}
\begin{proof}
    As a consequence of Theorem~\ref{prop:blockul}, the
    number of $A'$-colored $n$-vertex trees whose root is not one of the last $m'$ colors is given by $t_A(n)$, and $t_{A'}(n)-t_A(n)$ gives the number of trees whose root \emph{is} one of the last $m'$ colors. 
    Any tree whose root is colored by one of the last $m'$ colors has its rightmost subtree not having a root color among the last $m'$ colors. 
    By the standard Catalan decomposition, we have
    \[
    t_{A'}(n)=t_A(n)+\sum_{k=1}^{n-1}(t_{A'}(k)-t_A(k))t_A(n-k),
    \]
    since either the tree does not have a root color from among the last $m'$ colors or it does. 
    Converting to ordinary generating functions,  this gives 
    \[
    \gfll{A'}=\gfall+\gfll{A'}\gfall-\gfall^2+m'x,
    \]
    with the addition of $m'x$ present because there are $m'$ more colorings of the $1$-vertex tree for $A'$ than there are for $A$. 
    Solving for $\gfll{A'}$ gives
    \[
    \gfll{A'}=\frac{\gfall-\gfall^2+m'x}{1-\gfall}=\gfall+\frac{m'x}{1-\gfall},
    \]
    as required.
\end{proof}
Theorem~\ref{prop:rootonly} applies to Entries~2 and~4 in Table~\ref{tb:2cols}, which can be obtained, up to isomorphism, from the two $1\times1$ coloring matrices.
It also applies to Entries~3 and~9 in Appendix~\ref{app:3by3}, which can be obtained, up to isomorphism, from the two $1\times 1$ coloring matrices. 
In addition, it applies to Entries~5, 15 (first three entries), 16, 19, 25, 26, 30 and 34 in Appendix~\ref{app:3by3}, which can be obtained, up to isomorphism, from the eight entries of Table~\ref{tb:2cols} in order.

Next, we define an up-set in a tree. An \emph{up-set} is a set of vertices closed under taking ancestors. With this in mind,
here is a rule for when there are $m'$ colors, each of which can only be used for some up-set, 
and these colors can freely intermingle in said up-set. Also, any other color can follow these $m'$ colors.
\begin{theorem}\label{prop:upsetonly}
    Let $A$ be an $m\times m$ coloring matrix, and define $A'$ to be the block matrix
    \[
    A'=\left[\begin{array}{c|c}A&0\\\hline1&1\end{array}\right],
    \]
    where $1$ represents a block matrix of $m'$ rows of ones and $0$ represents a block of zeroes.
    We have the relation
    $\gfll{A'}=\gfall+\gfll{A'}^2-\gfll{A'}\gfall+m'x$,
    and accordingly we have that
    \[
    \gfll{A'}=\frac{1+\gfall-\sqrt{(1-\gfall)^2-4m'x}}{2}.
    \]
\end{theorem}
\begin{proof}
    As a consequence of Theorem~\ref{prop:blockul}, the
    number of $A'$-colored $n$-vertex trees whose root is not one of the last $m'$ colors is given by $t_A(n)$, and $t_{A'}(n)-t_A(n)$ gives the number of trees whose root \emph{is} one of the last $m'$ colors. 
    Any tree whose root is colored by one of the last $m'$ colors
    has its rightmost subtree as possibly any smaller tree colored by $A'$. By the standard Catalan decomposition, we have
    \[
    t_{A'}(n)=t_A(n)+\sum_{k=1}^{n-1}(t_{A'}(k)-t_A(k))t_{A'}(n-k),
    \]
    as either the tree does not have a root color from among the last $m'$ colors or it does. Converting to ordinary generating functions, this gives the desired expression $\gfll{A'}=\gfall+\gfll{A'}^2-\gfll{A'}\gfall+m'x$, with the addition of $m'x$ present because there are $m'$ more colorings of the $1$-vertex tree for $A'$ than there are for $A$. Solving for $\gfll{A'}$ gives
    \begin{align*}
    \gfll{A'} &=\frac{1+\gfall-\sqrt{(1+\gfall)^2-4(m'x+\gfall})}{2} \\ &=\frac{1+\gfall-\sqrt{(1-\gfall)^2-4m'x}}{2},
    \end{align*}
    as required.
\end{proof}
Theorem~\ref{prop:upsetonly} applies to Entries~5 and~6 in Table~\ref{tb:2cols}, which can be obtained, up to isomorphism, from the two $1\times1$ coloring matrices. 
It also applies to Entries~64 and~65 in Appendix~\ref{app:3by3}, which can be obtained, up to isomorphism, from the two $1\times 1$ coloring matrices. In addition, it applies to Entries~27, 28, 29 (first three matrices), 31, 32, 33, 35, and~36 in Appendix~\ref{app:3by3}, which can be obtained, up to isomorphism, from the eight entries of Table~\ref{tb:2cols} in order.

Next, we provide a rule for when there are $m'$ colors, each of which can only be used for some up-set, 
and these up-set must be monochromatic. Also, any other color can follow these $m'$ colors.
\begin{theorem}\label{prop:separateupsetsonly}
    Let $A$ be an $m\times m$ coloring matrix, and define $A'$ to be the block matrix
    \[
    A'=\left[\begin{array}{c|c}A&0\\\hline1&I\end{array}\right],
    \]
    where $1$ represents a block matrix of $m'$ rows of ones, $0$ represents a block of zeroes, and $I$ represents the $m'\times m'$ identity matrix.
    We have the relation
    \[
    \frac{1}{m'}\gfll{A'}^2+\left(\frac{m'-2}{m'}\gfall-1\right)\gfll{A'}+\left(\gfall-\frac{m'-1}{m'}\gfall^2+m'x\right)=0,
    \]
    and accordingly we have that
    \begin{align*}
    \gfll{A'} &=\frac{m'-(m'-2)\gfall}{2} \\
    &\qquad - \frac{\sqrt{(m'-(m'-2)\gfall)^2-4(m'\gfall-(m'-1)\gfall^2+(m')^2x)}}{2}.
    \end{align*}
\end{theorem}
\begin{proof}
    As a consequence of Theorem~\ref{prop:blockul}, the
    number of $A'$-colored $n$-vertex trees whose root is not one of the last $m'$ colors is given by $t_A(n)$, and $t_{A'}(n)-t_A(n)$ gives the number of trees whose root \emph{is} one of the last $m'$ colors. Any such tree has its rightmost subtree as possibly any smaller tree colored by $A$ or colored by $A'$ with the same colored root as the whole tree. All $m'$ additional colors are interchangeable, so by the standard Catalan decomposition, we have
    \begin{align*}
    t_{A'}(n)&=t_A(n)+\sum_{k=1}^{n-1}(t_{A'}(k)-t_A(k))\left(t_A(n-k)+\frac{1}{m'}(t_{A'}(n-k)-t_A(n-k))\right)\\
    &=t_A(n)+\sum_{k=1}^{n-1}(t_{A'}(k)-t_A(k))\left(\frac{m'-1}{m'}\cdot t_A(n-k)+\frac{1}{m'}\cdot t_{A'}(n-k)\right)\\
    &=t_A(n)+\frac{1}{m'}\sum_{k=1}^{n-1}(t_{A'}(k)-t_A(k))((m'-1)t_A(n-k)+t_{A'}(n-k)),
    \end{align*}
    as either the tree does not have a root color from among the last $m'$ colors or it does. Converting to ordinary generating functions, this gives
    \begin{align}
    \nonumber\gfll{A'}&=\gfall+\frac{1}{m'}((m'-1)\gfll{A'}\gfall+\gfll{A'}^2 \\
    &\nonumber\qquad -(m'-1)\gfall^2-\gfll{A'}\gfall+m'x\\
    &=\gfall+\frac{1}{m'}((m'-2)\gfll{A'}\gfall+\gfll{A'}^2-(m'-1)\gfall^2)+m'x\label{eq:idgfthm},
    \end{align}
    with the addition of $m'x$ present because there are $m'$ more colorings of the $1$-vertex tree for $A'$ than there are for $A$. 
    Equation~\eqref{eq:idgfthm} is equivalent to the desired relation
    \[
    \frac{1}{m'}\gfll{A'}^2+\left(\frac{m'-2}{m'}\gfall-1\right)\gfll{A'}+\left(\gfall-\frac{m'-1}{m'}\gfall^2+m'x\right)=0.
    \]
    Solving for $\gfll{A'}$ gives
    \begin{align*}
        \gfll{A'}&=\frac{1-\frac{m'-2}{m'}\gfall-\sqrt{\left(1-\frac{m'-2}{m'}\gfall\right)^2-\frac{4}{m'}\left(\gfall-\frac{m'-1}{m'}\gfall^2+(m')^2x\right)}}{2}\\
        &= \frac{m'-(m'-2)\gfall}{2} \\
    &\qquad - \frac{\sqrt{(m'-(m'-2)\gfall)^2-4(m'\gfall-(m'-1)\gfall^2+(m')^2x)}}{2},
    \end{align*}
    as required.
\end{proof}
Theorem~\ref{prop:separateupsetsonly} applies to Entries~15 (first three matrices) and~16 in Appendix~\ref{app:3by3}, which can be obtained, up to isomorphism, from the two $1\times 1$ matrices. Otherwise, it says exactly the same things about the other entries in Appendix~\ref{app:3by3} that Theorem~\ref{prop:upsetonly} applies to.

Finally, we include several easy observations.
\begin{theorem}\label{prop:zerorows}
    Let $A$ be an $m\times m$ coloring matrix and $B$ an $m\times m'$ matrix of zeroes and ones. Define $A'$ to be the block matrix
    \[
    A'=\left[\begin{array}{c|c}A&B\\\hline 0&0\end{array}\right]
    \]
    where $0$ represents a block of zeroes of the appropriate dimensions.
    For any $m+1\leq i\leq m+m'$, we have
    \[
    \gff{i}{A'}=x.
    \]
\end{theorem}
\begin{proof}
    Fix a color $i$ with
    $m+1\leq i\leq m+m'$. 
    A vertex with color $i$ cannot have any successors. So, the only tree with root color $i$ colored according to $A'$ is the tree with a single vertex colored $i$. This implies that $\gff{i}{A'}=x$, as required.
\end{proof}

\begin{theorem}\label{prop:catrows}
    Let $A$ be an $m\times m$ coloring matrix and $B$ an $m\times m'$ matrix of zeroes and ones. Define $A'$ to be the block matrix
    \[
    A'=\left[\begin{array}{c|c}A&B\\\hline 0&I\end{array}\right]
    \]
    where $0$ represents a block of zeroes and $I$ represents the $m'\times m'$ identity matrix.
    For any $m+1\leq i\leq m+m'$, we have
    \[
    \gff{i}{A'}=\frac{1-\sqrt{1-4x}}{2}.
    \]
\end{theorem}
\begin{proof}
    Fix a color $i$ with
    $m+1\leq i\leq m+m'$.
    A vertex with color $i$ can only have descendants of color~$i$. So, each tree structure has exactly one coloring with root color $i$. This means that $t_{A'}^{(i)}(n)=C_{n-1}$ for $n\geq1$.
    This implies that
    \[
    \gff{i}{A'}=\frac{1-\sqrt{1-4x}}{2},
    \]
    as required.
\end{proof}
Theorems~\ref{prop:zerorows} and~\ref{prop:catrows} apply to several entries in Table~\ref{tb:2cols} and several  entries in Appendix~\ref{app:3by3}.

\begin{theorem}\label{prop:unusablecolors}
    Let $A$ be an $m\times m$ coloring matrix, and define $A'$ to be the block matrix
    \[
    A'=\left[\begin{array}{c|c}A&0\\\hline0&0\end{array}\right],
    \]
    where $m'$ rows and columns of zeroes are added on to $A$ to make $A'$.
    We have
    \[
    \gfll{A'}=\gfall+m'x.
    \]
\end{theorem}
\begin{proof}
    This follows by applying Theorem~\ref{prop:blockul} to the first $m$ colors and Theorem~\ref{prop:zerorows} to the last $m'$ colors.
\end{proof}
Theorem~\ref{prop:unusablecolors} applies to Entries~1 and~3 in Table~\ref{tb:2cols}, which can be obtained, up to isomorphism, from the two $1\times1$ coloring matrices. It also applies to Entries~1 and~6 in Appendix~\ref{app:3by3}, which can be obtained, up to isomorphism, from the two $1\times 1$ coloring matrices. In addition, it applies, up to isomorphism, to Entries~1, 2, 6, 8, 10, 12, 37, and~39 in Appendix~\ref{app:3by3}, which can be obtained from the eight entries of Table~\ref{tb:2cols} in order.

\begin{theorem}\label{prop:singletoncolors}
    Fix positive integers $m_1$, $m_2$, and $m_3$.
    Let $A_1$ be an $m_1\times m_1$ coloring matrix, let $A_2$ be an $m_1\times m_2$ matrix of zeroes and ones, and let $B$ be an $m_3\times m_2$ matrix of zeroes and ones. Let
    \[
    A=\left[\begin{array}{c|c}A_1&A_2\\\hline0&0\end{array}\right]
    \]
    and
    \[
    A'=\left[\begin{array}{c|c|c}A_1&A_2&0\\\hline0&0&0\\\hline0&B&0\end{array}\right],
    \]
    where each $0$ stands for a block of zeroes so that $A$ is $(m_1+m_2)\times(m_1+m_2)$ and $A'$ is $(m_1+m_2+m_3)\times (m_1+m_2+m_3)$.

    Then, we have
    \[
    \gfll{A'}=\gfall+x\sum_{i=1}^{m_3}\frac{1}{1-x\sum_{j=1}^{m_2}b_{ij}}.
    \]
    In particular, for $1\leq i\leq m_1+m_2$ we have $t^{(i)}_{A'}(n)=t^{(i)}_{A}(n)$ for all $n\geq1$,
    and for $m_1+m_2<i\leq m_1+m_2+m_3$ we have
    \[
    t^{(i)}_{A'}(n)=\left(\sum_{j=1}^{m_2}b_{i-m_1-m_2,j}\right)^{n-1}
    \]
    for all $n\geq1$. 
\end{theorem}
\begin{proof}
    We prove the latter claim, which implies the former by converting to ordinary generating functions. First, suppose we assign $1\leq i\leq m_1+m_2$ as the root's color.  
    We obtain from Theorem~\ref{prop:blockul} that $t^{(i)}_{A'}(n)=t^{(i)}_{A}(n)$ for all $n\geq1$, as required. 
    Now, suppose we assign the root a color $m_1+m_2<i\leq m_1+m_2+m_3$ in an $A'$-coloring.
    In this case, the root must be followed by a color $m_1+1\leq j\leq m_2$. Each of these colors has no allowed successors. 
    That is, the only trees that can be colored have height at most $1$; such a tree with $n$ vertices has $n-1$ vertices at depth $1$. 
    Each can be colored independently with any of the colors that can follow color $i$, of which there are $\sum_{j=1}^{m_2}b_{i-m_1-m_2,j}$ options. 
    The result immediately follows.
\end{proof}
Theorem~\ref{prop:singletoncolors} applies to Entries~2, 3, 7, and~14 in Appendix~\ref{app:3by3}. In each case, $A_1$, $A_2$, and $B$ are $1\times 1$ coloring matrices with $B$ containing a $1$. The four entries correspond to the four possible choices for $(A_1,A_2)$, with the resulting matrices being isomorphic to those listed in the entries.

\section{Small matrices representing coloring rules}\label{sec:small}

Sequence \seqnum{A000595} enumerates the number of equivalence classes of isomorphic coloring matrices. In this section, we exhaustively study all $2\times 2$ and $3\times 3$ coloring matrices, which can be generated by brute force. A faster generation of matrices of larger order can be achieved by looking at each coloring matrix as an adjacency matrix of a graph, and by using one of the existing libraries for generation of these graphs first, e.g., \texttt{networkx} in Python based on the article of Cordella et al. \cite{cordella2001improved}. Table~\ref{tb:2cols} includes a representative of each of the $10$ equivalence classes of $2\times 2$ coloring matrices, and Appendix~\ref{app:3by3} includes a representative of each of the $104$ equivalence classes of $3\times 3$ coloring matrices. Proofs of the corresponding counting sequences are also given (in the nontrivial cases). We use the following notation: 
\begin{itemize}
    \item[] $C_{n}$ denotes the $n$-th Catalan number $\frac{1}{n+1}\binom{2n}{n}$ (\seqnum{A000108}).
    \item[] $N_{n,k}$ denotes the corresponding Narayana number, $\frac{1}{n}\binom{n}{k-1}\binom{n}{k}$ (\seqnum{A001263}).
    \item[] $F_{n}$ denotes the $n$-th Fibonacci number, with $F_{1}=F_{2}=1$ (\seqnum{A000045}).
    \item[] $D_{n}$ denotes the number of (nonempty) antichains in rooted plane trees with $n$ vertices (\seqnum{A007852}). 
    \item[] $R_{n}$ denotes the large Schr\"{o}der numbers (\seqnum{A006318}).
    \item[] $r_{n}$ denotes the little Schr\"{o}der numbers (\seqnum{A001003}).
    ($R_n=2\cdot r_n$ for $n\geq1$.)
\end{itemize}
Note that an antichain in a tree is a set of vertices where no vertex in the set is an ancestor of another.
\subsection{Coloring rules with two colors}\label{ss:2by2}
Table~\ref{tb:2cols} describes all counting sequences corresponding to all coloring rules with two colors with all non-isomorphic coloring matrices listed. We also give references or proofs for each nontrivial counting sequence. 
\begin{table}[htbp]
    \setlength{\tabcolsep}{0.25em}
    \begin{tabular}{ | c | c | c | c | } 
  \hline
   Entry & Matrix & Sequence ($n = 1, 2,3,\ldots$) & Formula ($n\geq2$)\\
  \hline
  1 & $\begin{bmatrix}
0 & 0 \\
0 & 0 
\end{bmatrix}$ & $2,0,0,0,\ldots$ & $0$ \\
\hline
  2 & $\begin{bmatrix}
0 & 1 \\
0 & 0 
\end{bmatrix}$ & $2,1,1,1,\ldots$ & $1$ \\
  \hline
  3 & $\begin{bmatrix}
1 & 0 \\
0 & 0 
\end{bmatrix}$ & $2, 1, 2, 5, 14, 42, \ldots$ & 
$C_{n-1}$\\
  \hline
  4 & $\begin{bmatrix}
1 & 0 \\
0 & 1 
\end{bmatrix}$, $\begin{bmatrix}
1 & 0 \\
1 & 0 
\end{bmatrix}$, $\begin{bmatrix}
0 & 1 \\
1 & 0 
\end{bmatrix}$ & $2, 2, 4, 10, 28, 84, \ldots$ & $2C_{n-1}$ \\ 
  \hline
  5 & $\begin{bmatrix}
1 & 1 \\
0 & 0 
\end{bmatrix}$ & $2, 2,6,22,90, 394, \ldots$ & 
$R_{n-1}$ \\
\hline
6 & $\begin{bmatrix}
1 & 1 \\
0 & 1 
\end{bmatrix}$ & $2, 3, 9,34,145, 667, \ldots$ & $C_{n-1} + D_{n}$  \\
\hline
7 & $\begin{bmatrix}
1 & 1 \\
1 & 0 
\end{bmatrix}$ & $2, 3, 10,42,198,1001, \ldots$ & $
\frac{2}{n}\binom{3n-3}{n-1}$\\
\hline
8 & $\begin{bmatrix}
1 & 1 \\
1 & 1 
\end{bmatrix}$ & $2, 4, 16, 80, 448, 2688, \ldots$ & $
2^n C_{n-1}$\\
\hline
\end{tabular}
\caption{Coloring rules with two colors}
\label{tb:2cols}
\end{table}
\begin{description}
    \item[Entry 4] This follows from Theorem~\ref{prop:row-sums}, since all row sums are equal to $1$, and it follows from Corollary~\ref{cor:EqivCols} since we have identical rows for the colors $1$ and $2$ in $\begin{bmatrix}
1 & 0 \\
1 & 0 
\end{bmatrix}$. Switching the $0$ and the $1$ in rows $1$ and $2$, respectively, gives us the matrices $\begin{bmatrix}
0 & 1 \\
1 & 0 
\end{bmatrix}$ and $\begin{bmatrix}
1 & 0 \\
0 & 1 
\end{bmatrix}$.\\
\vspace{1mm}
    \item[Entry 5] 
    We have to count the number of trees with $n$ vertices whose leaves can be of two different colors (vertices of color $2$ can be only leaves). We know that the number of trees in $\mathcal{T}_{n}$ with $k$ leaves is given by the Narayana numbers $N_{n-1,k} = \frac{1}{n-1}\binom{n-1}{k}\binom{n-1}{k-1}$. The identity 
    $$
    \sum\limits_{k=1}^{n-1}2^{k}N_{n-1,k} = R_{n-1}
    $$
    is well-known and was also mentioned in Section \ref{sec:intro}.\\
\vspace{1mm}
\item[Entry 6] The trees in $\mathcal{T}_{n}$ with root of color $2$ are enumerated by $C_{n-1}$. It remains to count all trees with root color~1. Every such choice corresponds to a choice of vertices to be the roots of subtrees colored entirely by color~2. In addition, for any pair $x,y$ of these vertices, neither $x$ is a descendant of $y$, nor $y$ of $x$. The one exception is when the antichain is just the root; in this case we color the entire tree with color~$1$. Thus, it suffices to count antichains in trees with $n$ vertices. The generating function for this sequence, $D_{n}$, is obtained in~\cite{Klazar1997Twelve,Klazar1997Addendum}.\\
\vspace{1mm}
\item[Entry 7] We need to count those trees for which we have two colors, but no two vertices of color 2 are adjacent. 
This follows from Theorem~\ref{thm:red-blue}.

In Appendix~\ref{app:ent7bij}, we exhibit a new bijective proof of this entry, using a bijection with some restricted Dyck paths. In addition, we prove that trees with root color~$1$ are enumerated by $\frac{1}{n}\binom{3n-2}{n-1}=\frac{3n-2}{n(2n-1)}\binom{3n-3}{n-1}$ (\seqnum{A006013}), while trees with root color~$2$ are enumerated by $\frac{1}{2n-1}\binom{3n-3}{n-1}$ (\seqnum{A001764}). These formulas were (in the setting of independent sets in plane trees) already obtained by Kirschenhofer, Prodinger and Tichy~\cite{KPT1986Fibonacci} by means of generating functions.
\vspace{1mm}
\item[Entry 8] This follows from Theorem~\ref{prop:row-sums}, since all row sums are equal to $2$. As we pointed out, the fact follows from Theorem~\ref{prop:blowup}, as well.
\end{description}

\subsection{Coloring rules with three colors}\label{ss:3by3}
An exhaustive search reveals that there are $74$ strong tree coloring equivalence classes of tree coloring rules with three colors. 
These $74$ equivalence classes result in $72$ different enumeration sequences, as there are two instances of tree coloring equivalent matrices that are not strongly tree coloring equivalent.
Information about all of the equivalence classes can be found in Appendix~\ref{app:3by3}. 
Here, we highlight a few notable findings from Appendix~\ref{app:3by3}.

\subsubsection{Equivalent matrices that are not strongly equivalent}
Based on our exhaustive study in Subsection~\ref{ss:2by2}, two $2\times2$ coloring matrices $A$ and $B$ are tree coloring equivalent if and only if they are strongly tree coloring equivalent. 
This is no longer true for $3\times 3$ coloring matrices. The matrices $A=\begin{bmatrix}
    0&0&0\\
    1&0&1\\
    1&0&1
\end{bmatrix}$ and $B=\begin{bmatrix}
    0&1&1\\
    1&0&0\\
    1&0&0
\end{bmatrix}$ have $t_A(n)=t_B(n)=2R_{n-1}$ for $n\geq2$. However, these matrices are not strongly tree coloring equivalent by Theorem~\ref{prop:equalones}, as the first has a row of all zeroes which the second lacks. Indeed, we also see this lack of strong equivalence because, for $n\geq2$, $t^{(1)}_A(n)=0$ and $t^{(2)}_A(n)=t^{(3)}_A(n)=R_{n-1}$, whereas 
$t^{(1)}_B(n)=R_{n-1}$ and $t^{(2)}_A(n)=t^{(3)}_A(n)=r_{n-1}$. These findings are described in detail under Entry~15 of Appendix~\ref{app:3by3}, including a description of a bijection between trees colored according to $A$ and according to a matrix isomorphic to $B$.
A similar phenomenon exists for the matrices $C=\begin{bmatrix}
    1&1&1\\
    0&1&0\\
    0&0&1
\end{bmatrix}$ and $D=\begin{bmatrix}
    1&1&0\\
    1&1&0\\
    0&0&1
\end{bmatrix}$, which satisfy $t_C(n)=t_D(n)=\left(2^n+1\right)C_{n-1}$ for $n\geq1$. Details regarding these matrices are available under Entry~29 of Appendix~\ref{app:3by3}.

Note that these entries are connected by Corollary~\ref{cor:complement}, so the fact that one of these entries has matrices that are tree coloring equivalent but not strongly tree coloring equivalent implies that the other entry must as well.

\subsubsection{An appearance of Fibonacci numbers}
In Entry~30 of Appendix~\ref{app:3by3}, the counting sequence is found to involve Fibonacci numbers. In particular, for the matrix $A=\begin{bmatrix}
    0&0&0\\
    1&0&1\\
    1&0&0
\end{bmatrix}$, we have $t^{(2)}_A(n)=F_{2n-1}$ for all $n\geq1$. This single entry involving Fibonacci numbers arises from the unique pair of circumstances where the rule colors only trees of bounded height, and no two colors are interchangeable. We generalize this entry in Section~\ref{sec:special} with Theorem~\ref{thm:genfib}.

\subsubsection{Sequences related to Entry~7 in Table~\ref{tb:2cols}}
Entries~34 and~40 in Appendix~\ref{app:3by3} are both closely related to Entry~7 in Table~\ref{tb:2cols}. In Entry~34, color~1 is effectively duplicated; Entry~40 duplicates color~2. Entry~7 in Table~\ref{tb:2cols} has a formula involving a binomial coefficient containing $3n$, as do the sequences (\seqnum{A006013} and \seqnum{A001764}) resulting from fixing the root color for that rule. As such, we find a counting formula of $t_A(n)=\frac{7n-4}{n\left(2n-1\right)}\binom{3n-3}{n-1}$ for Entry~34 and $t_A(n)=\frac{5n-2}{n\left(2n-1\right)}\binom{3n-3}{n-1}$ for Entry~40. At time of writing, neither of these sequences was in the OEIS.

Entry~42 in Appendix~\ref{app:3by3} is also closely related to Entry~7 in Table~\ref{tb:2cols}. This rule is just like Entry~7 in Table~\ref{tb:2cols}, except that the root is allowed to be colored with color~3 if all vertices at depth $1$ are colored with color~2. As such, the sequences $t^{(1)}_A(n)$ and $t^{(2)}_A(n)$ match those for Entry~7 in Table~\ref{tb:2cols}. The sequence $t^{(3)}_A(n)$ is also intriguing; it is enumerated by \seqnum{A098746}. This is proved in the discussion under Entry~42 in Appendix~\ref{app:3by3}. One such proof is obtained via a variation of the bijection for Entry~7 in Table~\ref{tb:2cols} exhibited in Appendix~\ref{app:ent7bij}.

Next, Entry~54 in Appendix~\ref{app:3by3} includes the matrix $A=\begin{bmatrix}
    1&1&1\\
    1&1&0\\
    1&0&0
\end{bmatrix}$, a natural extension of Entry~7 in Table~\ref{tb:2cols}. 
This is the next rule in an infinite family that is fully enumerated by Theorem~\ref{thm:k-plane}. 
As such, this entry, and only this entry, includes a formula with the term $4n$ in a binomial coefficient.

There are several additional entries in Appendix~\ref{app:3by3} closely tied to Entry~7 in Table~\ref{tb:2cols}, but we conclude by highlighting only two more of them. Entry~53 results from a different notion of duplicating color~2, and Entry~71 results from that same notion of duplicating color~1. 
The counting sequences for fixed root colors for both of these entries are existing OEIS sequences. 
However, unlike Entries~34 and~40, these are not hypergeometric.

\section{Special families of coloring rules}\label{sec:special}
In this section, we look at families of coloring rules akin to Theorem~\ref{thm:k-plane} that apply to an infinite family of coloring rules with arbitrarily many colors. The six families presented here are largely unrelated to one another, and they show off the diversity of general results available in this realm. Some of these families generalize entries in Appendix~\ref{app:3by3}; others are inspired by them.

\subsection{Family 1}

The discussion of trees with bicolored leaves in the introduction generalizes to coloring rules corresponding to matrices with the following structure: $\ell$ rows consisting entirely of ones (without loss of generality, the first $\ell$ rows), and $m$ rows consisting entirely of zeroes. The associated coloring rule is: every internal vertex can be colored in one of $\ell$ colors, and every leaf can be colored in one of $\ell+m$ colors. Thus, by the same argument that gave us~\eqref{eq:schroder}, we obtain a generalization of the large Schr\"oder numbers (the case $l=m=1$ gives Entry~5 in Table~\ref{tb:2cols}).

\begin{theorem}\label{thm:gennarayana}
    Let $S(\ell,m)$ be the $(\ell+m) \times (\ell+m)$-matrix whose first $\ell$ rows consist entirely of ones while the rest consists entirely of zeroes. Then we have
    $$t_{S(\ell,m)}(n) = \sum_{k=1}^{n-1} N_{n-1,k} \ell^{n-k} (\ell+m)^k = \sum_{k=1}^{n-1} \frac{1}{n-1} \binom{n-1}{k} \binom{n-1}{k-1} \ell^{n-k} (\ell+m)^k$$
    for all $n > 1$.
\end{theorem}
In addition to the large Schr\"oder numbers, several other sequences of interest can be obtained as special cases. When $
\ell=1$ and $m=2$, we get \seqnum{A047891}: the number of plane trees with $n$ vertices and tricolored leaves. The case where $\ell=2$ and $m=1$ yields \seqnum{A239488}.
Note that when $\ell = m$, the formula simplifies to $m^n$ times the large Schr\"oder numbers. Specifically, $\ell=m=2$ yields \seqnum{A054726}, which also counts graphs with $n+1$ vertices on a circle without crossing edges.

\subsection{Family 2}\label{sec:fam3}

Recall that the matrix
$$A = \begin{bmatrix}
    1 & 1 \\
    1 & 0 \\    
\end{bmatrix}$$
gives rise to the counting sequence $t_A(n) = \frac{2}{n} \binom{3n-3}{n-1}$. This counts trees that are colored with two colors (blue and white) and the rule that there cannot be two adjacent white vertices. In fact, \cite{KPT1986Fibonacci} contains a refinement: the number of plane trees colored according to this rule, with $k$ blue vertices and $n-k$ white vertices, is equal to
$$\frac{2}{n} \binom{n}{k} \binom{2n-3}{n-k-1},$$
of which
$$\frac{1}{n} \binom{n}{k} \binom{2n-2}{n-k-1}$$
have a blue root, and
$$\frac{1}{n-1} \binom{n-1}{k-1} \binom{2n-2}{n-k-1}$$
have a white root. Note that the formula for $t_A(n)$ is a simple corollary of the Vandermonde's identity. We can use this refinement to provide a formula for a slight generalization of the matrix $A$ above: consider
the matrix $A(\ell,m)$ whose block form is
\[
    A(\ell,m) =\left[\begin{array}{c|c}1&1\\\hline 1&0\end{array}\right],
\]
where the upper left block is an $\ell \times \ell$-matrix of ones, and the lower right block an $m \times m$-matrix of zeroes. The coloring rule determined by $A(\ell,m)$ can be described as follows: take a feasible blue-white coloring as described above, then recolor  every blue vertex in one of $\ell$ new colors and repaint every white vertex in one of $m$ new colors. As for the previous family, we immediately get the following theorem.

\begin{theorem}\label{thm:removesquare}
For all positive integers $\ell,m,n$, we have
    $$t_{A(\ell,m)}(n) = \sum_{k=0}^{n-1} \frac{2}{n} \binom{n}{k} \binom{2n-3}{n-k-1} \ell^{n-k} m^k.$$
Moreover, for $1 \leq i \leq \ell$,
    $$t_{A(\ell,m)}^{(i)}(n) = \sum_{k=0}^{n-1} \frac{1}{n} \binom{n}{k} \binom{2n-2}{n-k-1} \ell^{n-k-1} m^k,$$
and for $\ell + 1 \leq i \leq \ell+m$ and $n \geq 2$,
    $$t_{A(\ell,m)}^{(i)}(n) = \sum_{k=1}^{n-1} \frac{1}{n-1} \binom{n-1}{k-1} \binom{2n-2}{n-k-1} \ell^{n-k} m^{k-1}.$$
\end{theorem}

The generating functions $F_{A(\ell,m)}^{(i)}$ satisfy cubic equations, as will be shown in the following. 
Note that $F_{A(\ell,m)}^{(i)} = F_{A(\ell,m)}^{(1)}$ for $1 \leq i \leq \ell$, and $F_{A(\ell,m)}^{(i)} = F_{A(\ell,m)}^{(\ell+1)}$ for $\ell+1 \leq i \leq \ell+m$. For simplicity, let us abbreviate these generating functions as $g(x)$ and $h(x)$ respectively. This implies 
\[g(x) = \frac{x}{1-\ell g(x) - m h(x)}\]
and
\[h(x) = \frac{x}{1-\ell g(x)}\]
by Theorem~\ref{prop:gfgeneral}.
Plugging the second equation into the first yields
\[g(x) = \frac{x}{1-\ell g(x) - \frac{m x}{1-\ell g(x)}},\]
which simplifies to
\[\ell^2 g(x)^3 - 2\ell g(x)^2 + ((\ell-m)x+1) g(x) - x = 0.\]
In this equation, we can substitute $g(x) = \frac{1}{\ell} (1 - \frac{x}{h(x)})$
and obtain an equation for $h(x)$ after some simplifications:
\[m h(x)^3 + (\ell-m)x h(x)^2 - x h(x) + x^2 = 0.\]
Hence, we have the following theorem.

\begin{theorem}\label{thm:removesquare-gfs}
For all positive integers $\ell,m$, we have
$$\ell^2 F_{A(\ell,m)}^{(i)}(x)^3 - 2\ell F_{A(\ell,m)}^{(i)}(x)^2 + ((\ell-m)x+1) F_{A(\ell,m)}^{(i)}(x) - x = 0$$
for $1 \leq i \leq \ell$, and
$$ m F_{A(\ell,m)}^{(i)}(x)^3 + (\ell-m)x F_{A(\ell,m)}^{(i)}(x)^2 - x F_{A(\ell,m)}^{(i)}(x) + x^2 = 0$$
for $\ell + 1 \leq i \leq \ell + m$.
\end{theorem}

We can alternatively also derive  sum representations for the sequences: we solve the equation
$$g(x) = \frac{x}{1-\ell g(x) - m h(x)} = \frac{x}{1-\ell g(x) - \frac{mx}{1-\ell g(x)}}$$
for $x$ to obtain
$$x = \frac{g(x)(1-\ell g(x))^2}{1 + (m-\ell)g(x)}$$
or
$$g(x) = \frac{x(1+(m-\ell)g(x))}{(1-\ell g(x))^2}.$$
Letting $[x^i]f(x)$ denote the coefficient of $x^i$ in formal power series $f(x)$, the Lagrange inversion formula gives us
\begin{align*}
[x^n] g(x) &= \frac{1}{n} [t^{n-1}] \Big( \frac{(1+(m-\ell)t)}{(1-\ell t)^2} \Big)^n \\
&= \frac{1}{n} [t^{n-1}] (1+(m-\ell)t)^n (1-\ell t)^{-2n} \\
&= \frac{1}{n} \sum_{k=0}^{n-1} (m-\ell)^{n-k-1} \ell^k \binom{n}{n-k-1} \binom{2n+k-1}{k} \\
&= \frac{1}{n} \sum_{k=1}^{n} (m-\ell)^{n-k} \ell^{k-1} \binom{n}{k} \binom{2n+k-2}{k-1}.
\end{align*}
Moreover, since $h(x) = \frac{x}{1-\ell g(x)}$, the Lagrange--B\"urmann formula~\cite[Corollary 5.4.3]{Stanley1999Enumerative} yields
\begin{align*}
[x^n] h(x) &= [x^{n-1}] \frac{1}{1- \ell g(x)} \\
&= \frac{1}{n-1} [t^{n-2}] \frac{\ell}{(1- \ell t)^2}
\Big( \frac{(1+(m-\ell)t)}{(1-\ell t)^2} \Big)^{n-1} \\
&= \frac{\ell}{n-1} [t^{n-2}] (1+(m-\ell)t)^{n-1} (1-\ell t)^{-2n} \\
&= \frac{\ell}{n-1} \sum_{k=0}^{n-2} (m-\ell)^{n-k-2} \ell^k \binom{n-1}{n-k-2} \binom{2n+k-1}{k} \\
&= \frac{1}{n-1} \sum_{k=1}^{n-1} (m-\ell)^{n-k-1} \ell^k \binom{n-1}{k} \binom{2n+k-2}{k-1}.
\end{align*}

Thus, we have the following theorem.

\begin{theorem}\label{thm:removesquare-alternative}
For all positive integers $\ell,m,n$ and $1 \leq i \leq \ell$,
    $$t_{A(\ell,m)}^{(i)}(n) = \frac{1}{n} \sum_{k=1}^{n} (m-\ell)^{n-k} \ell^{k-1} \binom{n}{k} \binom{2n+k-2}{k-1}.$$
For $\ell + 1 \leq i \leq \ell+m$ and $n \geq 2$,
    $$t_{A(\ell,m)}^{(i)}(n) = \frac{1}{n-1} \sum_{k=1}^{n-1} (m-\ell)^{n-k-1} \ell^k \binom{n-1}{k} \binom{2n+k-2}{k-1}.$$
In total, we have, for $n \geq 2$,
    $$t_{A(\ell,m)}(n) = \sum_{k=1}^{n} \frac{(2m-\ell)n - m k + (\ell-m)}{n(n-1)}  (m-\ell)^{n-k-1} \ell^k \binom{n}{k} \binom{2n+k-2}{k-1}.$$
\end{theorem}

Yet another set of sum formulas is obtained by means of the substitution $g(x) = \frac{Y}{1+(\ell-m)Y}$, where $Y$ is a new formal variable. Then, we have
$$h(x) = \frac{x}{1-\ell g(x)} = \frac{x(1+(\ell-m)Y)}{1-mY}$$
and thus
$$\frac{Y}{1+(\ell-m)Y} = g(x) = \frac{x}{1-\ell g(x) - m h(x)} = \frac{x}{1- \frac{\ell Y}{1+(\ell-m)Y}
- \frac{m x(1+(\ell-m)Y)}{1-m Y}}.$$
Solving for $x$, this yields
$$x = \frac{Y(1-mY)^2}{(1+(\ell-m)Y)^2}.$$
Now, we can apply the Lagrange--B\"urmann formula to obtain
\begin{align*}
[x^n] g(x) &= [x^n] \frac{Y}{1+(\ell-m)Y} 
= \frac{1}{n} [y^{n-1}] \frac{1}{(1+(\ell-m)y)^2} \Big( \frac{(1+(\ell-m)y)^2}{(1-my)^2} \Big)^n \\
&= \frac{1}{n} [y^{n-1}] \frac{(1+(\ell-m)y)^{2n-2}}{(1-my)^{2n}}
= \frac{1}{n} \sum_{k=0}^{n-1} (\ell-m)^k m^{n-k-1} \binom{2n-2}{k} \binom{3n-k-2}{n-k-1}
\end{align*}
as well as
\begin{align*}
[x^n] h(x) &= [x^n] \frac{x (1+(\ell-m)Y)}{1-mY} = [x^{n-1}] \frac{1+(\ell-m)Y}{1-mY} \\
&= \frac{1}{n-1} [y^{n-2}] \frac{\ell}{(1-my)^2} \Big( \frac{(1+(\ell-m)y)^2}{(1-my)^2} \Big)^{n-1} = \frac{\ell}{n-1} [y^{n-2}]
\frac{(1+(\ell-m)y)^{2n-2}}{(1-my)^{2n}} \\
&= \frac{\ell}{n-1} \sum_{k=0}^{n-2} (\ell-m)^k m^{n-k-2} \binom{2n-2}{k} \binom{3n-k-3}{n-k-2}.
\end{align*}
We summarize this in the following theorem providing alternative summation formulas for the quantities in the last Theorem~\ref{thm:removesquare-alternative}.

\begin{theorem}\label{thm:removesquare-alternative2}
For all positive integers $\ell,m,n$ and $1 \leq i \leq \ell$,
    $$t_{A(\ell,m)}^{(i)}(n) = \frac{1}{n} \sum_{k=0}^{n-1} (\ell-m)^k m^{n-k-1} \binom{2n-2}{k} \binom{3n-k-2}{n-k-1}.$$
For $\ell + 1 \leq i \leq \ell+m$ and $n \geq 2$,
    $$t_{A(\ell,m)}^{(i)}(n) = \frac{\ell}{n-1} \sum_{k=0}^{n-2} (\ell-m)^k m^{n-k-2} \binom{2n-2}{k} \binom{3n-k-3}{n-k-2}.$$
In total, we have, for $n \geq 2$,
    $$t_{A(\ell,m)}(n) = \frac{2\ell}{n} \sum_{k=0}^{n-1} (\ell-m)^k m^{n-k-1} \binom{2n-3}{k} \binom{3n-k-3}{n-k-1}.$$
\end{theorem}

Note that the formulas in the two previous theorems simplify further for $\ell=m$ (since then only one term of the sum remains).

\subsection{Family 3}

We now encounter two families of sparse coloring matrices exhibiting similar behavior to Entries~15 and~29 in Appendix~\ref{app:3by3} in that 
they are tree coloring equivalent without being strongly tree coloring equivalent.
\begin{theorem}\label{thm:sparsefamily}
For $m\geq4$, define $m\times m$ matrices
\[
A=\begin{bmatrix}
    0&0&0&0&\cdots&0\\
    0&0&0&0&\cdots&0\\
    0&1&0&0&\cdots&0\\
    0&1&0&0&\cdots&0\\
    \vdots&\vdots&\vdots&\vdots&\ddots&\vdots\\
    0&1&0&0&\cdots&0\\
    1&1&0&0&\cdots&0
\end{bmatrix}
\]
and
\[
B=\begin{bmatrix}
    0&0&0&0&\cdots&0\\
    0&0&0&0&\cdots&1\\
    0&0&0&0&\cdots&1\\
    1&0&0&0&\cdots&0\\
    \vdots&\vdots&\vdots&\vdots&\ddots&\vdots\\
    1&0&0&0&\cdots&0\\
    1&0&0&0&\cdots&0
\end{bmatrix}.
\]
For all $n\geq2$, we have $t_A(n)=t_B(n)=2^{n-1}+m-3$. (Also, of course $t_A(1)=t_B(1)=m$.) However, $A$ and $B$ are not strongly tree coloring equivalent.
\end{theorem}
\begin{proof}
        Matrix $A$ specifies a coloring rule where a vertex of any color from 3 through $m-1$ must be followed by a vertex of color~2 and where a vertex of color~$m$ must be followed by a vertex of color~1 or color~2. Since vertices of colors~1 and~2 cannot be followed by anything, this rule can only be used to color trees of height at most  $1$. Given $n\geq2$, there is one coloring with each root color from~3 through $m-1$ (all leaves have color~2). There are $2^{n-1}$ colorings with roots of color~$m$ (two choices, color~1 or color~2, for each leaf independently). Thus, there are $2^{n-1}+m-3$ colorings according to $A$.
        
        For $B$, a vertex of color~2 or~3 must be followed by a vertex of color~$m$, and a vertex of any color from $4$ through $m$ must be followed by a vertex of color~1. Since vertices of color~1 cannot be followed by anything and every descending path eventually reaches color~1, this rule can only be used to color trees of height at most $2$. On $n\geq2$ vertices, there are $2^{n-2}$ such trees. There are exactly two ways to color most of these trees: A root of color~2 or~3 which then forces the root's children to be of color~$m$ and these vertices' children (which are the root's grandchildren) to be of color~1. There are $m-3$ additional colored trees: the trees of height $1$ with a root of some color from $4$ to $m$. Thus, there are a total of $2^{n-1}+m-3$ trees colored according to $B$.

        Finally, to see that $A$ and $B$ are not strongly tree coloring equivalent, note that $A$ contains a row with sum $2$, whereas $B$ contains no such row. 
        So, Theorem~\ref{prop:equalones} tells us that they are not strongly tree coloring equivalent.
    \end{proof}

    While Theorem~\ref{thm:sparsefamily} leads only to sparse matrices, we can obtain corresponding dense matrices via Corollary~\ref{cor:complement}. So, for $m\geq4$, there exists a pair of $m\times m$ coloring matrices, each with exactly $m-1$ zeroes, that are tree coloring equivalent but not strongly tree coloring equivalent.

\subsection{Family 4}

As our next example, we describe a variant of $m$-plane trees (Theorem~\ref{thm:k-plane}) that also gives rise to hypergeometric counting formulas. Let $m = 2\ell$ be an even integer, and define the $m \times m$ matrix $C^{(m)}$ by its entries $c_{ij}^{(m)}$ as follows:

$$c_{ij}^{(m)} = \begin{cases} 1 & i+j \leq m+1, \text{ except when } i+j = m \text{ and } i,j \text{ are odd,} \\ 0 & \text{otherwise.} \end{cases}$$

For example, when $m=6$, the matrix is
$$C^{(6)} = \begin{bmatrix}
    1 & 1 & 1 & 1 & 0 & 1 \\
    1 & 1 & 1 & 1 & 1 & 0 \\
    1 & 1 & 0 & 1 & 0 & 0 \\
    1 & 1 & 1 & 0 & 0 & 0 \\
    0 & 1 & 0 & 0 & 0 & 0 \\
    1 & 0 & 0 & 0 & 0 & 0 \\
\end{bmatrix}\,.$$

For this coloring rule, we have the following theorem.

\begin{theorem}\label{thm:kplane-variant}
For every even integer $m = 2\ell$ and every $n \geq 1$, we have $$t_{C^{(m)}}(n) = \frac{\ell}{n} \cdot 2^{2n-1} \binom{\frac{\ell+1}{2} n - \frac{\ell}{2}-1}{n-1}.$$
\end{theorem}

In order to prove Theorem~\ref{thm:kplane-variant}, we need a hypergeometric identity.
\begin{lemma}\label{lem:hyper}
For every nonnegative integer $n$ and every real number $x$, we have
$$\sum_{k=0}^n \binom{2n-k}{n-k} (-2)^k \binom{x}{k} = 4^n \binom{n-\frac{x+1}{2}}{n}.$$
\end{lemma}

In fact, the identity in Lemma~\ref{lem:hyper} can be reduced to a known identity. Recall that the hypergeometric function ${}_2F_1(a,b;c;z)$ is defined by
$${}_2F_1(a,b;c;z) = \sum_{k=0}^{\infty}
\frac{a^{\overline{k}}b^{\overline{k}}}{c^{\overline{k}}} \frac{z^k}{k!},$$
where $a^{\overline{k}} = a(a+1)\cdots(a+k-1)$ is the rising factorial.

\begin{proof}
Observe that
\begin{align*}
{}_2F_1(-x,-n;-2n;z) &= \sum_{k=0}^{\infty} \frac{(-x)^{\overline{k}}(-n)^{\overline{k}}}{(-2n)^{\overline{k}}} \frac{z^k}{k!} = \sum_{k=0}^{\infty} \frac{x^{\underline{k}}n^{\underline{k}}}{(2n)^{\underline{k}}} \frac{(-z)^k}{k!} \\
&= \sum_{k=0}^{\infty} \frac{\binom{x}{k} \binom{n}{k}}{\binom{2n}{k}} (-z)^k = \sum_{k=0}^n \frac{\binom{x}{k} \binom{n}{k}}{\binom{2n}{k}} (-z)^k \\
&= \frac{1}{\binom{2n}{n}} \sum_{k=0}^n \binom{x}{k} \binom{2n-k}{n-k} (-z)^k.
\end{align*}
Hence, we have
$$\sum_{k=0}^n \binom{2n-k}{n-k} (-2)^k \binom{x}{k} =
\binom{2n}{n} {}_2F_1(-x,-n;-2n;2).
$$
Now \cite[p.415]{Koepf1995Algorithms} gives the identity
$${}_2F_1(-x,-n;-2n;2) = \frac{(1/2)^{\overline{x/2}}}{(1/2-n)^{\overline{x/2}}}$$
for even positive integers $x$. 
Letting $\Gamma$ denote the gamma function, it follows that
\begin{align*}
    \sum_{k=0}^n \binom{2n-k}{n-k} (-2)^k \binom{x}{k} &=  \binom{2n}{n} \frac{(1/2)^{\overline{x/2}}}{(1/2-n)^{\overline{x/2}}} = \binom{2n}{n} \frac{\Gamma(\frac{x+1}{2})/\Gamma(\frac12)}{\Gamma(\frac{x+1}{2}-n)/\Gamma(\frac12-n)} \\
    &= \binom{2n}{n} \frac{\Gamma(\frac{x+1}{2})/\Gamma(\frac{x+1}{2}-n)}{\Gamma(\frac12)/\Gamma(\frac12-n)} = \binom{2n}{n} \frac{(\frac{x-1}{2})^{\underline{n}}}{(-\frac12)^{\underline{n}}} \\
    &= \binom{2n}{n} \frac{(\frac{1-x}{2})^{\overline{n}}}{(\frac12)^{\overline{n}}} = \binom{2n}{n} \frac{(n - \frac{x+1}{2})^{\underline{n}}}{(n-\frac12)^{\underline{n}}} \\
    &= \frac{(2n)!}{n!(n-\frac12)^{\underline{n}}} \cdot \frac{(n - \frac{x+1}{2})^{\underline{n}}}{n!} = \frac{(2n)!!}{n!} \cdot \frac{(2n-1)!!}{(n-\frac12)^{\underline{n}}} \cdot \frac{(n - \frac{x+1}{2})^{\underline{n}}}{n!} \\
    &= 2^n \cdot 2^n \cdot \binom{n-\frac{x+1}{2}}{n},
\end{align*}
where $a^{\underline{n}} = a(a-1)\cdots(a-n+1)$ is the falling factorial.
This proves the lemma (a priori only for even positive integers, but since both sides are polynomials of degree $n$ in $x$, the identity must hold for all real $x$).
\end{proof}

We are now ready to prove Theorem~\ref{thm:kplane-variant}.

\begin{proof}[Proof of Theorem~\ref{thm:kplane-variant}]
    By Theorem~\ref{prop:gfgeneral}, we have
\begin{equation}\label{eq:gfeq-cm}
    F_{C^{(m)}}^{(i)}(x)=x+ F_{C^{(m)}}^{(i)}(x) \sum_{j=1}^mc_{ij}^{(m)}F_{C^{(m)}}^{(i)}(x).
    \end{equation}
It is straightforward to prove (e.g., by induction or by applying Theorem~\ref{thm:EqivColsStrictIff}) that colors $2j-1$ and $2j$ are interchangeable for all $j \in \{1,2,\ldots,\ell\}$, i.e.,
$$F_{C^{(m)}}^{(2j-1)}(x) = F_{C^{(m)}}^{(2j)}(x).$$
Let us set $G_j(x) = F_{C^{(m)}}^{(2j-1)}(x)$ (for simplicity, we drop the dependence on $m$). Then,~\eqref{eq:gfeq-cm} becomes
\begin{equation}\label{eq:G-eq}
G_j(x) = x + G_j(x) \Big( 2 \sum_{i=1}^{\ell-j} G_i(x) + G_{\ell+1-j}(x) \Big),
\end{equation}
for all $j \in \{1,2,\ldots,\ell\}$.
Motivated by the treatment of $m$-plane trees in \cite{GPW2010Bijections}, we use the \emph{ansatz} $G_j(x) = U V^{j-1}$, where $U$ and $V$ are power series. Then, by summing the geometric series in~\eqref{eq:G-eq}, we obtain the equation
\[UV^{j-1} = x + \frac{2U^2 V^{j-1}}{1-V} - \frac{U^2(1+V)V^{\ell-1}}{1-V}.
\]
The only terms that depend on $j$ become equal if we set $V = 1 - 2U$. Then, we are only left with the equation
\[x = U(1-U)(1-2U)^{\ell-1},\]
which determines the power series $U = G_1(x)$. The generating function for all trees (regardless of root color) that follow the coloring rule determined by $C^{(m)}$ is
\[F_{C^{(m)}}(x) = \sum_{i=1}^m F_{C^{(m)}}^{(i)}(x) = 2 \sum_{j=1}^{\ell} G_j(x) = \frac{2U(1-V^{\ell})}{1-V} = 1-V^{\ell} = 1 - (1-2U)^{\ell}.\]
Now, the coefficients of $F_{C^{(m)}}(x)$ can be determined by means of the Lagrange--B\"urmann formula. We have
\begin{align*}
t_{C^{(m)}}(n) &= \frac{2\ell}{n} [u^{n-1}] (1-u)^{-n}(1-2u)^{-(\ell-1)(n-1)} \\
&= \frac{2\ell}{n} \sum_{k=0}^{n-1} [u^{n-k-1}](1-u)^{-n} [u^k] (1-2u)^{-(\ell-1)(n-1)} \\
&= \frac{2\ell}{n} \sum_{k=0}^{n-1} (-1)^{n-k-1} \binom{-n}{n-k-1} \cdot (-2)^k \binom{-(\ell-1)(n-1)}{k} \\
&= \frac{2\ell}{n} \sum_{k=0}^{n-1} \binom{2n-k-2}{n-k-1} \cdot (-2)^k \binom{-(\ell-1)(n-1)}{k}.
\end{align*}

Now, if we replace $n$ by $n-1$ and $x$ by $-(\ell-1)(n-1)$ in Lemma~\ref{lem:hyper}, we obtain exactly the required sum, showing that
\[t_{C^{(m)}}(n) = \frac{2\ell}{n} \cdot 4^{n-1} \binom{n -1 + \frac{(\ell-1)(n-1)-1}{2}}{n-1} = \frac{\ell}{n} \cdot 2^{2n-1} \binom{\frac{\ell+1}{2} n - \frac{\ell}{2}-1}{n-1}.\]
This completes the proof.
\end{proof}

Let us remark that since the coloring rule associated with $C^{(m)}$ has $\frac{m}{2}$ pairs of interchangeable colors, there are many more matrices that yield the same sequence (by Corollary~\ref{cor:EqivCols}).

\subsection{Family 5}

As another example of an infinite family of coloring rules, let us generalize Entry~30 in Appendix~\ref{app:3by3}. As discussed in Subsection~\ref{ss:3by3}, the coloring rules classified by Entry~30 are unique in that their counting sequences include Fibonacci numbers. In fact, this is the only entry with a rational generating function, aside from the many constant or geometric sequences.
One matrix equivalent to the matrix listed in Entry~30 in Appendix~\ref{app:3by3} is $\begin{bmatrix}
    0&0&0\\
    1&0&0\\
    1&1&0
\end{bmatrix}$; it is this matrix that we generalize. Define $A_m$ to be the matrix with zeroes on and above the main diagonal and ones below the main diagonal. This is the matrix for the $m$-color rule where each color can be followed by any color with lesser value. Using this coloring rule, color~$m$ can only appear as the root color. It immediately follows that the number of colored trees with any given root color does not depend on the total number of colors available (provided there are enough colors to use that color for the root). With this in mind, throughout our discussion of this family of rules we drop subscripts on generating functions. That is, $\gfnoa{i}$ stands for $\gff{i}{A_j}$ for any $j\geq i$.
\begin{theorem}\label{thm:genfib}
    For $m\geq1$, define polynomials $p_m(x)$ and $q_m(x)$ recursively as follows:
    \[
    p_m(x)=\begin{cases}
        x & m = 1,\\
        p_{m-1}(x)q_{m-1}(x) & m>1,
    \end{cases}\]
and
    \[q_m(x)=\begin{cases}
        1 & m = 1,\\
        q_{m-1}(x)^2-\frac{p_{m-1}(x)^2}{x} & m>1.
    \end{cases}
    \]
    For $m\geq1$, we have that the generating function $\gfnoa{m}$ is rational, with $\gfnoa{m}=\frac{p_m(x)}{q_m(x)}$.
\end{theorem}
\begin{proof}    
    The 
    theorem statement is implied by the claim about the generating function (plus inspecting the case $m=2$, which is Entry~2 in Table~\ref{tb:2cols}). Observe that $A_m$ is a block matrix for $m\geq2$:
    \[
    A_m=\left[\begin{array}{c|c}A_{m-1}&0\\\hline1&0\end{array}\right].
    \]
    Theorem~\ref{prop:rootonly} now applies (with $m'=1$) and tells us that
    \[
    \gfll{A_m}=\gfll{A_{m-1}}+\frac{x}{1-\gfll{A_{m-1}}}.
    \]
    We note that $\gfnoa{m}=\gfll{A_m}-\gfll{A_{m-1}}$ and that
    \[
    \gfll{A_{m-1}}=\sum_{i=1}^{m-1}\gfnoa{i}.
    \]
    Thus, we have
    \[
    \gfnoa{m}=\frac{x}{1-\sum_{i=1}^{m-1}\gfnoa{i}}.
    \]

    We now prove the theorem by induction. As a base case, we observe that $\gfnoa{1}=x=\frac{x}{1}=\frac{p_1(x)}{q_1(x)}$. Now, suppose we know that $\gfnoa{m-1}=\frac{p_{m-1}(x)}{q_{m-1}(x)}$. We have
    \begin{align*}
        \gfnoa{m}&=\frac{x}{1-\sum_{i=1}^{m-1}\gfnoa{i}}\\
        &=\frac{x}{1-\sum_{i=1}^{m-2}\gfnoa{i}-\gfnoa{m-1}}\\
        &=\frac{x}{1-\sum_{i=1}^{m-2}\frac{p_i(x)}{q_i(x)}-\frac{p_{m-1}(x)}{q_{m-1}(x)}}\\
        &=\frac{x\prod_{i=1}^{m-2}q_i(x)}{\left(1-\sum_{i=1}^{m-2}\frac{p_i(x)}{q_i(x)}\right)\prod_{i=1}^{m-2}q_i(x)-\frac{p_{m-1}(x)}{q_{m-1}(x)}\prod_{i=1}^{m-2}q_i(x)}.
    \end{align*}
    By similar steps, we can find that
    \begin{align*}
        \gfnoa{m-1}&=\frac{x}{1-\sum_{i=1}^{m-2}\gfnoa{i}}\\
        &=\frac{x}{1-\sum_{i=1}^{m-2}\frac{p_i(x)}{q_i(x)}}\\
        &=\frac{x\prod_{i=1}^{m-2}q_i(x)}{\left(1-\sum_{i=1}^{m-2}\frac{p_i(x)}{q_i(x)}\right)\prod_{i=1}^{m-2}q_i(x)}.
    \end{align*}
    This second expression has polynomials in its numerator and denominator, as the last step cleared the denominators. Also, both the numerator and denominator have degree $2^{m-2}$, as the degree of each $p_i(x)$ and each $q_i(x)$ is $2^{i-1}$ for $i\geq2$ (and $p_1(x)$ has degree $1$ while $q_1(x)$ has degree $0$). The degrees of $p_{m-1}$ and $q_{m-1}$ are also $2^{m-2}$. 
    Furthermore, $p_i(x)$ and $q_i(x)$ are always monic, so both the numerator and denominator are monic. This means that 
    the numerator equals $p_{m-1}(x)$ and the denominator equals $q_{m-1}(x)$.
    Notably, observe that $p_{m-1}(x)=x\prod_{i=1}^{m-2}q_i(x)$. 
    Hence, we have
    \begin{align*}
        \gfnoa{m}&=\frac{p_{m-1}(x)}{q_{m-1}(x)-\frac{p_{m-1}(x)}{q_{m-1}(x)}\prod_{i=1}^{m-2}q_i(x)}\\
        &=\frac{p_{m-1}(x)q_{m-1}(x)}{q_{m-1}(x)^2-p_{m-1}(x)\prod_{i=1}^{m-2}q_i(x)}.
    \end{align*}
    This expression has a degree $2^{m-1}$ monic polynomial in its numerator and a degree $2^{m-1}$ monic polynomial in its denominator. Observe that the numerator is precisely $p_m(x)$, as required. To handle the denominator, our earlier observation that $p_{m-1}(x)=x\prod_{i=1}^{m-2}q_i(x)$ implies that $\prod_{i=1}^{m-2}q_i(x)=\frac{p_{m-1}(x)}{x}$. 
    We have that the denominator of $\gfnoa{m}$ equals
    \[
    q_{m-1}(x)^2-p_{m-1}(x)\prod_{i=1}^{m-2}q_i(x)=q_{m-1}(x)^2-\frac{p_{m-1}(x)^2}{x}=q_m(x),
    \]
    as required.
\end{proof}

Theorem~\ref{thm:genfib} has a couple immediate corollaries. The first is a tie-in to some existing OEIS sequences.
\begin{corollary}\label{cor:recippolys}
    Define polynomials $p_m(x)$ and $q_m(x)$ for $m\geq1$ as in the statement of Theorem~\ref{thm:genfib}. For all $m\geq3$, we have $p_m(x^2)=-xT_m(x)$ and $q_m(x^2)=S_m(x)$, where $S_m(x)$ and $T_m(x)$ come from the entries for sequences \seqnum{A147985} and \seqnum{A147986} and satisfy the mutual recurrences $S_m(x)=S_{m-1}(x)^2-T_{m-1}(x)^2$ and $T_m(x)=S_{m-1}(x)T_{m-1}(x)$ with initial conditions $S_1(x)=x$ and $T_1(x)=1$.
\end{corollary}
\begin{proof}
    The claimed fact can be easily checked for $m=3$, which serves as a base case for an induction argument. 
    Now, suppose inductively that for some $m\geq4$, $p_{m-1}(x^2)=-xT_{m-1}(x)$ and $q_{m-1}(x^2)=S_m(x)$. 
    We have
    \[
    p_m(x^2)=p_{m-1}(x^2)q_{m-1}(x^2)=-xT_{m-1}(x)S_{m-1}(x)=-xT_m(x)
    \]
    and
    \begin{align*}
    q_m(x^2)&=q_{m-1}(x^2)^2-\frac{p_{m-1}(x^2)^2}{x^2}=S_{m-1}(x)^2-\frac{\left(-xT_{m-1}(x)\right)^2}{x^2}\\
    &=S_{m-1}(x)^2-T_{m-1}(x)^2=S_m(x),
    \end{align*}
    as required.
\end{proof}
The second corollary describes the structure of the counting sequences.
\begin{corollary}
    For $m\geq2$, the sequence $t^{(m)}_{A_m}(n)$ for $n\geq1$ satisfies a homogeneous linear recurrence relation of order $2^{m-2}$ with constant coefficients. Moreover, the characteristic polynomial of this recurrence has coefficients coming from \seqnum{A147990} and is palindromic for $m\geq3$.
\end{corollary}
\begin{proof}
    By Theorem~\ref{thm:genfib}, we have that the generating function $\gfnoa{m}=\frac{p_m(x)}{q_m(x)}$, where $p_m(x)$ and $q_m(x)$ are defined in the statement of Theorem~\ref{thm:genfib}. Looking at the recurrences defining those polynomials, the degree of $q_m(x)$ is $2^{m-2}$ when $m\geq2$. This immediately implies that the sequence $t^{(m)}_{A_m}(n)$ for $n\geq1$ satisfies a homogeneous linear recurrence relation of order $2^{m-2}$ with constant coefficients. By Corollary~\ref{cor:recippolys}, the coefficients of $q_m(x)$ come from sequence \seqnum{A147985} with some zeroes removed. The result is precisely sequence \seqnum{A147990}. Based on properties of these sequences, we have that $q_m(x)$ is palindromic for all $m\geq3$.
\end{proof}

\subsection{Family 6}

A related phenomenon occurs when the coloring rules from Family~5 are modified so that each color is also allowed to follow itself. This idea generalizes Entry~6 in Table~\ref{tb:2cols} and Entry~33 in Appendix~\ref{app:3by3}. Now, the generating functions are not rational, but it turns out that the algebraic relations satisfied by the generating functions themselves are related to the same group of OEIS sequences. Sequences \seqnum{A147987} and \seqnum{A147988} define families of polynomials $P_m(x)$ and $Q_m(x)$ for $m\geq1$ recursively as follows:
\[
    P_m(x)=\begin{cases}
        x & m = 1\\
        P_{m-1}(x)^2+Q_{m-1}(x)^2 & m>1
    \end{cases}
    \quad \text{and} \quad
    Q_m(x)=\begin{cases}
        1 & m = 1\\
        P_{m-1}(x)\cdot Q_{m-1}(x) & m>1.
    \end{cases}
    \]

Define $B_m$ to be the matrix with zeroes above the main diagonal and ones on and below the main diagonal. This is the matrix for the $m$-color rule where each color can be followed by any color with lesser or equal value. As with the $A_m$ rule, for this coloring rule color~$m$ can only appear as the root color. It immediately follows here too that the number of colored trees with any given root color does not depend on the total number of colors available (provided there are enough colors to use that color for the root). With this in mind, throughout our discussion of this family of rules we also drop subscripts on generating functions. That is, from here onward $\gfnoa{i}$ stands for $\gff{i}{B_j}$ for any $j\geq i$.

The summative result is the following.
\begin{theorem}\label{thm:gen33}   
    For $m\geq1$ we have that the generating function $\gfnoa{m}$ satisfies the following polynomial relation:
    \[
    x^{2^{m-1}}\left(\frac{1}{\sqrt{x}}Q_{m+1}\!\left(\frac{\gfnoa{m}}{\sqrt{x}}\right)-P_{m+1}\!\left(\frac{\gfnoa{m}}{\sqrt{x}}\right)\right).
    \]
\end{theorem}

The following lemmas are key to proving Theorem~\ref{thm:gen33}.
\begin{lemma}\label{lem:gen33gf}
    For all $m\geq1$, we have
    \[
    \gfnoa{m}\cdot\left(1-\sum_{i=1}^{m}\gfnoa{i}\right)=x.
    \]
\end{lemma}
\begin{proof}
    Observe that $B_m$ is a block matrix for $m\geq2$:
    \[
    B_m=\left[\begin{array}{c|c}B_{m-1}&0\\\hline1&1\end{array}\right].
    \]
    Theorem~\ref{prop:upsetonly} now applies (with $m'=1$),
    which tells us that
    \[
    \gfll{B_m}=\gfll{B_{m-1}}+\gfll{B_m}^2-\gfll{B_{m-1}}\gfll{B_m}+x.
    \]
    Using the fact that $\gfnoa{m}=\gfll{B_m}-\gfll{B_{m-1}}$ yields
    \begin{align*}
    \gfnoa{m}&=\gfll{B_m}^2-\gfll{B_{m-1}}\gfll{B_m}+x\\
    &=\gfll{B_m}\left(\gfll{B_m}-\gfll{B_{m-1}}\right)+x=\gfll{B_m}\gfnoa{m}+x.
    \end{align*}
    Thus, $\gfnoa{m}\cdot\left(1-\gfll{B_m}\right)=x$.
    The fact that
    \[
    \gfll{B_m}=\sum_{i=1}^m\gfnoa{i}
    \]
    completes the proof.
\end{proof}
\begin{lemma}\label{lem:gen33}
    For any $m\geq2$ and $1\leq j<m$, we have
    \[
    \gfnoa{m-j}=\frac{-\sqrt{x}P_{j+1}\!\left(\frac{\gfnoa{m}}{\sqrt{x}}\right)}{Q_{j+1}\!\left(\frac{\gfnoa{m}}{\sqrt{x}}\right)}+1-\sum_{i=1}^{m-j-1}\gfnoa{i}.
    \]
\end{lemma}
\begin{proof}
    We prove the lemma by induction on $j$. As a base case, we verify the statement when $j=1$. Starting from Lemma~\ref{lem:gen33gf}, we have
    \begin{align}
        &\nonumber \gfnoa{m}\cdot\left(1-\sum_{i=1}^{m}\gfnoa{i}\right)=x\\
        &\nonumber \Rightarrow\gfnoa{m}\cdot\left(1-\sum_{i=1}^{m-2}\gfnoa{i}-\gfnoa{m-1}-\gfnoa{m}\right)=x\\
        &\nonumber \Rightarrow1-\sum_{i=1}^{m-2}\gfnoa{i}-\gfnoa{m-1}-\gfnoa{m}=\frac{x}{\gfnoa{m}}\\
        &\nonumber \Rightarrow\gfnoa{m-1}=-\frac{x}{\gfnoa{m}}-\gfnoa{m}+1-\sum_{i=1}^{m-2}\gfnoa{i}\\
        &\Rightarrow\gfnoa{m-1}=\frac{-x-\gfnoa{m}^2}{\gfnoa{m}}+1-\sum_{i=1}^{m-2}\gfnoa{i}\label{eq:firstgfexpression}.
    \end{align}
    Next, we observe that $P_2(x)=P_1(x)^2+Q_1(x)^2=x^2+1$ and $Q_2(x)=P_1(x)\cdot Q_1(x)=x$. This implies
    \[
    P_2\!\left(\frac{\gfnoa{m}}{\sqrt{x}}\right)=\frac{\gfnoa{m}^2}{x}+1
    \]
    and
    \[
    Q_2\!\left(\frac{\gfnoa{m}}{\sqrt{x}}\right)=\frac{\gfnoa{m}}{\sqrt{x}}.
    \]
    By substitution followed by algebraic manipulation, we have 
    \begin{align}
    \frac{-\sqrt{x}P_{2}\!\left(\frac{\gfnoa{m}}{\sqrt{x}}\right)}{Q_{2}\!\left(\frac{\gfnoa{m}}{\sqrt{x}}\right)}&=\frac{-\sqrt{x}\cdot\left(\frac{\gfnoa{m}^2}{x}+1\right)}{\frac{\gfnoa{m}}{\sqrt{x}}}
    =\frac{-x\cdot\left(\frac{\gfnoa{m}^2}{x}+1\right)}{\gfnoa{m}}
    =\frac{-x-\gfnoa{m}^2}{\gfnoa{m}}\label{eq:secondgfexpression}.
    \end{align}
    Thus, we have, by substituting the last expression from~\eqref{eq:secondgfexpression} into~\eqref{eq:firstgfexpression}, 
    \[
    \gfnoa{m-1}=\frac{-\sqrt{x}P_{2}\!\left(\frac{\gfnoa{m}}{\sqrt{x}}\right)}{Q_{2}\!\left(\frac{\gfnoa{m}}{\sqrt{x}}\right)}+1-\sum_{i=1}^{m-2}\gfnoa{i},
    \]
    as required.

    For ease of notation in the inductive step, define
    \[
    R_j=\frac{-\sqrt{x}P_{j+1}\!\left(\frac{\gfnoa{m}}{\sqrt{x}}\right)}{Q_{j+1}\!\left(\frac{\gfnoa{m}}{\sqrt{x}}\right)}.
    \]
    For some $j\geq2$, suppose inductively that
    \[
    \gfnoa{m-j+1}=R_{j-1}+1-\sum_{i=1}^{m-j}\gfnoa{i}.
    \]
    By Lemma~\ref{lem:gen33gf}, we have
    \[
    \gfnoa{m-j+1}\cdot\left(1-\sum_{i=1}^{m-j+1}\gfnoa{i}\right)=x,
    \]
    which implies that
    \[
    \frac{x}{\gfnoa{m-j+1}}=1-\sum_{i=1}^{m-j+1}\gfnoa{i}=1-\sum_{i=1}^{m-j}\gfnoa{i}-\gfnoa{m-j+1}.
    \]
    By the inductive hypothesis, we have
    \[
    \frac{x}{R_{j-1}+1-\sum_{i=1}^{m-j}\gfnoa{i}}=1-\sum_{i=1}^{m-j}\gfnoa{i}-\left(R_{j-1}+1-\sum_{i=1}^{m-j}\gfnoa{i}\right)=-R_{j-1}.
    \]
    Rearranging terms, we see that
    \[
    -\frac{x}{R_{j-1}}=R_{j-1}+1-\sum_{i=1}^{m-j}\gfnoa{i}=R_{j-1}+1-\sum_{i=1}^{m-j-1}\gfnoa{i}-\gfnoa{m-j},
    \]
    and thus,
    \[
    \gfnoa{m-j}=R_{j-1}+\frac{x}{R_{j-1}}+1-\sum_{i=1}^{m-j-1}\gfnoa{i}.
    \]

    To conclude the proof, it suffices to show that
    \[
    R_j=R_{j-1}+\frac{x}{R_{j-1}}=\frac{R_{j-1}^2+x}{R_{j-1}}.
    \]
    We have
    \begin{align*}
    R_{j-1}^2+x&=\left(\frac{-\sqrt{x}P_j\!\left(\frac{\gfnoa{m}}{\sqrt{x}}\right)}{Q_{j}\!\left(\frac{\gfnoa{m}}{\sqrt{x}}\right)}\right)^2+x
    \\
    &=\frac{xP_{j}\!\left(\frac{\gfnoa{m}}{\sqrt{x}}\right)^2}{Q_{j}\!\left(\frac{\gfnoa{m}}{\sqrt{x}}\right)^2}+\frac{xQ_{j}\!\left(\frac{\gfnoa{m}}{\sqrt{x}}\right)^2}{Q_{j}\!\left(\frac{\gfnoa{m}}{\sqrt{x}}\right)^2}
    =\frac{xP_{j+1}\left(\frac{\gfnoa{m}}{\sqrt{x}}\right)}{Q_{j}\!\left(\frac{\gfnoa{m}}{\sqrt{x}}\right)^2}.
    \end{align*}
    Dividing through by $R_{j-1}$ gives
    \begin{align*}
        \frac{R_{j-1}^2+x}{R_{j-1}}&=\frac{xP_{j+1}\left(\frac{\gfnoa{m}}{\sqrt{x}}\right)}{Q_{j}\!\left(\frac{\gfnoa{m}}{\sqrt{x}}\right)^2}\cdot\frac{Q_{j}\!\left(\frac{\gfnoa{m}}{\sqrt{x}}\right)}{-\sqrt{x}P_j\!\left(\frac{\gfnoa{m}}{\sqrt{x}}\right)}
        =\frac{xP_{j+1}\left(\frac{\gfnoa{m}}{\sqrt{x}}\right)}{-\sqrt{x}P_{j}\!\left(\frac{\gfnoa{m}}{\sqrt{x}}\right)Q_{j}\!\left(\frac{\gfnoa{m}}{\sqrt{x}}\right)}\\
        &=\frac{-\sqrt{x}P_{j+1}\!\left(\frac{\gfnoa{m}}{\sqrt{x}}\right)}{Q_{j+1}\!\left(\frac{\gfnoa{m}}{\sqrt{x}}\right)}
        =R_j,
    \end{align*}
    as required.
\end{proof}

Now, we are ready to prove Theorem~\ref{thm:gen33}.
\begin{proof}[Proof of Theorem~\ref{thm:gen33}]
    First, let $m=1$. 
    We have $P_2(x)=x^2+1$ and $Q_2(x)=x$. So,
    \begin{align*}
        x^{2^{m-1}}&\left(\frac{1}{\sqrt{x}}Q_{m+1}\!\left(\frac{\gfnoa{m}}{\sqrt{x}}\right)-P_{m+1}\!\left(\frac{\gfnoa{m}}{\sqrt{x}}\right)\right) \\
        &=x\left(\frac{1}{\sqrt{x}}\cdot\frac{\gfnoa{1}}{\sqrt{x}}-\left(\left(\frac{\gfnoa{1}}{\sqrt{x}}\right)^2+1\right)\right)
        =\gfnoa{1}-\gfnoa{1}^2-x.
    \end{align*}
    This yields $\gfnoa{1}\cdot\left(1-\gfnoa{1}\right)=x$, which agrees with Lemma~\ref{lem:gen33gf}.
    
    Now, suppose $m\geq2$. By Lemma~\ref{lem:gen33},
    \[
    \gfnoa{1}=\frac{-\sqrt{x}P_{m}\!\left(\frac{\gfnoa{m}}{\sqrt{x}}\right)}{Q_{m}\!\left(\frac{\gfnoa{m}}{\sqrt{x}}\right)}+1=1-\frac{x^{2^{m-2}}P_m\!\left(\frac{\gfnoa{m}}{\sqrt{x}}\right)}{x^{2^{m-2}-\frac{1}{2}}Q_m\!\left(\frac{\gfnoa{m}}{\sqrt{x}}\right)},
    \]
    where the latter expression has polynomials in the numerator and denominator of its fraction. By Lemma~\ref{lem:gen33gf}, $\gfnoa{1}\cdot\left(1-\gfnoa{1}\right)=x$. Then,
    \[
    x=\left(1-\frac{x^{2^{m-2}}P_m\!\left(\frac{\gfnoa{m}}{\sqrt{x}}\right)}{x^{2^{m-2}-\frac{1}{2}}Q_m\!\left(\frac{\gfnoa{m}}{\sqrt{x}}\right)}\right)\cdot\frac{x^{2^{m-2}}P_m\!\left(\frac{\gfnoa{m}}{\sqrt{x}}\right)}{x^{2^{m-2}-\frac{1}{2}}Q_m\!\left(\frac{\gfnoa{m}}{\sqrt{x}}\right)}.
    \]
    If we were to combine this expression into a single rational expression, the denominator would be $x^{2^{m-1}-1}Q_m\!\left(\frac{\gfnoa{m}}{\sqrt{x}}\right)^2$. We care about the numerator of this rational expression, as this numerator is a polynomial satisfied by $\gfnoa{m}$. The numerator would be
    \begin{align*}
        &\phantom{=}x^{2^{m-1}-\frac{1}{2}}P_m\!\left(\frac{\gfnoa{m}}{\sqrt{x}}\right)Q_m\!\left(\frac{\gfnoa{m}}{\sqrt{x}}\right)-x^{2^{m-1}}P_m\!\left(\frac{\gfnoa{m}}{\sqrt{x}}\right)^2-x^{2^{m-1}}Q_m\!\left(\frac{\gfnoa{m}}{\sqrt{x}}\right)^2\\
        &=x^{2^{m-1}}\left(\frac{1}{\sqrt{x}}P_m\!\left(\frac{\gfnoa{m}}{\sqrt{x}}\right)Q_m\!\left(\frac{\gfnoa{m}}{\sqrt{x}}\right)-\left(P_m\!\left(\frac{\gfnoa{m}}{\sqrt{x}}\right)^2+Q_m\!\left(\frac{\gfnoa{m}}{\sqrt{x}}\right)^2\right)\right)\\
        &=x^{2^{m-1}}\left(\frac{1}{\sqrt{x}}Q_{m+1}\!\left(\frac{\gfnoa{m}}{\sqrt{x}}\right)-P_{m+1}\!\left(\frac{\gfnoa{m}}{\sqrt{x}}\right)\right),
    \end{align*}
    as required.
\end{proof}

\section{Future Work}\label{sec:future}

This paper illustrates the richness of combinatorial results related to colored plane trees. 
We have classified all coloring rules using at most three colors, and we have several results that apply to certain rules with more colors. A bulk of natural further questions are worth considering.

An obvious next step would be to carry out an exhaustive analysis of the coloring rules with four colors, akin to  Appendix~\ref{app:3by3}. Preliminary computation indicates that there are $2037$ tree coloring equivalence classes of $4\times 4$ matrices and $2079$ strong tree coloring equivalence classes. The same computation yields that the excess $42$ strong tree coloring equivalence classes are found within $39$ of the tree coloring equivalence classes; there are three such classes that include three different matrices, no two of which are strongly tree coloring equivalent. 
Several of the $39$ tree coloring equivalence classes covering multiple strong classes come from combining Entry~15 or Entry~29 from Appendix~\ref{app:3by3} with results from Section~\ref{sec:new_from_old}, and one corresponds to Family~4 from Section~\ref{sec:special}.

In addition, there remain some combinatorial questions related to the objects we have studied. For example, Entry~15 in Appendix~\ref{app:3by3} has two matrices that are not strongly tree coloring equivalent, but we exhibit a direct bijection between trees colored according to those matrices. Does a similar ``nice'' bijection exist for Entry~29 (the other entry with two matrices that are not strongly tree coloring equivalent)? A direct bijective proof of the fact that $(2^{n}+1)C_{n-1}$ counts the trees colored according to the first matrix there, is also of interest. Currently, we have a proof which is partially bijective. 

Other combinatorial questions arise from looking at the OEIS. For example, for the matrix $A$ in Entry~71 of Appendix~\ref{app:3by3}, $t_{A}^{(3)}(n)$ is enumerated by \seqnum{A027307}. This sequence has a description in terms of lattice paths: ``Number of paths from $(0,0)$ to $(3n,0)$ that stay in first quadrant (but may touch horizontal axis) and where each step is $(2,1)$, $(1,2)$ or $(1,-1)$.'' Disallowing the $(2,1)$ step results in sequence \seqnum{A001764}, associated with Entry~7 in Table~\ref{tb:2cols}. We have a bijection (Appendix~\ref{app:ent7bij}) for that sequence with its lattice path interpretation. Is there a ``nice'' extension of that bijection to Entry~71 in Appendix~\ref{app:3by3}? And, is there some generalization that would apply with four or more colors?

Another route of inquiry would be to expand the study of the sequences $t_A(n,q)$ (and $t_A^{(i)}(n,q)$), where a certain number of coloring violations are allowed. We introduce these in this paper primarily to arrive at Corollary~\ref{cor:complement}, which provides extra insight into several of our results. But, it seems likely that rich combinatorics can be found in these sequences if studied in their own right.

Finally, there are undoubtedly several notable infinite families of coloring rules beyond those we introduce in Section~\ref{sec:special}. In particular, perhaps there are some families related to Family~5 where some of the ones are replaced with zeroes. Such coloring rules would also have rational generating functions, as they would also color trees of bounded height.

\bibliographystyle{abbrv}
\bibliography{main}

\appendix

\section{New bijective proof for Entry~7 in Table~\ref{tb:2cols}}\label{app:ent7bij}
Throughout this appendix, let $A=\begin{bmatrix}
    1&1\\1&0
\end{bmatrix}$, and let $\mathcal{T}_n^{A}$ denote the set of $n$-vertex trees ($n\geq1$) colored according to the rule $A$. This corresponds to Entry~7 in Table~\ref{tb:2cols}. Also, throughout this appendix, we refer to color~1 as ``blue'' and color~2 as ``white.''

Here, we exhibit a bijection between $\mathcal{T}_n^{A}$ and the union of the following two sets of Dyck paths:
\begin{itemize}
    \item Dyck paths of length $4n-2$ where all ascents are of even length except for the first one (in bijection with trees with blue root). This is enumerated by \seqnum{A006013}. 
    \item Dyck paths of length $4n-4$ where all ascents are of even length (in bijection with trees with white root). This is enumerated by \seqnum{A001764}. 
\end{itemize}
This union can succinctly be described as the set of Dyck paths of length $4n-4$ or $4n-2$ where all ascents are of even length except possibly the first one.

In any tree, define a \emph{downward path} as any path that starts at a vertex with at least two children, progresses one step downward to a non-leftmost child, and then proceeds from there via leftmost children to a leaf.
Note that the number of downward paths in any tree is one less than the number of leaves and that there is exactly one downward path ending in every leaf except for the leftmost leaf in the entire tree, for which no such path exists.
Let $\mathcal{E}_k$
denotes the set of (non-colored) rooted trees on $k$ vertices where each 
downward path has even length.
Note that if $k$ is even, the leftmost leaf is forced to be at odd depth, and if $k$ is odd, the leftmost leaf is forced to be at even depth. 
This is true since the edges of the tree can be decomposed into a disjoint union of the downward paths corresponding to its non-leftmost leaves along with a path from the root to the leftmost leaf.
We can
observe that $\mathcal{E}_k$ is in bijection with the set of Dyck paths of length $2k-2$ where all ascents are of even length except possibly the first one via the glove bijection between rooted trees and Dyck paths.
So, it suffices to biject elements of $\mathcal{T}_n^{A}$ having blue root with $\mathcal{E}_{2n}$ and elements of $\mathcal{T}_n^{A}$ having white root with $\mathcal{E}_{2n-1}$. We describe a recursive map $\tau:\mathcal{T}_n^{A}\to\mathcal{E}_{2n}\cup\mathcal{E}_{2n-1}$, together with its inverse $\tau^{-1}:\mathcal{E}_{2n}\cup\mathcal{E}_{2n-1}\to\mathcal{T}_n^{A}$. The map $\tau$ sends each vertex of a tree in $\mathcal{T}_n^A$ to a parent-child pair of vertices (an edge) in a tree in $\mathcal{E}_{2n}\cup\mathcal{E}_{2n-1}$, except when the vertex is a white root, which gets sent to a single vertex.

In our construction of $\tau$, we use the following piece of terminology:
    A \emph{left-sibling} of a vertex $v$ in a tree is another vertex $w$ with the same parent as $v$ where $w$ is to the left of $v$ in the ordered tree.
\begin{description}
    \item[Description of $\tau$] 
    Let $T\in\mathcal{T}_n^{A}$. We define the tree $\tau (T) = U\in \left(\mathcal{E}_{2n}\cup\mathcal{E}_{2n-1}\right)$ recursively. 
    If $n=1$ and the single vertex of $T$ is $x$, then $U$ is a single edge if $x$ is colored blue, and a single vertex if $x$ is colored white.
    If $n>1$, let $T'$ denote $T$ with its rightmost leaf deleted, and let $U' = \tau(T')\in\left(\mathcal{E}_{2n-2}\cup\mathcal{E}_{2n-3}\right)$.
    The rightmost leaf $v$ of $T$ has a parent $x$, and there is a pair of vertices $x_1$ and $x_2$ in $U'$ corresponding to $x$, with $x_2$ a child of $x_1$ (see Figure~\ref{fig:TU'}), unless $x$ is the root of $T$ and it is colored white. Then, there is only one corresponding vertex, say $x_2$, which is the root of $U'$.
    
    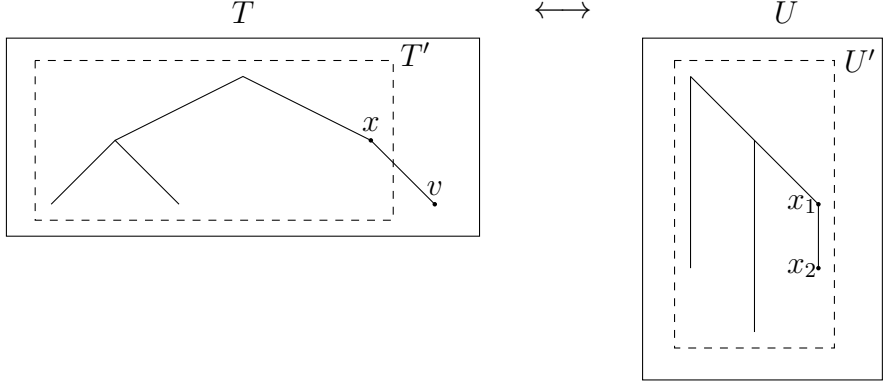
\begin{figure}[htbp]
        \begin{tikzpicture}[scale=0.85]
            \tikzset{blpoint/.style = {circle, fill=blue, draw=blue, inner sep=2pt, text=white}}
            \tikzset{wpoint/.style = {circle, fill=white, draw=black, inner sep=2pt}}
            \draw (-2,-1)--(0,0)--(2,-1)--(3,-2);
            \draw (-3,-2)--(-2,-1)--(-1,-2);

            \draw (-3.7,0.6) rectangle (3.7,-2.5);
            \draw[dashed] (-3.25,0.25) rectangle (2.35,-2.25);

            \node at (2.7,0.35) {$T'$};
            \node at (0,1) {$T$};
            \node at (2,-0.75) {$x$};
            \fill (2,-1) circle (1pt);
            \node at (3,-1.75) {$v$};
            \fill (3,-2) circle (1pt);
            \node at (5,1) {$\longleftrightarrow$};
            \node at (8.5,1) {$U$};
            
            \draw (7,0)--(7,-1)--(7,-2)--(7,-3);
            \draw (7,0)--(8,-1)--(8,-2)--(8,-3)--(8,-4);
            \draw (8,-1)--(9,-2)--(9,-3);

            \draw (6.25,0.6) rectangle (10,-4.75);
            \draw[dashed] (6.75,0.25) rectangle (9.25,-4.25);
            \node at (9.65,0.25) {$U'$};
            \node at (8.75,-2) {$x_1$};
            \fill (9,-2) circle (1pt);
            \node at (8.75,-3) {$x_2$};
            \fill (9,-3) circle (1pt);
            
        \end{tikzpicture}
        \caption{The image $U$ of $T$ is obtained from the image $U'$ of $T'$ ($T$ without its rightmost leaf).}
        \label{fig:TU'}
    \end{figure}
    
    It suffices to consider the three cases below. In each of them, we construct $U$ from $U'$ by adding a new edge $e$ between two new vertices $v_1$ and $v_2$, as described:
    \begin{description}
        \item[$x$ or $v$ is colored white] Note that either $x$ or $v$ is colored white, but not both (if $x$ is white, $v$ must blue). In both cases, add $e$, so that $v_1$ is the rightmost child of $x_2$.
        See Figure~\ref{fig:root1bij}, where $v=e$ (with $v_1=e_1$, $v_2=e_2$) and $x=a$ (with $x_1=a_1$, $x_2=a_2$), for an example.
        \item[$x$ and $v$ are blue, and $v$ has a left-sibling colored white] Let us write $w$ for the nearest left-sibling of $v$ colored white, and let $w_1$ and $w_2$ denote the vertices in $U'$ corresponding to $w$, where $w_1$ is the parent of $w_2$. Add $e$, so that $v_1$ is the rightmost child of $w_1$. See Figure~\ref{fig:root2bij}, where $v=h$ (with $v_1=h_1$, $v_2=h_2$) and $x=e$ (with $x_1=e_1$, $x_2=e_2$), for an example of this case, where $f$ is $v$'s nearest left-sibling that is colored white.        
        \item[$x$ and $v$ are blue, and $v$ has no left-sibling colored white] Add $e$, so that $v_1$ is the rightmost child of $x_1$.
        
    \end{description}

    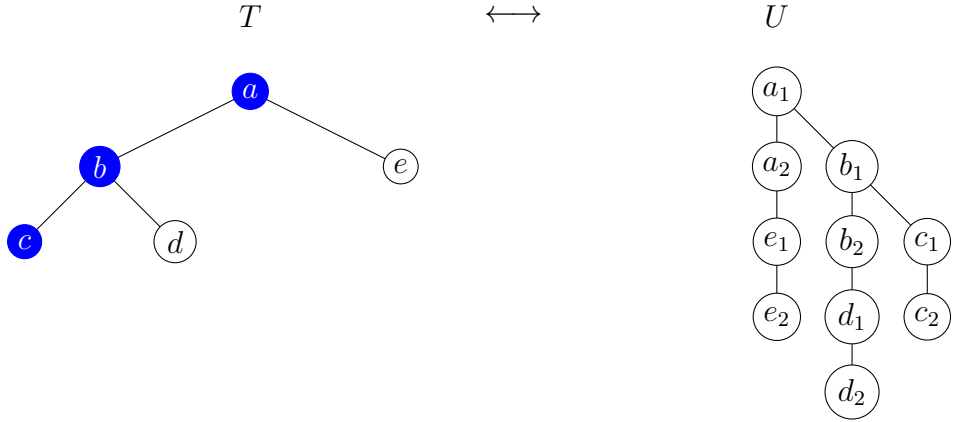
\begin{figure}[htbp]
        \begin{tikzpicture}
            \tikzset{blpoint/.style = {circle, fill=blue, draw=blue, inner sep=2pt, text=white}}
            \tikzset{wpoint/.style = {circle, fill=white, draw=black, inner sep=2pt}}
            \draw (-2,-1)--(0,0)--(2,-1);
            \draw (-3,-2)--(-2,-1)--(-1,-2);
            
            \node at (0,1) {$T$};
            \node at (3.5,1) {$\longleftrightarrow$};
            \node at (7,1) {$U$};
            
            \node[blpoint] at (0,0) {$a$};
            \node[blpoint] at (-2,-1) {$b$};
            \node[wpoint] at (2,-1) {$e$};
            \node[blpoint] at (-3,-2) {$c$};
            \node[wpoint] at (-1,-2) {$d$};
            
            \draw (7,0)--(7,-1)--(7,-2)--(7,-3);
            \draw (7,0)--(8,-1)--(8,-2)--(8,-3)--(8,-4);
            \draw (8,-1)--(9,-2)--(9,-3);
            
            \node[wpoint] at (7,0) {$a_1$};
            \node[wpoint] at (7,-1) {$a_2$};
            \node[wpoint] at (8,-1) {$b_1$};
            \node[wpoint] at (8,-2) {$b_2$};
            \node[wpoint] at (8,-3) {$d_1$};
            \node[wpoint] at (8,-4) {$d_2$};
            \node[wpoint] at (9,-2) {$c_1$};
            \node[wpoint] at (9,-3) {$c_2$};
            \node[wpoint] at (7,-2) {$e_1$};
            \node[wpoint] at (7,-3) {$e_2$};
            
        \end{tikzpicture}
        \caption{Example of the action of the bijection $\tau$. A colored tree $T$ with $5$ vertices 
        maps to a tree $U$ with $10$ vertices. Vertex labels in $T$ are in preorder, which is the order they are added when constructing $U$. $U$ contains two downward paths ($a_1,b_1,b_2,d_1,d_2$ and $b_1,c_1,c_2$), both of even length.}\label{fig:root1bij}
    \end{figure}

    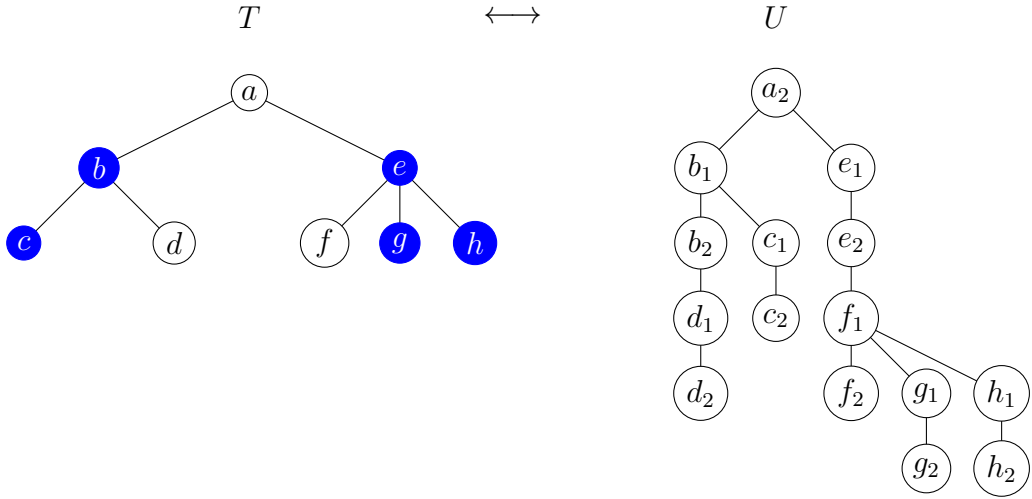
\begin{figure}[htbp]
        \begin{tikzpicture}
            \tikzset{blpoint/.style = {circle, fill=blue, draw=blue, inner sep=2pt, text=white}}
            \tikzset{wpoint/.style = {circle, fill=white, draw=black, inner sep=2pt}}
            \draw (-2,-1)--(0,0)--(2,-1);
            \draw (-3,-2)--(-2,-1)--(-1,-2);
            \draw (1,-2)--(2,-1)--(3,-2);
            \draw (2,-1)--(2,-2);
            
            \node at (0,1) {$T$};
            \node at (3.5,1) {$\longleftrightarrow$};
            \node at (7,1) {$U$};
            
            \node[wpoint] at (0,0) {$a$};
            \node[blpoint] at (-2,-1) {$b$};
            \node[blpoint] at (2,-1) {$e$};
            \node[blpoint] at (-3,-2) {$c$};
            \node[wpoint] at (-1,-2) {$d$};
            \node[blpoint] at (3,-2) {$h$};
            \node[wpoint] at (1,-2) {$f$};
            \node[blpoint] at (2,-2) {$g$};
            
            \draw (6,-1)--(7,0)--(8,-1)--(8,-2)--(8,-3)--(8,-4);
            \draw (7,-3)--(7,-2)--(6,-1)--(6,-2)--(6,-3)--(6,-4);
            \draw (8,-3)--(9,-4)--(9,-5);
            \draw (8,-3)--(10,-4)--(10,-5);
            
            \node[wpoint] at (7,0) {$a_2$};
            \node[wpoint] at (6,-1) {$b_1$};
            \node[wpoint] at (6,-2) {$b_2$};
            \node[wpoint] at (6,-3) {$d_1$};
            \node[wpoint] at (6,-4) {$d_2$};
            \node[wpoint] at (7,-2) {$c_1$};
            \node[wpoint] at (7,-3) {$c_2$};
            \node[wpoint] at (8,-1) {$e_1$};
            \node[wpoint] at (8,-2) {$e_2$};
            \node[wpoint] at (8,-3) {$f_1$};
            \node[wpoint] at (8,-4) {$f_2$};
            \node[wpoint] at (9,-4) {$g_1$};
            \node[wpoint] at (9,-5) {$g_2$};
            \node[wpoint] at (10,-4) {$h_1$};
            \node[wpoint] at (10,-5) {$h_2$};
            
        \end{tikzpicture}
        \caption{Bijection between given colored tree $T$ with $8$ vertices 
        and tree $U$ with $15$ vertices. Vertex labels in $T$ are in preorder, which is the order they are added when constructing $U$. $U$ contains four downward paths ($b_1,c_1,c_2$; $a_2,e_1,e_2,f_1,f_2$; $f_1,g_1,g_2$; and $f_1,h_1,h_2$), all of even length.}\label{fig:root2bij}
    \end{figure}
    
    We must show that, for any $T\in\mathcal{T}_n^A$, $\tau(T) = U\in\mathcal{E}_{2n}$ if the root of $T$ is blue, and $U\in\mathcal{E}_{2n-1}$ if the root of $T$ is white. It is clear that $U$ has $2n$ vertices in the first case and $2n-1$ vertices in the second, as the root of $T$ corresponds to two vertices or one vertex in $U$ respectively, while each non-root vertex of $T$ corresponds to a pair of vertices of $U$. It remains to show that 
    all of the downward paths in $U$ have even length.
    We show this inductively. When $n=1$, the condition is vacuously satisfied. Now, suppose that $U'\in\mathcal{E}_{2n-2}\cup\mathcal{E}_{2n-3}$. First, if the rightmost leaf $v$ of $T$ is blue, then the downward paths in $U$ are precisely the downward paths in $U'$ along with one additional downward path of length $2$ consisting of the two vertices $v_1$ and $v_2$ corresponding to $v$ along with their parent. On the other hand, if $v$ is white, then the downward paths of $U$ are precisely the downward paths of $U'$, except that the rightmost downward path has been extended by $2$ vertices. In either case, all of the downward paths in $U$ have even length, as required.

    \item[Description of $\tau^{-1}$] To prove that $\tau$ is a bijection, we construct its inverse. For a given tree $U\in\mathcal{E}_{2n}\cup\mathcal{E}_{2n-1}$, we wish to reconstruct the colored tree $\tau^{-1}(U) = T\in\mathcal{T}_n^A$. Again, we do this construction recursively. We can count vertices to see which set $U$ lies in. If $n=1$, $U$ consists of a single vertex if it lies in $\mathcal{E}_{2n-1}$. We map this to the single-vertex tree $T$ colored white. Alternatively, $U$ consists of a pair of vertices if it lies in $\mathcal{E}_{2n}$. We map this to the single-vertex tree $T$ colored blue. Now, suppose $n>1$. 
    
    Let $v_2$ denote the rightmost leaf in $U$, and denote the parent of $v_2$ by $v_1$. Note that $v_2$ must be the only child of $v_1$, as otherwise, $v_1$ to $v_2$ would be a downward path of odd length. Let $U'$ be the tree obtained from $U$ be deleting $v_1$ and $v_2$. We claim that $U'\in\mathcal{E}_{2n-2}\cup\mathcal{E}_{2n-3}$. First, suppose that the parent of $v_1$ has $v_1$ as its only child. Note that $U'$ is a path if and only if $U$ is a path; these are the only trees for which $v_2$ is the leftmost leaf. Supposing that $U'$ is not a path, then the parent of $v_1$ is a leaf in $U'$ whose downward path has length two less than the downward path in $U$ to $v_2$. Since the latter path has even length, 
    so does the former, and the other downward paths are all identical. Otherwise, if the parent of $v_1$ has multiple children, then $U'$ has all the same leaves as $U$ aside from $v_2$. They all have the same downward paths in $U'$ as they did in $U$, and the downward path to $v_2$ has length $2$.

    Since we have $U'\in\mathcal{E}_{2n-2}\cup\mathcal{E}_{2n-3}$, we can let $T'\in\mathcal{T}_{n-1}^A$ denote the tree mapping to $U'$. The parent of $v_1$ is a vertex $y$ corresponding to some vertex $x$ of $T'$.
    Give $y$ the designation $x_1$ if it has a child that also corresponds to $x$; otherwise call it $x_2$. If such a child of $y$ exists, call that child $x_2$. Similarly, if $y$ was designated as $x_2$ and is not the root, call its parent $x_1$, noting that the parent of $y$ must also correspond to $x$ in this case.
    We construct $T$ by considering three possibilities:
    \begin{description}
        \item[The parent of $v_1$ is $x_2$] In this case, obtain $T$ from $T'$ by adding a new rightmost child of $x$ and coloring it the opposite color of $x$.
        \item[The parent of $v_1$ is $x_1$, and $x$ is blue] In this case, $T$ is obtained from $T'$ by adding a new child to $x$ immediately to the left of its leftmost child that is white (or as its rightmost child if no such child exists) and coloring it blue. See 
        Figure~\ref{fig:root1bijrev}, where $v_1=e_1$, $v_2=e_2$, $x_1=c_1$, $x_2=c_2$, and $x=c$.
        \item[The parent of $v_1$ is $x_1$, and $x$ is white] In this case, obtain $T$ from $T'$ by adding a new rightmost child of the \emph{parent} of $x$ and coloring it blue. See Figure~\ref{fig:root2bijrev}, where $v_1=h_1$, $v_2=h_2$, $x_1=f_1$, $x_2=f_2$, $x=f$, and the parent of $x$ is $e$.
    \end{description}
    These operations result in a 
    colored tree $T\in\mathcal{T}_n^A$, and they reverse the operations used to construct $U$, as required for this to be the inverse mapping.
    
    \begin{figure}[htbp]
        \begin{tikzpicture}
            \tikzset{blpoint/.style = {circle, fill=blue, draw=blue, inner sep=2pt, text=white}}
            \tikzset{wpoint/.style = {circle, fill=white, draw=black, inner sep=2pt}}
            \draw (5,-1)--(7,0)--(9,-1);
            \draw (4,-2)--(5,-1)--(6,-2);
            
            \node at (7,1) {$T$};
            \node at (3.5,1) {$\longleftrightarrow$};
            \node at (0,1) {$U$};
            
            \node[blpoint] at (7,0) {$a$};
            \node[blpoint] at (5,-1) {$c$};
            \node[wpoint] at (9,-1) {$b$};
            \node[blpoint] at (4,-2) {$e$};
            \node[wpoint] at (6,-2) {$d$};
            
            \draw (0,0)--(0,-1)--(0,-2)--(0,-3);
            \draw (0,0)--(1,-1)--(1,-2)--(1,-3)--(1,-4);
            \draw (1,-1)--(2,-2)--(2,-3);
            
            \node[wpoint] at (0,0) {$a_1$};
            \node[wpoint] at (0,-1) {$a_2$};
            \node[wpoint] at (1,-1) {$c_1$};
            \node[wpoint] at (1,-2) {$c_2$};
            \node[wpoint] at (1,-3) {$d_1$};
            \node[wpoint] at (1,-4) {$d_2$};
            \node[wpoint] at (2,-2) {$e_1$};
            \node[wpoint] at (2,-3) {$e_2$};
            \node[wpoint] at (0,-2) {$b_1$};
            \node[wpoint] at (0,-3) {$b_2$};
            
        \end{tikzpicture}
        \caption{Reversed bijection between given tree $U$ with $10$ vertices having all downward paths of even length and colored tree $T$ with $5$ vertices. 
        Vertex pair labels in $U$ are in preorder, which is the order they are added when constructing $T$. Note that the second rule results in $e$ appearing to the left of $d$ and $c$ to the left of $b$ in $T$.}\label{fig:root1bijrev}
    \end{figure}
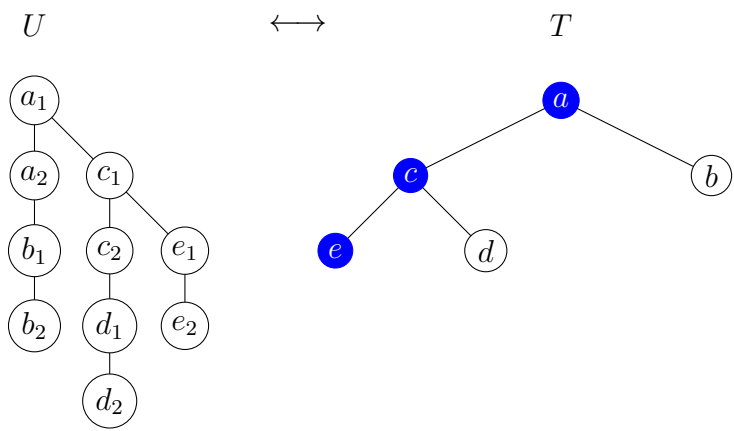

    \begin{figure}[htbp]
        \begin{tikzpicture}
            \tikzset{blpoint/.style = {circle, fill=blue, draw=blue, inner sep=2pt, text=white}}
            \tikzset{wpoint/.style = {circle, fill=white, draw=black, inner sep=2pt}}
            \draw (5,-1)--(7,0)--(9,-1);
            \draw (4,-2)--(5,-1)--(6,-2);
            \draw (8,-2)--(9,-1)--(10,-2);
            \draw (9,-1)--(9,-2);
            
            \node at (7,1) {$T$};
            \node at (3.5,1) {$\longleftrightarrow$};
            \node at (0,1) {$U$};
            
            \node[wpoint] at (7,0) {$a$};
            \node[blpoint] at (5,-1) {$b$};
            \node[blpoint] at (9,-1) {$e$};
            \node[blpoint] at (4,-2) {$d$};
            \node[wpoint] at (6,-2) {$c$};
            \node[blpoint] at (10,-2) {$h$};
            \node[wpoint] at (8,-2) {$f$};
            \node[blpoint] at (9,-2) {$g$};
            
            \draw (-1,-1)--(0,0)--(1,-1)--(1,-2)--(1,-3)--(1,-4);
            \draw (0,-3)--(0,-2)--(-1,-1)--(-1,-2)--(-1,-3)--(-1,-4);
            \draw (1,-3)--(2,-4)--(2,-5);
            \draw (1,-3)--(3,-4)--(3,-5);
            
            \node[wpoint] at (0,0) {$a_2$};
            \node[wpoint] at (-1,-1) {$b_1$};
            \node[wpoint] at (-1,-2) {$b_2$};
            \node[wpoint] at (-1,-3) {$c_1$};
            \node[wpoint] at (-1,-4) {$c_2$};
            \node[wpoint] at (0,-2) {$d_1$};
            \node[wpoint] at (0,-3) {$d_2$};
            \node[wpoint] at (1,-1) {$e_1$};
            \node[wpoint] at (1,-2) {$e_2$};
            \node[wpoint] at (1,-3) {$f_1$};
            \node[wpoint] at (1,-4) {$f_2$};
            \node[wpoint] at (2,-4) {$g_1$};
            \node[wpoint] at (2,-5) {$g_2$};
            \node[wpoint] at (3,-4) {$h_1$};
            \node[wpoint] at (3,-5) {$h_2$};
            
        \end{tikzpicture}
        \caption{Reversed bijection between given tree $U$ with $15$ vertices having all downward paths of even length and colored tree $T$ with $8$ vertices. 
        Vertex pair labels in $U$ are in preorder, which is the order they are added when constructing $T$. Note that the second rule results in $d$ appearing to the left of $c$ in $T$.}\label{fig:root2bijrev}
    \end{figure}
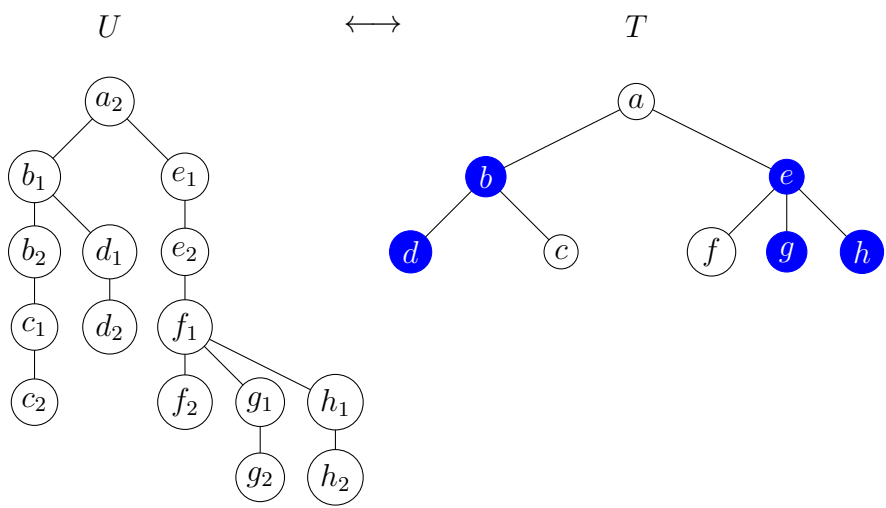

    See Figures~\ref{fig:root1bijrev} 
    and~\ref{fig:root2bijrev} for depictions of this direction of the bijection on the same tree structures as Figures~\ref{fig:root1bij}, 
     and~\ref{fig:root2bij}.
\end{description}

\section{Data on distinct $3\times3$ coloring rules}\label{app:3by3}
As mentioned in Subsection~\ref{ss:3by3}, there are $74$ strong equivalence classes of tree coloring rules with three colors that result in $72$ different counting sequences. This appendix contains $72$ entries, one for each tree coloring equivalence class. Each entry contains some of the following elements:
\begin{description}
    \item[Matrices] This is a listing of one or more coloring matrices resulting in the counting sequence for this entry. If multiple matrices are present, they are strictly tree coloring equivalent, 
    except where noted in Entries~15 and~29. Any matrix 
    isomorphic to a listed matrix would also fall within this entry, though the sequences with fixed root color would be correspondingly permuted.
    \item[Sequences] This is a table of the first eight terms of the counting sequence along with the first eight terms of the sequences where the root color is fixed. The entry in the last column of each row summarizes the eventual behavior of that sequence with a formula, which may be in terms of an existing OEIS sequence. 
    If the sequence's formula is too long to fit in the table, the entry is listed as $***$, and the formula appears below the table. Additionally, a superscript $h$ is used in the first column of the table whenever we have verified that that sequence is a nontrivial hypergeometric sequence, i.e., when it is not eventually constant and satisfies a first-order recurrence with polynomial coefficients. In all instances where the total number of trees for a given matrix is a hypergeometric sequence, each of the three sequences for the different colors is also hypergeometric (possibly trivial). Out of the 72 different sequences, 13 are found to be nontrivially hypergeometric. 
    \item[Nontrivial Functional Equations] As established by Corollary~\ref{cor:algebraicgf}, all of the sequences in the tables have algebraic generating functions. Via computation with a polynomial system elimination algorithm, we obtain polynomial equations satisfied by each of the generating functions with fixed root color. We list these polynomials here unless the sequence is given by $t_A^{(i)}(n)=a\cdot b^{n}$ for some real number $a$, some integer $b$, and sufficiently large $n$. We also omit these functional equations for the specific sequences given below, instead opting to list these frequent expressions here. Recall the notation for these sequences from Section~\ref{sec:small}.
    \begin{description}
        \item[$t^{(i)}_A(n)=C_{n-1}$] $\gfa{i}^2-\gfa{i}+x$
        \item[$t^{(i)}_A(n)=2^{n-1}C_{n-1}$] $2\gfa{i}^2-\gfa{i}+x$
        \item[$t^{(i)}_A(n)=3^{n-1}C_{n-1}$] $3\gfa{i}^2-\gfa{i}+x$
        \item[$t^{(i)}_A(n)=R_{n-1}$] $\gfa{i}^2+(x - 1)\gfa{i}+x$
        \item[$t^{(i)}_A(n)=r_{n-1}$] $2\gfa{i}^2+(-x - 1)\gfa{i}+x$
        \item[$t^{(i)}_A(n)=D_{n}$] $\gfa{i}^4-\gfa{i}^3+3x\gfa{i}^2-x\gfa{i}+x^2$
    \end{description}
    \item[Discussion] It is here where we prove that our formulas in the ``Sequences'' table are correct. For simple cases, where a formula follows from Theorem~\ref{prop:zerorows} or~\ref{prop:catrows} (possibly after permuting colors with Theorem~\ref{prop:permmat}) or is otherwise immediate, we may omit any discussion. When one matrix is present, we refer to it as $A$ in the discussion. When several matrices are present, we refer to them by $A_1,A_2,\ldots$ in the order they are listed.

    In our discussions, we do not address sequences with A-numbers greater than 394000, as these sequences were added to the OEIS because of this work.
\end{description}

\subsection*{Entry 1}
\begin{description}
    \item[Matrices]
    
    \[\begin{bmatrix}
0 & 0 & 0 \\
0 & 0 & 0 \\
0 & 0 & 0
\end{bmatrix}\]
    \item[Sequences]\hspace*{\fill}\vspace{-0.5cm}
    \begin{center}
        \[
        \begin{array}{|c||c|c|c|c|c|c|c|c||c|}\hline
           n  & 1 & 2 & 3 & 4 & 5 & 6 & 7 & 8  & \text{Formula}\\\hline\hline
           \text{Total}  & 3 & 0 & 0 & 0 & 0 & 0 & 0 & 0 & 0\\\hline\hline
           \text{Root Color }1 & 1 & 0 & 0 & 0 & 0 & 0 & 0 & 0 & 0\\\hline
           \text{Root Color }2 & 1 & 0 & 0 & 0 & 0 & 0 & 0 & 0 & 0\\\hline
           \text{Root Color }3 & 1 & 0 & 0 & 0 & 0 & 0 & 0 & 0 & 0\\\hline
        \end{array}
        \]
    \end{center}

\end{description}

\subsection*{Entry 2}
\begin{description}
    \item[Matrices]
    
    \[\begin{bmatrix}
0 & 1 & 0 \\
0 & 0 & 0 \\
0 & 0 & 0
\end{bmatrix}\]
    \item[Sequences]\hspace*{\fill}\vspace{-0.5cm}
    \begin{center}
        \[
        \begin{array}{|c||c|c|c|c|c|c|c|c||c|}\hline
           n  & 1 & 2 & 3 & 4 & 5 & 6 & 7 & 8  & \text{Formula}\\\hline\hline
           \text{Total}  & 3 & 1 & 1 & 1 & 1 & 1 & 1 & 1 & 1\\\hline\hline
           \text{Root Color }1 & 1 & 1 & 1 & 1 & 1 & 1 & 1 & 1 & 1\\\hline
           \text{Root Color }2 & 1 & 0 & 0 & 0 & 0 & 0 & 0 & 0 & 0\\\hline
           \text{Root Color }3 & 1 & 0 & 0 & 0 & 0 & 0 & 0 & 0 & 0\\\hline
        \end{array}
        \]
    \end{center}
    
\end{description}

\subsection*{Entry 3}
\begin{description}
    \item[Matrices]
    
    \[\begin{bmatrix}
0 & 0 & 0 \\
1 & 0 & 0 \\
1 & 0 & 0
\end{bmatrix}\]
    \item[Sequences]\hspace*{\fill}\vspace{-0.5cm}
    \begin{center}
        \[
        \begin{array}{|c||c|c|c|c|c|c|c|c||c|}\hline
           n  & 1 & 2 & 3 & 4 & 5 & 6 & 7 & 8  & \text{Formula}\\\hline\hline
           \text{Total}  & 3 & 2 & 2 & 2 & 2 & 2 & 2 & 2 & 2  \\\hline\hline
           \text{Root Color }1 & 1 & 0 & 0 & 0 & 0 & 0 & 0 & 0 & 0\\\hline
           \text{Root Color }2 & 1 & 1 & 1 & 1 & 1 & 1 & 1 & 1 & 1\\\hline
           \text{Root Color }3 & 1 & 1 & 1 & 1 & 1 & 1 & 1 & 1 & 1\\\hline
        \end{array}
        \]
    \end{center}

\end{description}

\subsection*{Entry 4}
\begin{description}
    \item[Matrices]
    
    \[\begin{bmatrix}
0 & 0 & 0 \\
1 & 0 & 0 \\
0 & 1 & 0
\end{bmatrix}\]
    \item[Sequences]\hspace*{\fill}\vspace{-0.5cm}
    \begin{center}
        \[
        \begin{array}{|c||c|c|c|c|c|c|c|c||c|}\hline
           n  & 1 & 2 & 3 & 4 & 5 & 6 & 7 & 8  & \text{Formula}\\\hline\hline
           \text{Total}  & 3 & 2 & 3 & 5 & 9 & 17 & 33 & 65 & 2^{n-2}+1\\\hline\hline
           \text{Root Color }1 & 1 & 0 & 0 & 0 & 0 & 0 & 0 & 0 & 0\\\hline
           \text{Root Color }2 & 1 & 1 & 1 & 1 & 1 & 1 & 1 & 1 & 1\\\hline
           \text{Root Color }3^h & 1 & 1 & 2 & 4 & 8 & 16 & 32 & 64 & 2^{n-2}\\\hline
        \end{array}
        \]
    \end{center}

    \item[Discussion] There is only one tree with $n$ vertices and root of color 2: the root needs to have $n-1$ children that are leaves of color 1. It remains to count the trees whose root is of color 3. Note that each of the children of such a root must be of color 2 and have a number of children (possibly $0$) of color 1 which have to be leaves. Thus, to every such tree, we have a corresponding composition $x_{1} + x_{2}+ \cdots $, where $x_i$ is the number of vertices in the subtree of the $i$-th child of the root (including the root). Conversely, every such composition determines uniquely one of the trees we want to count. The number of all compositions of $n-1$ into positive parts is $2^{n-2}$.
\end{description}

\subsection*{Entry 5}
\begin{description}
    \item[Matrices]
    
    \[\begin{bmatrix}
0 & 1 & 1 \\
0 & 0 & 0 \\
0 & 0 & 0
\end{bmatrix}\]
    \item[Sequences]\hspace*{\fill}\vspace{-0.5cm}
    \begin{center}
        \[
        \begin{array}{|c||c|c|c|c|c|c|c|c||c|}\hline
           n  & 1 & 2 & 3 & 4 & 5 & 6 & 7 & 8  & \text{Formula}\\\hline\hline
           \text{Total}^h  & 3 & 2 & 4 & 8 & 16 & 32 & 64 & 128 & 2^{n-1}\\\hline\hline
           \text{Root Color }1^h & 1 & 2 & 4 & 8 & 16 & 32 & 64 & 128 & 2^{n-1}\\\hline
           \text{Root Color }2 & 1 & 0 & 0 & 0 & 0 & 0 & 0 & 0 & 0\\\hline
           \text{Root Color }3 & 1 & 0 & 0 & 0 & 0 & 0 & 0 & 0 & 0\\\hline
        \end{array}
        \]
    \end{center}

    \item[Discussion] A root of color~1 must have $n-1$ children, each of which can be colored either with color 2 or color 3.
\end{description}

\subsection*{Entry 6}
\begin{description}
    \item[Matrices]
    
    \[\begin{bmatrix}
1 & 0 & 0 \\
0 & 0 & 0 \\
0 & 0 & 0
\end{bmatrix}\]
    \item[Sequences]\hspace*{\fill}\vspace{-0.5cm}
    \begin{center}
        \[
        \begin{array}{|c||c|c|c|c|c|c|c|c||c|}\hline
           n  & 1 & 2 & 3 & 4 & 5 & 6 & 7 & 8  & \text{Formula}\\\hline\hline
           \text{Total}^h  & 3 & 1 & 2 & 5 & 14 & 42 & 132 & 429 & C_{n-1}\\\hline\hline
           \text{Root Color }1^h & 1 & 1 & 2 & 5 & 14 & 42 & 132 & 429 & C_{n-1}\\\hline
           \text{Root Color }2 & 1 & 0 & 0 & 0 & 0 & 0 & 0 & 0 & 0\\\hline
           \text{Root Color }3 & 1 & 0 & 0 & 0 & 0 & 0 & 0 & 0 & 0\\\hline
        \end{array}
        \]
    \end{center}

\end{description}

\subsection*{Entry 7}
\begin{description}
    \item[Matrices]
    
    \[\begin{bmatrix}
1 & 0 & 0 \\
0 & 0 & 1 \\
0 & 0 & 0
\end{bmatrix}\]
    \item[Sequences]\hspace*{\fill}\vspace{-0.5cm}
    \begin{center}
        \[
        \begin{array}{|c||c|c|c|c|c|c|c|c||c|}\hline
           n  & 1 & 2 & 3 & 4 & 5 & 6 & 7 & 8  & \text{Formula}\\\hline\hline
           \text{Total}  & 3 & 2 & 3 & 6 & 15 & 43 & 133 & 430 & 1+C_{n-1}\\\hline\hline
           \text{Root Color }1^h & 1 & 1 & 2 & 5 & 14 & 42 & 132 & 429 & C_{n-1}\\\hline
           \text{Root Color }2 & 1 & 1 & 1 & 1 & 1 & 1 & 1 & 1 & 1\\\hline
           \text{Root Color }3 & 1 & 0 & 0 & 0 & 0 & 0 & 0 & 0 & 0\\\hline
        \end{array}
        \]
    \end{center}
    \item[Discussion] Apply Theorem~\ref{prop:blockdiag} with Entry~2 in Table~\ref{tb:2cols}, or swap colors $2$ and $3$ and apply Theorem~\ref{prop:singletoncolors}.

\end{description}

\subsection*{Entry 8}
\begin{description}
    \item[Matrices]
    
    \[\begin{bmatrix}
1 & 0 & 0 \\
0 & 1 & 0 \\
0 & 0 & 0
\end{bmatrix}, \begin{bmatrix}
1 & 0 & 0 \\
1 & 0 & 0 \\
0 & 0 & 0
\end{bmatrix}, \begin{bmatrix}
0 & 1 & 0 \\
1 & 0 & 0 \\
0 & 0 & 0
\end{bmatrix}\]
    \item[Sequences]\hspace*{\fill}\vspace{-0.5cm}
    \begin{center}
        \[
        \begin{array}{|c||c|c|c|c|c|c|c|c||c|}\hline
           n  & 1 & 2 & 3 & 4 & 5 & 6 & 7 & 8  & \text{Formula}\\\hline\hline
           \text{Total}^h  & 3 & 2 & 4 & 10 & 28 & 84 & 264 & 858 & 2C_{n-1}\\\hline\hline
           \text{Root Color }1^h & 1 & 1 & 2 & 5 & 14 & 42 & 132 & 429 & C_{n-1}\\\hline
           \text{Root Color }2^h & 1 & 1 & 2 & 5 & 14 & 42 & 132 & 429 & C_{n-1}\\\hline
           \text{Root Color }3 & 1 & 0 & 0 & 0 & 0 & 0 & 0 & 0 & 0\\\hline
        \end{array}
        \]
    \end{center}

    \item[Discussion] 
    Apply Theorem~\ref{prop:unusablecolors} to Entry~4 in Table~\ref{tb:2cols}.
\end{description}

\subsection*{Entry 9}
\begin{description}
    \item[Matrices]\hspace*{\fill}
    
    \makecell{$\begin{bmatrix}
1 & 0 & 0 \\
0 & 1 & 0 \\
0 & 0 & 1
\end{bmatrix}$, $\begin{bmatrix}
1 & 0 & 0 \\
0 & 1 & 0 \\
1 & 0 & 0
\end{bmatrix}$, $\begin{bmatrix}
1 & 0 & 0 \\
1 & 0 & 0 \\
1 & 0 & 0
\end{bmatrix}$, $\begin{bmatrix}
1 & 0 & 0 \\
1 & 0 & 0 \\
0 & 1 & 0
\end{bmatrix}$, \\ $\begin{bmatrix}
0 & 1 & 0 \\
0 & 0 & 1 \\
1 & 0 & 0
\end{bmatrix}$, $\begin{bmatrix}
1 & 0 & 0 \\
0 & 0 & 1 \\
0 & 1 & 0
\end{bmatrix}$, $\begin{bmatrix}
0 & 1 & 0 \\
1 & 0 & 0 \\
0 & 1 & 0
\end{bmatrix}$}
    \item[Sequences]\hspace*{\fill}\vspace{-0.5cm}
    \begin{center}
        \[
        \begin{array}{|c||c|c|c|c|c|c|c|c||c|}\hline
           n  & 1 & 2 & 3 & 4 & 5 & 6 & 7 & 8  & \text{Formula}\\\hline\hline
           \text{Total}^h  & 3 & 3 & 6 & 15 & 42 & 126 & 396 & 1287 & 3C_{n-1}\\\hline\hline
           \text{Root Color }1^h & 1 & 1 & 2 & 5 & 14 & 42 & 132 & 429 & C_{n-1}\\\hline
           \text{Root Color }2^h & 1 & 1 & 2 & 5 & 14 & 42 & 132 & 429 & C_{n-1}\\\hline
           \text{Root Color }3^h & 1 & 1 & 2 & 5 & 14 & 42 & 132 & 429 & C_{n-1}\\\hline
        \end{array}
        \]
    \end{center}

    \item[Discussion] This follows from Theorem~\ref{prop:row-sums}, since all row sums are equal to $1$.    
\end{description}

\subsection*{Entry 10}
\begin{description}
    \item[Matrices]
    
    \[\begin{bmatrix}
1 & 1 & 0 \\
0 & 0 & 0 \\
0 & 0 & 0
\end{bmatrix}\]
    \item[Sequences]\hspace*{\fill}\vspace{-0.5cm}
    \begin{center}
        \[
        \begin{array}{|c||c|c|c|c|c|c|c|c||c|}\hline
           n  & 1 & 2 & 3 & 4 & 5 & 6 & 7 & 8  & \text{Formula}\\\hline\hline
           \text{Total}  & 3 & 2 & 6 & 22 & 90 & 394 & 1806 & 8558 & R_{n-1}\\\hline\hline
           \text{Root Color }1 & 1 & 2 & 6 & 22 & 90 & 394 & 1806 & 8558 & R_{n-1}\\\hline
           \text{Root Color }2 & 1 & 0 & 0 & 0 & 0 & 0 & 0 & 0 & 0\\\hline
           \text{Root Color }3 & 1 & 0 & 0 & 0 & 0 & 0 & 0 & 0 & 0\\\hline
        \end{array}
        \]
    \end{center}

    \item[Discussion]
    Apply Theorem~\ref{prop:unusablecolors} to Entry~5 in Table~\ref{tb:2cols}.
\end{description}

\subsection*{Entry 11}
\begin{description}
    \item[Matrices]
    
    \[\begin{bmatrix}
1 & 0 & 0 \\
0 & 1 & 1 \\
0 & 0 & 0
\end{bmatrix}\]
    \item[Sequences]\hspace*{\fill}\vspace{-0.5cm}
    \begin{center}
        \[
        \begin{array}{|c||c|c|c|c|c|c|c|c||c|}\hline
           n  & 1 & 2 & 3 & 4 & 5 & 6 & 7 & 8  & \text{Formula}\\\hline\hline
           \text{Total}  & 3 & 3 & 8 & 27 & 104 & 436 & 1938 & 8987 & C_{n-1}+R_{n-1}.\\\hline\hline
           \text{Root Color }1^{h} & 1 & 1 & 2 & 5 & 14 & 42 & 132 & 429 & C_{n-1}\\\hline
           \text{Root Color }2 & 1 & 2 & 6 & 22 & 90 & 394 & 1806 & 8558 & R_{n-1}\\\hline
           \text{Root Color }3 & 1 & 0 & 0 & 0 & 0 & 0 & 0 & 0 & 0\\\hline
        \end{array}
        \]
    \end{center}

\item[Discussion] Apply Theorem~\ref{prop:blockdiag} with Entry~5 in Table~\ref{tb:2cols}. 
    
\end{description}

\subsection*{Entry 12}
\begin{description}
    \item[Matrices]
    
    \[\begin{bmatrix}
1 & 0 & 0 \\
1 & 1 & 0 \\
0 & 0 & 0
\end{bmatrix}\]
    \item[Sequences]\hspace*{\fill}\vspace{-0.5cm}
    \begin{center}
        \[
        \begin{array}{|c||c|c|c|c|c|c|c|c||c|}\hline
           n  & 1 & 2 & 3 & 4 & 5 & 6 & 7 & 8  & \text{Formula}\\\hline\hline
           \text{Total}  & 3 & 3 & 9 & 34 & 145 & 667 & 3231 & 16247 & C_{n-1} + D_{n}\\\hline\hline
           \text{Root Color }1^{h} & 1 & 1 & 2 & 5 & 14 & 42 & 132 & 429 & C_{n-1}\\\hline
           \text{Root Color }2 & 1 & 2 & 7 & 29 & 131 & 625 & 3099 & 15818 & D_n\\\hline
           \text{Root Color }3 & 1 & 0 & 0 & 0 & 0 & 0 & 0 & 0 & 0\\\hline
        \end{array}
        \]
    \end{center}

    \item[Discussion] 
    Apply Theorem~\ref{prop:unusablecolors} to Entry~6 in Table~\ref{tb:2cols}.
\end{description}

\subsection*{Entry 13}
\begin{description}
    \item[Matrices]
    
    \[\begin{bmatrix}
1 & 0 & 0 \\
0 & 1 & 0 \\
1 & 0 & 1
\end{bmatrix}, \begin{bmatrix}
1 & 0 & 0 \\
1 & 0 & 0 \\
1 & 0 & 1
\end{bmatrix}, \begin{bmatrix}
1 & 0 & 0 \\
1 & 0 & 0 \\
0 & 1 & 1
\end{bmatrix}, \begin{bmatrix}
0 & 1 & 0 \\
1 & 0 & 0 \\
0 & 1 & 1
\end{bmatrix}\]
    \item[Sequences]\hspace*{\fill}\vspace{-0.5cm}
    \begin{center}
        \[
        \begin{array}{|c||c|c|c|c|c|c|c|c||c|}\hline
           n  & 1 & 2 & 3 & 4 & 5 & 6 & 7 & 8  & \text{Formula}\\\hline\hline
           \text{Total}  & 3 & 4 & 11 & 39 & 159 & 709 & 3363 & 16676 & 2C_{n-1} + D_{n}\\\hline\hline
           \text{Root Color }1^{h} & 1 & 1 & 2 & 5 & 14 & 42 & 132 & 429 & C_{n-1}\\\hline
           \text{Root Color }2^{h} & 1 & 1 & 2 & 5 & 14 & 42 & 132 & 429 & C_{n-1}\\\hline
           \text{Root Color }3 & 1 & 2 & 7 & 29 & 131 & 625 & 3099 & 15818 & D_n\\\hline
        \end{array}
        \]
    \end{center}
    
    \item[Discussion] 
    Swap colors $2$ and $3$ in $A_1$, then apply Theorem~\ref{prop:blockdiag} with Entry~6 in Table~\ref{tb:2cols}. The equivalence of the four matrices follows from the fact that colors 1 and 2 are interchangeable.
\end{description}

\subsection*{Entry 14}
\begin{description}
    \item[Matrices]
    
    \[\begin{bmatrix}
0 & 0 & 0 \\
1 & 0 & 0 \\
1 & 0 & 1
\end{bmatrix}\]
    \item[Sequences]\hspace*{\fill}\vspace{-0.5cm}
    \begin{center}
        \[
        \begin{array}{|c||c|c|c|c|c|c|c|c||c|}\hline
           n  & 1 & 2 & 3 & 4 & 5 & 6 & 7 & 8  & \text{Formula}\\\hline\hline
           \text{Total}  & 3 & 3 & 7 & 23 & 91 & 395 & 1807 & 8559 & 1+R_{n-1}\\\hline\hline
           \text{Root Color }1 & 1 & 0 & 0 & 0 & 0 & 0 & 0 & 0 & 0\\\hline
           \text{Root Color }2 & 1 & 1 & 1 & 1 & 1 & 1 & 1 & 1 & 1\\\hline
           \text{Root Color }3 & 1 & 2 & 6 & 22 & 90 & 394 & 1806 & 8558 & R_{n-1}\\\hline
        \end{array}
        \]
    \end{center}

    \item[Discussion] 
    Permute colors so that color~$1$ becomes color~$2$, color $2$ becomes color~$3$, and color~$3$ becomes color~$1$. Then, apply Theorem~\ref{prop:singletoncolors} with all of the matrices in that theorem equal to the $1\times1$ matrix consisting of a single $1$, and refer to Entry~5 in Table~\ref{tb:2cols}.
\end{description}

\subsection*{Entry 15}
\begin{description}
    \item[Matrices]
    \[\begin{bmatrix}
    0 & 0 & 0 \\
    1 & 1 & 0 \\
    1 & 0 & 1
    \end{bmatrix}, \begin{bmatrix}
    0 & 0 & 0 \\
    1 & 0 & 1 \\
    1 & 1 & 0
    \end{bmatrix}, \begin{bmatrix}
    0 & 0 & 0 \\
    1 & 0 & 1 \\
    1 & 0 & 1
    \end{bmatrix}, \begin{bmatrix}
    0 & 1 & 1 \\
    1 & 0 & 0 \\
    1 & 0 & 0
    \end{bmatrix}\]
    \item[Sequences]\hspace*{\fill}
    \begin{description}
        \item[First Three Matrices]\hspace*{\fill}\vspace{-0.5cm}
        \begin{center}
            \[
            \begin{array}{|c||c|c|c|c|c|c|c|c||c|}\hline
               n  & 1 & 2 & 3 & 4 & 5 & 6 & 7 & 8  & \text{Formula}\\\hline\hline
               \text{Total}  & 3 & 4 & 12 & 44 & 180 & 788 & 3612 & 17116 & 2R_{n-1}\\\hline\hline
               \text{Root Color }1 & 1 & 0 & 0 & 0 & 0 & 0 & 0 & 0 & 0\\\hline
               \text{Root Color }2 & 1 & 2 & 6 & 22 & 90 & 394 & 1806 & 8558 & R_{n-1}\\\hline
               \text{Root Color }3 & 1 & 2 & 6 & 22 & 90 & 394 & 1806 & 8558 & R_{n-1}\\\hline
            \end{array}
            \]
        \end{center}
        \item[Last Matrix]\hspace*{\fill}\vspace{-0.5cm}
        \begin{center}
            \[
            \begin{array}{|c||c|c|c|c|c|c|c|c||c|}\hline
               n  & 1 & 2 & 3 & 4 & 5 & 6 & 7 & 8  & \text{Formula}\\\hline\hline
               \text{Total}  & 3 & 4 & 12 & 44 & 180 & 788 & 3612 & 17116 & 2R_{n-1}\\\hline\hline
               \text{Root Color }1 & 1 & 2 & 6 & 22 & 90 & 394 & 1806 & 8558 & R_{n-1}\\\hline
               \text{Root Color }2 & 1 & 1 & 3 & 11 & 45 & 197 & 903 & 4279 & r_{n-1}\\\hline
               \text{Root Color }3 & 1 & 1 & 3 & 11 & 45 & 197 & 903 & 4279 & r_{n-1}\\\hline
            \end{array}
            \]
        \end{center}
    \end{description}

    \item[Discussion]
    Swap colors $1$ and $2$ in $A_3$ and apply Theorem~\ref{prop:blockul} to obtain the sequence for root color~$2$ from Entry~5 in Table~\ref{tb:2cols}. The sequence for root color~$3$ follows from observing that colors~$2$ and~$3$ are interchangeable.

     Observe, though, that Theorem~\ref{prop:equalones} tells us that $A_4$ is not strongly tree coloring equivalent to $A_1$, $A_2$, and $A_3$.
    We give a bijection between the trees colored according to $A_{4}' = \begin{bmatrix}
    0 & 1 & 0 \\
    1 & 0 & 1 \\
    0 & 1 & 0
    \end{bmatrix}$, which is isomorphic to $A_4$ by swapping colors~$1$ and~$2$, and the trees colored according to $A_3$.

    Observe that a nontrivial tree colored according to $A_3$ has a root of color~$2$ or~$3$, leaves that can be of color~$1$ or~$3$, and vertices of color~$3$ everywhere else. On the other hand, according to the rule specified by $A_{4}'$, vertices of color~$1$ or $3$ can only be followed by vertices of color~$2$, and vertices of color~$2$ can be followed by vertices of either color~$1$ or~$3$. As a result, a tree colored according to $A_{4}'$ has one of the following structures:
    \begin{itemize}
        \item All vertices at even distance from the root have color~$2$, and all vertices an odd distance from the root have color~$1$ or~$3$.
        \item All vertices an odd distance from the root have color~$2$, and all vertices at even distance from the root have color~$1$ or~$3$.
    \end{itemize}
    
    Deutsch~\cite{Deutsch2000bijection} gives a bijection $f_e$ between trees of order $n$ and themselves, where a tree with $k$ leaves is mapped to a tree with $k$ vertices at even distance from the root. The sets of vertices in $\mathcal{T}_{n}$ at even and odd distance from the root, respectively, are equinumerous (see \cite[Corollary 4.3]{CLS2007Butterfly} and the involution $\phi$ there). Thus, a similar bijection $f_o$ exists mapping a tree with $n$ vertices and $k$ leaves to a tree with $n$ vertices, $k$ of which are at odd distance from the root. 
    
    We are now ready to describe a bijection $f$ between trees colored according to $A_{3}$ and trees colored according to $A_{4}'$. Let $T$ be a tree colored according to $A_3$, and let $T\in\mathcal{T}_n$. If $n=1$, let $f(T)=T$. Otherwise, define $f(T)$ as follows:
    \begin{itemize}
        \item If the root of $T$ has color~$3$, let $f(T) = f_e(T)$, where if $v$ is a leaf of $T$, the color of $v$ and its corresponding vertex in $f(T)$ are the same. (The other vertices of $f(T)$ are required to be colored with color~$2$.)
        \item If the root of $T$ has color~$2$, let $f(T) = f_o(T)$, where if $v$ is a leaf of $T$, the color of $v$ and its corresponding vertex in $f(T)$ are the same. (The other vertices of $f(T)$ are required to be colored with color~$2$.)
    \end{itemize}
    It is not difficult to check that this is a bijection, because one can explicitly describe its inverse.

\end{description}

\subsection*{Entry 16}
\begin{description}
    \item[Matrices]
    
    \[\begin{bmatrix}
1 & 0 & 0 \\
1 & 1 & 0 \\
1 & 0 & 1
\end{bmatrix}, \begin{bmatrix}
1 & 0 & 0 \\
1 & 0 & 1 \\
1 & 0 & 1
\end{bmatrix}, \begin{bmatrix}
1 & 0 & 0 \\
1 & 0 & 1 \\
1 & 1 & 0
\end{bmatrix}\]
    \item[Sequences]\hspace*{\fill}\vspace{-0.5cm}
    \begin{center}
        \[
        \begin{array}{|c||c|c|c|c|c|c|c|c||c|}\hline
           n  & 1 & 2 & 3 & 4 & 5 & 6 & 7 & 8  & \text{Formula}\\\hline\hline
           \text{Total}  & 3 & 5 & 16 & 63 & 276 & 1292 & 6330 & 32065 & C_{n-1}+2D_{n}\\\hline\hline
           \text{Root Color }1^{h} & 1 & 1 & 2 & 5 & 14 & 42 & 132 & 429 & C_{n-1}\\\hline
           \text{Root Color }2 & 1 & 2 & 7 & 29 & 131 & 625 & 3099 & 15818 & D_n\\\hline
           \text{Root Color }3 & 1 & 2 & 7 & 29 & 131 & 625 & 3099 & 15818 & D_n\\\hline
        \end{array}
        \]
    \end{center}

    \item[Discussion] Swap colors $1$ and $2$ in $A_3$ and apply Theorem~\ref{prop:blockul} to obtain the sequence for root color~$2$ from Entry~6 in Table~\ref{tb:2cols}. The sequence for root color~$3$ follows from observing that colors~$2$ and~$3$ are interchangeable.
\end{description}

\subsection*{Entry 17}
\begin{description}
    \item[Matrices]
    
    \[\begin{bmatrix}
0 & 0 & 0 \\
1 & 0 & 0 \\
0 & 1 & 1
\end{bmatrix}\]
    \item[Sequences]\hspace*{\fill}\vspace{-0.5cm}
    \begin{center}
        \[
        \begin{array}{|c||c|c|c|c|c|c|c|c||c|}\hline
           n  & 1 & 2 & 3 & 4 & 5 & 6 & 7 & 8  & \text{Formula}\\\hline\hline
           \text{Total}  & 3 & 3 & 8 & 29 & 123 & 566 & 2736 & 13683 & 1+A215973(n-1)\\\hline\hline
           \text{Root Color }1 & 1 & 0 & 0 & 0 & 0 & 0 & 0 & 0 & 0\\\hline
           \text{Root Color }2 & 1 & 1 & 1 & 1 & 1 & 1 & 1 & 1 & 1\\\hline
           \text{Root Color }3 & 1 & 2 & 7 & 28 & 122 & 565 & 2735 & 13682 & A215973(n-1)\\\hline
        \end{array}
        \]
    \end{center}
    
    \item[Nontrivial Functional Equations]\hspace*{\fill}

    \begin{itemize}
\item $(x - 1)\gfa{3}^2+(1 - 2x)\gfa{3}+x^2 - x$

\end{itemize}
    \item[Discussion] The coloring rule can be interpreted as follows: Color~1 can only be assigned to leaves. Color~2 can only be followed by Color 1, so subtrees rooted at a vertex of color~2 are necessarily stars whose leaves are all of color~1. Finally, color~3 can either be followed by color~2 or by itself. This means that a tree colored according to this rule has the following shape: all vertices have color~3, except for some stars consisting of a vertex of color~2 whose children (if any) are all leaves of color~1.

According to a paper of Kung and de Mier~\cite{Kung2013Catalan}, sequence \seqnum{A215973} has an interpretation in terms of
a variant of Dyck paths from $(0,0)$ to $(t,0)$, where in addition to the usual $(1,1)$ (up) and $(1,-1)$ (down) steps, we are also allowed a horizontal $(a,0)$ step for any positive integer $a$ (arbitrary length).
We now give a bijection between $A$-colored trees and these Dyck paths with $t=n-1$.
Our bijection is a variant of the standard glove bijection. 
The only exception occurs when we encounter a vertex of color~2: when a subtree consisting of a vertex of color~2 and $k$ children of color~1 is traversed, this corresponds to a 
horizontal step $(k+1,0)$.
It is easy to see that the resulting path has the required properties and that the procedure can be reversed to construct a tree from a path. Thus $t_A(n)$ is (for $n > 1$) equal to the number of these Dyck paths,
as required.
\end{description}

\subsection*{Entry 18}
\begin{description}
    \item[Matrices]
    
    \[\begin{bmatrix}
0 & 0 & 0 \\
1 & 1 & 0 \\
0 & 1 & 0
\end{bmatrix}\]
    \item[Sequences]\hspace*{\fill}\vspace{-0.5cm}
    \begin{center}
        \[
        \begin{array}{|c||c|c|c|c|c|c|c|c||c|}\hline
           n  & 1 & 2 & 3 & 4 & 5 & 6 & 7 & 8  & \text{Formula}\\\hline\hline
           \text{Total}  & 3 & 3 & 9 & 33 & 135 & 591 & 2709 & 12837 & 3r_{n-1}\\\hline\hline
           \text{Root Color }1 & 1 & 0 & 0 & 0 & 0 & 0 & 0 & 0 & 0\\\hline
           \text{Root Color }2 & 1 & 2 & 6 & 22 & 90 & 394 & 1806 & 8558 & R_{n-1}\\\hline
           \text{Root Color }3 & 1 & 1 & 3 & 11 & 45 & 197 & 903 & 4279 & r_{n-1}\\\hline
        \end{array}
        \]
    \end{center}
    
    \item[Discussion] First, observe that if color 3 is not used, then we have $R_n$ many trees (see Entry~5 in Table~1 and apply Theorem~\ref{prop:blockul}). Denote the set of them by $X_n$. Also, note that a vertex can have color 3, if and only if it is the root of the tree and if all of its children are have color 2. It suffices to show that these trees are counted by $\frac{1}{2}R_{n} = r_{n}$, the little Schr\"{o}der numbers. For each of these trees, if the root had color 2, but not color 3, we would get one of the colored trees in $X_n$. Thus, we must just show that the trees in $X_n$ without a leaf at depth 1 of color 1 are counted by $\frac{1}{2}R_{n} = r_{n}$. We do that using a correspondence between the trees in $X_n$ and the so-called Schr\"{o}der paths, which are paths from $(0,0)$ to $(2n,0)$ with steps $(1,1), (1,-1)$ and $(2,0)$ (called U, D and F steps, respectively), not going above the 
    $x$-axis. Indeed, let us traverse a tree in $X_n$ using depth-first search and write F when we see a leaf of color 1, U when we go away from the root and D when we get closer to the root. We obtain a Schr\"{o}der path and the correspondence is one-to-one, as we obviously have the reverse map. Furthermore, a leaf at depth 1 of color 1 corresponds to an F step along the $x$-axis. 
    It is known that the Schr\"{o}der paths with F steps along the $x$-axis 
    and those without 
    are equinumerous and thus the trees in $X_n$ without a leaf at depth 1 of color 1 are counted by $\frac{1}{2}R_{n} = r_{n}$.
\end{description}

\subsection*{Entry 19}
\begin{description}
    \item[Matrices]
    
    \[\begin{bmatrix}
0 & 1 & 1 \\
0 & 1 & 0 \\
0 & 0 & 0
\end{bmatrix}\]
    \item[Sequences]\hspace*{\fill}\vspace{-0.5cm}
    \begin{center}
        \[
        \begin{array}{|c||c|c|c|c|c|c|c|c||c|}\hline
           n  & 1 & 2 & 3 & 4 & 5 & 6 & 7 & 8  & \text{Formula}\\\hline\hline
           \text{Total}^{h}  & 3 & 3 & 7 & 19 & 56 & 174 & 561 & 1859 & C_{n-1} + C_{n}\\\hline\hline
           \text{Root Color }1^h & 1 & 2 & 5 & 14 & 42 & 132 & 429 & 1430 & C_n\\\hline
           \text{Root Color }2^h & 1 & 1 & 2 & 5 & 14 & 42 & 132 & 429 & C_{n-1}\\\hline
           \text{Root Color }3 & 1 & 0 & 0 & 0 & 0 & 0 & 0 & 0 & 0\\\hline
        \end{array}
        \]
    \end{center}
    
    \item[Nontrivial Functional Equations]\hspace*{\fill}

    \begin{itemize}
\item $x\gfa{1}^2+(2x - 1)\gfa{1}+x$

\end{itemize}
    \item[Discussion] 
    Swapping color~$1$ and~$2$ and applying Theorem~\ref{prop:blockul} yields that we have $C_{n-1}$ such trees with root color~$2$.
    It remains to show that when the root has color $1$, we have $C_{n}$ trees. 
    Trees on two or more vertices with root color~$1$ colored according to this matrix $A$ consist of a root of color~$1$, leaves at depth $1$ of any combination of colors~$2$ and~$3$, and vertices of color~$2$ everywhere else. So, showing that $t_A^{(1)}(n)=C_n$ amounts to proving that the number of ordered trees with $n$ vertices whose leaves at depth $1$ are colored either~$2$ or~$3$ is $C_{n}$. For a bijective proof, see Exercise~11 in \cite{stanley2015catalan}. Here, we give a proof with generating functions. The generating function for the number of ordered trees $T(x) = \sum\limits_{i=1}^{n}t_{n}x^{n}$ satisfies the equation $T(x) = \frac{x}{1-T(x)} = x(1+T(x)+T^{2}(x)+ \ldots )$. If the generating function for the number of ordered trees with $n$ vertices whose leaves at depth $1$ are colored with either color~$2$ or~$3$ is $S(x)$, then we have
\[
S(x) = x(1 + (x+T(x))+ (x+T(x))^{2} + \ldots) = \frac{x}{1-x-T(x)}. 
\]
Using that $T^{2}(x) = T(x)-x$, we can show that $\frac{x}{1-x-T(x)} = \frac{T(x)-x}{x}$,  which is the generating function for the shifted Catalan numbers.

This also follows from applying Theorem~\ref{prop:rootonly} to Entry~3 in Table~\ref{tb:2cols}.
\end{description}

\subsection*{Entry 20}
\begin{description}
    \item[Matrices]
    
    \[\begin{bmatrix}
0 & 0 & 0 \\
1 & 1 & 0 \\
0 & 1 & 1
\end{bmatrix}\]
    \item[Sequences]\hspace*{\fill}\vspace{-0.5cm}
    \begin{center}
        \[
        \begin{array}{|c||c|c|c|c|c|c|c|c||c|}\hline
           n  & 1 & 2 & 3 & 4 & 5 & 6 & 7 & 8  & \text{Formula}\\\hline\hline
           \text{Total}  & 3 & 4 & 14 & 60 & 286 & 1456 & 7754 & 42680 & R_{n-1}+A394113(n) \\\hline\hline
           \text{Root Color }1 & 1 & 0 & 0 & 0 & 0 & 0 & 0 & 0 & 0\\\hline
           \text{Root Color }2 & 1 & 2 & 6 & 22 & 90 & 394 & 1806 & 8558 & R_{n-1}\\\hline
           \text{Root Color }3 & 1 & 2 & 8 & 38 & 196 & 1062 & 5948 & 34122 & A394113(n) \\\hline
        \end{array}
        \]
    \end{center}
    
    \item[Nontrivial Functional Equations]\hspace*{\fill}

    \begin{itemize}
\item $\gfa{3}^4+(-x - 1)\gfa{3}^3+4x\gfa{3}^2+(-x^2 - x)\gfa{3}+x^2$

\end{itemize}
    \item[Discussion] Swapping colors~$1$ and~$2$ and applying Theorem~\ref{prop:blockul} yields that we have $R_{n-1}$ such trees with root color~$2$. 
    
\end{description}

\subsection*{Entry 21}
\begin{description}
    \item[Matrices]
    
    \[\begin{bmatrix}
1 & 0 & 0 \\
1 & 1 & 0 \\
0 & 1 & 0
\end{bmatrix}\]
    \item[Sequences]\hspace*{\fill}\vspace{-0.5cm}
    \begin{center}
        \[
        \begin{array}{|c||c|c|c|c|c|c|c|c||c|}\hline
           n  & 1 & 2 & 3 & 4 & 5 & 6 & 7 & 8  & \text{Formula}\\\hline\hline
           \text{Total}  & 3 & 4 & 12 & 46 & 199 & 926 & 4525 & 22902 & C_{n-1}+D_n+A125188(n-1)\\\hline\hline
           \text{Root Color }1^{h} & 1 & 1 & 2 & 5 & 14 & 42 & 132 & 429 & C_{n-1}\\\hline
           \text{Root Color }2 & 1 & 2 & 7 & 29 & 131 & 625 & 3099 & 15818 & D_n\\\hline
           \text{Root Color }3 & 1 & 1 & 3 & 12 & 54 & 259 & 1294 & 6655 & A125188(n-1)\\\hline
        \end{array}
        \]
    \end{center}
    
    \item[Nontrivial Functional Equations]\hspace*{\fill}

    \begin{itemize}
\item $(x + 2)\gfa{3}^4+(-5x - 1)\gfa{3}^3+(3x^2 + 3x)\gfa{3}^2-3x^2\gfa{3}+x^3$

\end{itemize}
    \item[Discussion] The sequences for root colors~$1$ and~$2$ follow from Theorem~\ref{prop:blockul} and Entry~6 in Table~\ref{tb:2cols}. We prove that the number of trees with root of color~$3$ colored according to $A$ is given by OEIS sequence $\text{A125188}(n-1)$.
    According to the first formula in the OEIS, this sequence (shifted appropriately) has generating function
\[
g(x)=\frac{1+xC(x)-\sqrt{1-xC(x)-5x}}{2\left(1+C(x)\right)},
\]
where
\[
C(x)=\frac{1-\sqrt{1-4x}}{2x}.
\]
Theorem~\ref{prop:rootonlysingle} tells us that
\[
\gfa{3}=\frac{x}{1-\gfa{2}}.
\]
Based on the OEIS entry \seqnum{A007852}, which is the sequence $D_n$, we have that \[
\gfa{2}=\frac{1-xC(x)-\sqrt{1-5x-xC(x)}}{2},
\]
where $C(x)$ is as above. We then have
\[
\gfa{3}=\frac{x}{1-\gfa{2}}=\frac{x}{1-\frac{1-xC(x)-\sqrt{1-5x-xC(x)}}{2}}.
\]
This can be shown by algebraic manipulation to equal the function $g(x)$ above, as required. 
\end{description}

\subsection*{Entry 22}
\begin{description}
    \item[Matrices]
    
    \[\begin{bmatrix}
1 & 0 & 0 \\
1 & 1 & 0 \\
0 & 1 & 1
\end{bmatrix}\]
    \item[Sequences]\hspace*{\fill}\vspace{-0.5cm}
    \begin{center}
        \[
        \begin{array}{|c||c|c|c|c|c|c|c|c||c|}\hline
           n  & 1 & 2 & 3 & 4 & 5 & 6 & 7 & 8  & \text{Formula}\\\hline\hline
           \text{Total}  & 3 & 5 & 17 & 73 & 353 & 1834 & 9996 & 56377 & C_{n-1}+D_n+A394114(n)\\\hline\hline
           \text{Root Color }1^{h} & 1 & 1 & 2 & 5 & 14 & 42 & 132 & 429 & C_{n-1}\\\hline
           \text{Root Color }2 & 1 & 2 & 7 & 29 & 131 & 625 & 3099 & 15818 & D_n\\\hline
           \text{Root Color }3 & 1 & 2 & 8 & 39 & 208 & 1167 & 6765 & 40130 & A394114(n)\\\hline
        \end{array}
        \]
    \end{center}
    
    \item[Nontrivial Functional Equations]\hspace*{\fill}

    \begin{itemize}
\item $\gfa{3}^8-3\gfa{3}^7+(7x + 3)\gfa{3}^6+(-14x - 1)\gfa{3}^5+(13x^2 + 8x)\gfa{3}^4$ $+(-14x^2 - x)\gfa{3}^3+(7x^3 + 3x^2)\gfa{3}^2-3x^3\gfa{3}+x^4$
\end{itemize}
        \item[Discussion] The sequences for root colors~$1$ and~$2$ follow from Theorem~\ref{prop:blockul} and Entry~6 in Table~\ref{tb:2cols} as for the previous entry.
\end{description}

\subsection*{Entry 23}
\begin{description}
    \item[Matrices]
    
    \[\begin{bmatrix}
0 & 0 & 1 \\
1 & 0 & 0 \\
0 & 1 & 1
\end{bmatrix}\]
    \item[Sequences]\hspace*{\fill}\vspace{-0.5cm}
    \begin{center}
        \[
        \begin{array}{|c||c|c|c|c|c|c|c|c||c|}\hline
           n  & 1 & 2 & 3 & 4 & 5 & 6 & 7 & 8  & \text{Formula}\\\hline\hline
           \text{Total}  & 3 & 4 & 12 & 47 & 209 & 1000 & 5020 & 26076 & *** \\\hline\hline
           \text{Root Color }1 & 1 & 1 & 3 & 12 & 54 & 260 & 1310 & 6821 & A006026(n)\\\hline
           \text{Root Color }2 & 1 & 1 & 2 & 6 & 23 & 101 & 480 & 2400 & A218225(n-2)\\\hline
           \text{Root Color }3 & 1 & 2 & 7 & 29 & 132 & 639 & 3230 & 16855 & A394115(n)\\\hline
        \end{array}
        \]
        $***$: Formula for Total is $t_A(n)=A006026(n) + A218225(n-2) + A394115(n)$.
    \end{center}
    
    \item[Nontrivial Functional Equations]\hspace*{\fill}

    \begin{itemize}
\item $\gfa{1}^3-3\gfa{1}^2+(2x + 1)\gfa{1}-x$

\item $(1 - x)\gfa{2}^3+(2x^2 - 2x)\gfa{2}^2+x^3$

\item $\gfa{3}^3+(2x - 2)\gfa{3}^2+(1 - x)\gfa{3}+x^2 - x$

\end{itemize}
    \item[Discussion] OEIS sequence $\text{A006026}(n)$ has a generating function $G(x)$ that satisfies the equation $G(x)^3-3G(x)^2+\left(2x+1\right)G(x)-x$, which is precisely the functional equation listed for $\gfa{1}$. 
    Given that the initial terms agree, these sequences match.

We now prove that the number of trees with root of color~$2$ colored according to $A$ is given by OEIS sequence $\text{A218225}(n-2)$, with value $1$ when $n=1$. According to the sequence's definition, this sequence has a generating function $H(x)$ satisfying the equation
\[
\frac{1 - xH(x)}{(1 - (xH(x))^2)^2} = 1 - x.
\]
We claim that $\gfa{2}=x^2H(x)+x$. This would be equivalent to $xH(x)=\frac{\gfa{2}}{x}-1$, or in turn
\[
\frac{1 - \left(\frac{\gfa{2}}{x}-1\right)}{\left(1 - \left(\frac{\gfa{2}}{x}-1\right)^2\right)^2} = 1 - x.
\]
Clearing denominators, expanding, and factoring gives a factor of $(1-x)\gfa{2}^3+(2x^2-2x)\gfa{2}^2+x^3$, which is precisely the functional equation listed for $\gfa{2}$.
\end{description}

\subsection*{Entry 24}
\begin{description}
    \item[Matrices]
    
    \[\begin{bmatrix}
0 & 0 & 1 \\
1 & 1 & 0 \\
0 & 1 & 1
\end{bmatrix}\]
    \item[Sequences]\hspace*{\fill}\vspace{-0.5cm}
    \begin{center}
        \[
        \begin{array}{|c||c|c|c|c|c|c|c|c||c|}\hline
           n  & 1 & 2 & 3 & 4 & 5 & 6 & 7 & 8  & \text{Formula}\\\hline\hline
           \text{Total}  & 3 & 5 & 18 & 82 & 419 & 2293 & 13138 & 77801 & ***\\\hline\hline
           \text{Root Color }1 & 1 & 1 & 3 & 13 & 66 & 364 & 2111 & 12664 & A394116(n)\\\hline
           \text{Root Color }2 & 1 & 2 & 7 & 30 & 144 & 744 & 4049 & 22906 & A394117(n) \\\hline
           \text{Root Color }3 & 1 & 2 & 8 & 39 & 209 & 1185 & 6978 & 42231 & A394118(n) \\\hline
        \end{array}
        \]
        $***$: Formula for Total is $t_A(n)=A394116(n) + A394117(n) + A394118(n)$.
    \end{center}
    
    \item[Nontrivial Functional Equations]\hspace*{\fill}

    \begin{itemize}
\item $\gfa{1}^5+(-2x - 3)\gfa{1}^4+(7x + 1)\gfa{1}^3+(-4x^2 - 3x)\gfa{1}^2+3x^2\gfa{1}$ $-x^3$

\item $\gfa{2}^5+(2x - 2)\gfa{2}^4+\gfa{2}^3+(4x^2 - 2x)\gfa{2}^2+x^3$

\item $\gfa{3}^5-2\gfa{3}^4+(4x + 1)\gfa{3}^3-5x\gfa{3}^2+(2x^2 + x)\gfa{3}-x^2$

\end{itemize}
    
\end{description}

\subsection*{Entry 25}
\begin{description}
    \item[Matrices]\hspace*{\fill}
    
    \makecell{$\begin{bmatrix}
1 & 0 & 1 \\
1 & 1 & 0 \\
0 & 1 & 1
\end{bmatrix}$, $\begin{bmatrix}
0 & 1 & 1 \\
0 & 1 & 1 \\
0 & 1 & 1
\end{bmatrix}$, $\begin{bmatrix}
1 & 1 & 0 \\
1 & 1 & 0 \\
0 & 1 & 1
\end{bmatrix}$, $\begin{bmatrix}
0 & 1 & 1 \\
1 & 1 & 0 \\
1 & 0 & 1
\end{bmatrix}$, \\$\begin{bmatrix}
1 & 1 & 0 \\
1 & 0 & 1 \\
1 & 0 & 1
\end{bmatrix}$, $\begin{bmatrix}
0 & 1 & 1 \\
1 & 0 & 1 \\
1 & 0 & 1
\end{bmatrix}$, $\begin{bmatrix}
0 & 1 & 1 \\
1 & 0 & 1 \\
1 & 1 & 0
\end{bmatrix}$}
    \item[Sequences]\hspace*{\fill}\vspace{-0.5cm}
    \begin{center}
        \[
        \begin{array}{|c||c|c|c|c|c|c|c|c||c|}\hline
           n  & 1 & 2 & 3 & 4 & 5 & 6 & 7 & 8  & \text{Formula}\\\hline\hline
           \text{Total}^{h}  & 3 & 6 & 24 & 120 & 672 & 4032 & 25344 & 164736 & 3\cdot 2^{n-1}C_{n-1}\\\hline\hline
           \text{Root Color }1^h & 1 & 2 & 8 & 40 & 224 & 1344 & 8448 & 54912 & 2^{n-1}C_{n-1}\\\hline
           \text{Root Color }2^h & 1 & 2 & 8 & 40 & 224 & 1344 & 8448 & 54912 & 2^{n-1}C_{n-1}\\\hline
           \text{Root Color }3^h & 1 & 2 & 8 & 40 & 224 & 1344 & 8448 & 54912 & 2^{n-1}C_{n-1}\\\hline
        \end{array}
        \]
    \end{center}

    \item[Discussion] This follows from Theorem~\ref{prop:row-sums}, since all row sums are equal to $2$.
\end{description}

\subsection*{Entry 26}
\begin{description}
    \item[Matrices]
    
    \[\begin{bmatrix}
0 & 1 & 1 \\
0 & 1 & 0 \\
0 & 0 & 1
\end{bmatrix}, \begin{bmatrix}
0 & 1 & 1 \\
0 & 0 & 1 \\
0 & 0 & 1
\end{bmatrix}, \begin{bmatrix}
0 & 1 & 1 \\
0 & 0 & 1 \\
0 & 1 & 0
\end{bmatrix}\]
    \item[Sequences]\hspace*{\fill}\vspace{-0.5cm}
    \begin{center}
        \[
        \begin{array}{|c||c|c|c|c|c|c|c|c||c|}\hline
           n  & 1 & 2 & 3 & 4 & 5 & 6 & 7 & 8  & \text{Formula}\\\hline\hline
           \text{Total}^{h}  & 3 & 4 & 10 & 30 & 98 & 336 & 1188 & 4290 & 2\cdot A189176(n-2)\\\hline\hline
           \text{Root Color }1^h & 1 & 2 & 6 & 20 & 70 & 252 & 924 & 3432 & \binom{2n-2}{n-1}\\\hline
           \text{Root Color }2^h & 1 & 1 & 2 & 5 & 14 & 42 & 132 & 429 & C_{n-1}\\\hline
           \text{Root Color }3^h & 1 & 1 & 2 & 5 & 14 & 42 & 132 & 429 & C_{n-1}\\\hline
        \end{array}
        \]
    \end{center}
    
    \item[Nontrivial Functional Equations]\hspace*{\fill}

    \begin{itemize}
\item $(4x - 1)\gfa{1}^2+x^2$

\end{itemize}
    \item[Discussion] 
Trees colored according to $A_1$ with root color~1 have free choice between colors 2 and 3 for the root's children, but all remaining vertices then have their colors determined. So, trees with root color~1 are in bijection with lattice paths from $(0,0)$ to $(n,n)$ by taking each subtree of color~2 as a Dyck path and each subtree of color~3 a ``negative'' Dyck path (below $y=0$). So, $t_A^{(1)}(n)=\binom{2n-2}{n-1}$, meaning
$\gfa{1}=\frac{x}{\sqrt{1-4x}}$. 
Furthermore, $\gfa{2}=\gfa{3}=\frac{1-\sqrt{1-4x}}{2}$ by Theorem~\ref{prop:catrows}.
So,
\[
\gfall=\frac{x}{\sqrt{1-4x}}+1-\sqrt{1-4x}=\frac{1-4x+(5x-1)\sqrt{1-4x}}{1-4x}.
\]
We now claim that $t_A(n)=2\cdot A189176(n-1)$ for $n\geq2$. Based on the given generating function on OEIS, $2\cdot A189176(n-1)$ has generating function
\[
G(x)=\frac{1-5x+4x^2-(1-5x)\sqrt{1-4x}}{1-4x}=F_A(x)-\frac{1}{2}.
\]
The claim immediately follows.
Alternatively, everything follows from applying Theorem~\ref{prop:rootonly} to Entry~4 in Table~\ref{tb:2cols}. The equivalence of the three matrices is due to the fact that colors 2 and 3 are interchangeable.
\end{description}

\subsection*{Entry 27}
\begin{description}
    \item[Matrices]
    
    \[\begin{bmatrix}
1 & 1 & 1 \\
0 & 0 & 0 \\
0 & 0 & 0
\end{bmatrix}\]
    \item[Sequences]\hspace*{\fill}\vspace{-0.5cm}
    \begin{center}
        \[
        \begin{array}{|c||c|c|c|c|c|c|c|c||c|}\hline
           n  & 1 & 2 & 3 & 4 & 5 & 6 & 7 & 8  & \text{Formula}\\\hline\hline
           \text{Total}  & 3 & 3 & 12 & 57 & 300 & 1686 & 9912 & 60213 & A047891(n)\\\hline\hline
           \text{Root Color }1 & 1 & 3 & 12 & 57 & 300 & 1686 & 9912 & 60213 & A047891(n)\\\hline
           \text{Root Color }2 & 1 & 0 & 0 & 0 & 0 & 0 & 0 & 0 & 0\\\hline
           \text{Root Color }3 & 1 & 0 & 0 & 0 & 0 & 0 & 0 & 0 & 0\\\hline
        \end{array}
        \]
    \end{center}
    
    \item[Nontrivial Functional Equations]\hspace*{\fill}

    \begin{itemize}
\item $\gfa{1}^2+(2x - 1)\gfa{1}+x$

\end{itemize}
    \item[Discussion]
This is the special case $\ell=1$, $m=2$ of Theorem~\ref{thm:gennarayana}: trees with root color 2 or 3 exist only for $n=1$, while the number of trees with root color 1, and thus the total number of trees, colored according to $A$ is
$$\sum_{k=1}^{n-1} N_{n-1,k} 3^k$$
for $n > 1$, which is exactly OEIS sequence \seqnum{A047891}.
\end{description}

\subsection*{Entry 28}
\begin{description}
    \item[Matrices]
    
    \[\begin{bmatrix}
1 & 1 & 1 \\
0 & 1 & 0 \\
0 & 0 & 0
\end{bmatrix}\]
    \item[Sequences]\hspace*{\fill}\vspace{-0.5cm}
    \begin{center}
        \[
        \begin{array}{|c||c|c|c|c|c|c|c|c||c|}\hline
           n  & 1 & 2 & 3 & 4 & 5 & 6 & 7 & 8  & \text{Formula}\\\hline\hline
           \text{Total}  & 3 & 4 & 15 & 71 & 380 & 2192 & 13291 & 83489 & A394119(n)+C_{n-1}\\\hline\hline
           \text{Root Color }1 & 1 & 3 & 13 & 66 & 366 & 2150 & 13159 & 83060 & A394119(n)\\\hline
           \text{Root Color }2^h & 1 & 1 & 2 & 5 & 14 & 42 & 132 & 429 & C_{n-1}\\\hline
           \text{Root Color }3 & 1 & 0 & 0 & 0 & 0 & 0 & 0 & 0 & 0\\\hline
        \end{array}
        \]
    \end{center}
    
    \item[Nontrivial Functional Equations]\hspace*{\fill}

    \begin{itemize}
\item $\gfa{1}^4+(2x - 1)\gfa{1}^3+(x^2 + 2x)\gfa{1}^2+(2x^2 - x)\gfa{1}+x^2$

\end{itemize}
    
\end{description}

\subsection*{Entry 29}
\begin{description}
    \item[Matrices]
    \[\begin{bmatrix}
1 & 1 & 1 \\
0 & 1 & 0 \\
0 & 0 & 1
\end{bmatrix}, \begin{bmatrix}
1 & 1 & 1 \\
0 & 1 & 0 \\
0 & 1 & 0
\end{bmatrix}, \begin{bmatrix}
1 & 1 & 1 \\
0 & 0 & 1 \\
0 & 1 & 0
\end{bmatrix}, \begin{bmatrix}
1 & 1 & 0 \\
1 & 1 & 0 \\
0 & 0 & 1
\end{bmatrix}\]
    \item[Sequences]\hspace*{\fill}
    \begin{description}
        \item[First Three Matrices]\hspace*{\fill}\vspace{-0.5cm}
        \begin{center}
            \[
            \begin{array}{|c||c|c|c|c|c|c|c|c||c|}\hline
               n  & 1 & 2 & 3 & 4 & 5 & 6 & 7 & 8  & \text{Formula}\\\hline\hline
               \text{Total}  & 3 & 5 & 18 & 85 & 462 & 2730 & 17028 & 110253 & (2^{n}+1)C_{n-1}\\\hline\hline
               \text{Root Color }1 & 1 & 3 & 14 & 75 & 434 & 2646 & 16764 & 109395 & (2^{n}-1)C_{n-1}\\\hline
               \text{Root Color }2^{h} & 1 & 1 & 2 & 5 & 14 & 42 & 132 & 429 & C_{n-1}\\\hline
               \text{Root Color }3^{h} & 1 & 1 & 2 & 5 & 14 & 42 & 132 & 429 & C_{n-1}\\\hline
            \end{array}
            \]
        \end{center}
        \item[Last Matrix]\hspace*{\fill}\vspace{-0.5cm}
        \begin{center}
            \[
            \begin{array}{|c||c|c|c|c|c|c|c|c||c|}\hline
               n  & 1 & 2 & 3 & 4 & 5 & 6 & 7 & 8  & \text{Formula}\\\hline\hline
               \text{Total}  & 3 & 5 & 18 & 85 & 462 & 2730 & 17028 & 110253 & (2^{n}+1)C_{n-1}\\\hline\hline
               \text{Root Color }1^{h} & 1 & 2 & 8 & 40 & 224 & 1344 & 8448 & 54912 & 2^{n-1}C_{n-1}\\\hline
               \text{Root Color }2^{h} & 1 & 2 & 8 & 40 & 224 & 1344 & 8448 & 54912 & 2^{n-1}C_{n-1}\\\hline
               \text{Root Color }3^{h} & 1 & 1 & 2 & 5 & 14 & 42 & 132 & 429 & C_{n-1}\\\hline
            \end{array}
            \]
        \end{center}
    \end{description}
    
    \item[Nontrivial Functional Equations]\hspace*{\fill}
    \begin{description}
        \item[First Three Matrices]\hspace*{\fill}

        \begin{itemize}
\item $\gfa{1}^4+(6x - 1)\gfa{1}^2+x^2$

\end{itemize}

    \end{description}
    \item[Discussion] The trees colored according to $A_4$ with root of color 3 are counted by $C_{n-1}$, while for the rest of the trees each of the vertices has either color $1$ or $2$. This gives the formula $(2^{n}+1)C_{n-1}$. This formula can also be obtained from Theorem~\ref{prop:blockdiag}. Theorem~\ref{prop:upsetonly} applied to Entry~4 of Table~\ref{tb:2cols} gives the same formula (via its generating function) for coloring according to the first three matrices. (This is based on $A_1$ with colors~$1$ and~$3$ swapped.)

    We describe another, partially bijective, proof that $A_1$ has counting sequence $(2^{n}+1)C_{n-1}$, utilizing the butterfly decomposition of plane trees \cite[Section 2]{CLS2007Butterfly}. One can still try to obtain a purely bijective proof of the fact that, for example, $A_1$ and $A_4$ have the same counting sequences. Note that the equivalence of $A_1$, $A_2$, and $A_3$ follows directly from Corollary~\ref{cor:EqivCols}. Note also that \cite[Section 5]{CLS2007Butterfly} uses the same butterfly decomposition to interpret the generating function for the total number of chains in a tree with $n$ edges, and below we are showing that this decomposition can be used to find the generating function of all antichains in trees with $n$ vertices.

Consider the trees colored according to $A_1$ and let us denote their generating function by $B$. It suffices to show that the number of those with at least one vertex not having color 1 (i.e., in color 2 or color 3) is $2^{n}C_{n-1}$. Note that the set of vertices not colored 1 forms an antichain in the tree, as defined in \cite{Klazar1997Twelve}. Consider the path from the root to the first vertex with color 2 or color 3, when traversing the tree with DFS. The butterfly decomposition for this path gives us the following equation with similar derivation to \cite[Equation~2.2]{CLS2007Butterfly}:
\begin{equation*}
    B = 2C + (2C)(xC(B-C)) + (2C)(xC(B-C))^{2} + \ldots
\end{equation*} 
This is true since each butterfly on the considered path has generating function $xC((B+C)-2C) = xC(B-C)$, as we know that both subtrees in it have a common root of color 1, and the left subtree must have all of its vertices of color 1, while the right subtree of the butterfly can be any of the trees with root color 1. The generating function for the trees for which this path is of length $i$, where $i\geq 0$, is $(2C)(xC(B-C))^{i}$. Now, if $F=B-C$, then we obtain
\begin{equation*}
    F = C + (2C)(xCF) + (2C)(xCF)^{2} + \ldots = C+ \frac{2xC^{2}F}{1-xCF}.
\end{equation*}
After simplifying we get $F= \frac{C+xC^{2}F}{1-xCF}$ or 
\begin{equation*}
    xCF^{2} + (xC^{2}-1)F + C = 0.
\end{equation*}
Thus, one of the roots is  
\begin{equation*}
    F = \frac{(1-xC^{2})-\sqrt{(xC^{2}-1)^{2}-4xC^{2}}}{2xC} = \frac{(1-xC^{2})-\sqrt{(xC^{2}+1)^{2}-8xC^{2}}}{2xC}.
\end{equation*}
Using that $1+xC^{2}= C$, we see that 
\begin{equation*}
    B = F+C = \frac{(1+xC^{2})-\sqrt{(xC^{2}+1)^{2}-8xC^{2}}}{2xC} = \frac{1-\sqrt{1-8x}}{2x} = 2C(2x).
\end{equation*}
The sequence with generating function $2C(2x)$ is $2^{n}C_{n-1}$.

\end{description}

\subsection*{Entry 30}
\begin{description}
    \item[Matrices]
    
    \[\begin{bmatrix}
0 & 0 & 0 \\
1 & 0 & 1 \\
1 & 0 & 0
\end{bmatrix}\]
    \item[Sequences]\hspace*{\fill}\vspace{-0.5cm}
    \begin{center}
        \[
        \begin{array}{|c||c|c|c|c|c|c|c|c||c|}\hline
           n  & 1 & 2 & 3 & 4 & 5 & 6 & 7 & 8  & \text{Formula}\\\hline\hline
           \text{Total}  & 3 & 3 & 6 & 14 & 35 & 90 & 234 & 611 & 
           1+F_{2n-1} \\\hline\hline
           \text{Root Color }1 & 1 & 0 & 0 & 0 & 0 & 0 & 0 & 0 & 0\\\hline
           \text{Root Color }2 & 1 & 2 & 5 & 13 & 34 & 89 & 233 & 610 & F_{2n-1}\\\hline
           \text{Root Color }3 & 1 & 1 & 1 & 1 & 1 & 1 & 1 & 1 & 1\\\hline
        \end{array}
        \]
    \end{center}
    
    \item[Nontrivial Functional Equations]\hspace*{\fill}

    \begin{itemize}
\item $(x^2 - 3x + 1)\gfa{2}+x^2 - x$

\end{itemize}
    \item[Discussion] This rule can only color trees of height at most $2$. For $n\geq2$, there are $1+2^{n-1}$ ways to color the tree of height $1$ with $n$ vertices: One way to color it with a root of color~3 (all leaves have color~1) and $2^{n-1}$ ways with a root of color~2 (where the leaves can be either of colors~1 or~3). For a tree of height~2, there are $2^{\ell}$ ways to color it, where $\ell$ is the number of leaves that are children of the root. Any coloring of such a tree must have root of color~2, intermediate vertices of color~3, and leaves at depth~2 of color~1. The leaves of depth~1 can be either color~1 or color~3. Another way of summarizing this is, forcing the root to have color~2, that there are $2^\ell$ ways to color a tree of height at most $2$ that has $\ell$ leaves of depth~1.

There are $2^{n-2}$ trees of height at most $2$ on $n$ vertices. Here is a one-to-two correspondence between trees on $n-1$ vertices and trees on $n$ vertices. Given a tree on $n-1$ vertices, either give the root a new rightmost child or give the root's rightmost child a new rightmost child. Every tree is generated by exactly one of these transformations, as the rightmost leaf is either a child of the root or not. Clearly, this map can be inverted, so it is, in fact, a one-to-two mapping.

Now, consider applying the above mapping to a tree $T$ on $n-1$ vertices colored with root of color~2 according to this color rule. Let $T_1$ be the tree obtained by giving the root a new rightmost child, and let $T_2$ be the tree obtained by giving the root's rightmost child a new rightmost child. If $T$ has $c$ colorings with root colored by color~2, then $T_1$ has $2c$ such colorings, as the coloring of $T$ transfers to $T_1$, and the new leaf of depth~1 can be either color~1 or color~3. Now, in $T_2$ the new leaf is of depth~2, so any valid coloring assigns it color~1. Consider the coloring of $T_2$ consisting of the coloring from $T$ transferred over where the new leaf is given color~1. This is a valid coloring if and only if the rightmost child of the root was colored with color~3. If that vertex is not a leaf in $T$, all valid colorings of $T$ assign it color~3. If that vertex is a leaf in $T$, half of the valid colorings of $T$ assign it color~3 (and the other half assign it color~1). If that vertex is a leaf in $T$, $T$ was obtained from a tree $T'$ with $n-2$ vertices by giving the root a new rightmost child. So, the number of colorings of $T$ is twice that of $T'$. The number of colorings of $T_2$ is the same as the number of colorings of $T'$, which also equals the number of colorings of $T$ minus the number of colorings of $T'$.

Let $E_n$ denote the total number of colorings of trees with $n$ vertices where the root is assigned color~2. We have
\begin{align*}
E_n&=\sum_{T''\text{ has $n$ vertices}}\text{colorings}(T'')\\
&=\sum_{T\text{ has $n-1$ vertices}}(\text{colorings}(T_1)+\text{colorings}(T_2))\\
&=\sum_{T\text{ has $n-1$ vertices}}(2\cdot\text{colorings}(T)+\text{colorings}(T))-\sum_{T'\text{ has $n-2$ vertices}}\text{colorings}(T')\\
&=3E_{n-1}-E_{n-2}.
\end{align*}
The formula $E_n=F_{2n-1}$ follows, and then the formula for Entry~30 follows from this plus adding back the one tree on $n$ vertices whose root has color~3.

This entry also follows (non-bijectively) from Theorem~\ref{thm:genfib}.
\end{description}

\subsection*{Entry 31}
\begin{description}
    \item[Matrices]
    
    \[\begin{bmatrix}
0 & 0 & 0 \\
1 & 1 & 1 \\
1 & 0 & 0
\end{bmatrix}\]
    \item[Sequences]\hspace*{\fill}\vspace{-0.5cm}
    \begin{center}
        \[
        \begin{array}{|c||c|c|c|c|c|c|c|c||c|}\hline
           n  & 1 & 2 & 3 & 4 & 5 & 6 & 7 & 8  & \text{Formula}\\\hline\hline
           \text{Total}  & 3 & 4 & 14 & 66 & 356 & 2062 & 12502 & 78324 & 1+A155867(n-1)\\\hline\hline
           \text{Root Color }1 & 1 & 0 & 0 & 0 & 0 & 0 & 0 & 0 & 0\\\hline
           \text{Root Color }2 & 1 & 3 & 13 & 65 & 355 & 2061 & 12501 & 78323 & A155867(n-1)\\\hline
           \text{Root Color }3 & 1 & 1 & 1 & 1 & 1 & 1 & 1 & 1 & 1\\\hline
        \end{array}
        \]
    \end{center}
    
    \item[Nontrivial Functional Equations]\hspace*{\fill}

    \begin{itemize}
\item $(x - 1)\gfa{2}^2+(x^2 - 3x + 1)\gfa{2}+x^2 - x$

\end{itemize}
    \item[Discussion] This entry can be obtained from Entry~2 in Table~\ref{tb:2cols} by adding a row of ones at the bottom (and replacing with an equivalent matrix). So, Theorem~\ref{prop:upsetonly} applies. Letting $B=\begin{bmatrix}
        0&1\\0&0
    \end{bmatrix}$ be the matrix from that table entry, 
    we have $\gfll{B}=\frac{x}{1-x} + x$. So, by Theorem~\ref{prop:upsetonly} we have
\[
\gfall=\frac{1+\gfll{B}-\sqrt{(1-\gfll{B})^2-4x}}{2}.
\]
Substituting the expression in for $\gfll{B}$ and manipulating with a computer algebra system verifies that
\[
\gfall=x + \frac{x}{1-x}+ \frac{1 - 3x + x^2 - \sqrt{1 - 10x + 19x^2 - 10 x^3+x^4}}{2(1-x)}.
\]
The last term in this sum is the generating function for $A155867(n-1)$. The other two terms serve to add $2$ to the term at index $1$ and to add $1$ to all other terms, yielding the desired formula for Entry~31.
\end{description}

\subsection*{Entry 32}
\begin{description}
    \item[Matrices]
    
    \[\begin{bmatrix}
0 & 0 & 0 \\
1 & 1 & 1 \\
1 & 0 & 1
\end{bmatrix}\]
    \item[Sequences]\hspace*{\fill}\vspace{-0.5cm}
    \begin{center}
        \[
        \begin{array}{|c||c|c|c|c|c|c|c|c||c|}\hline
           n  & 1 & 2 & 3 & 4 & 5 & 6 & 7 & 8  & \text{Formula}\\\hline\hline
           \text{Total}  & 3 & 5 & 20 & 99 & 550 & 3286 & 20644 & 134557 & A394120(n)+R_{n-1}\\\hline\hline
           \text{Root Color }1 & 1 & 0 & 0 & 0 & 0 & 0 & 0 & 0 & 0\\\hline
           \text{Root Color }2 & 1 & 3 & 14 & 77 & 460 & 2892 & 18838 & 125999 & A394120(n) \\\hline
           \text{Root Color }3 & 1 & 2 & 6 & 22 & 90 & 394 & 1806 & 8558 & R_{n-1}\\\hline
        \end{array}
        \]
    \end{center}
    
    \item[Nontrivial Functional Equations]\hspace*{\fill}

    \begin{itemize}
\item $\gfa{2}^4+(x - 1)\gfa{2}^3+3x\gfa{2}^2+(x^2 - x)\gfa{2}+x^2$

\end{itemize}
    \item[Discussion] The listed formulas follow from swapping colors~$2$ and~$3$ and applying Theorem~\ref{prop:blockul}.
    
\end{description}

\subsection*{Entry 33}
\begin{description}
    \item[Matrices]
    
    \[\begin{bmatrix}
1 & 0 & 0 \\
1 & 1 & 0 \\
1 & 1 & 1
\end{bmatrix}\]
    \item[Sequences]\hspace*{\fill}\vspace{-0.5cm}
    \begin{center}
        \[
        \begin{array}{|c||c|c|c|c|c|c|c|c||c|}\hline
           n  & 1 & 2 & 3 & 4 & 5 & 6 & 7 & 8  & \text{Formula}\\\hline\hline
           \text{Total}  & 3 & 6 & 24 & 121 & 689 & 4233 & 27396 & 184151 & C_{n-1}+D_n+A394121(n)\\\hline\hline
           \text{Root Color }1^{h} & 1 & 1 & 2 & 5 & 14 & 42 & 132 & 429 & C_{n-1}\\\hline
           \text{Root Color }2 & 1 & 2 & 7 & 29 & 131 & 625 & 3099 & 15818 & D_n\\\hline
           \text{Root Color }3 & 1 & 3 & 15 & 87 & 544 & 3566 & 24165 & 167904 & A394121(n)\\\hline
        \end{array}
        \]
    \end{center}
    
    \item[Nontrivial Functional Equations]\hspace*{\fill}

    \begin{itemize}
\item $\gfa{3}^8-\gfa{3}^7+7x\gfa{3}^6-4x\gfa{3}^5+13x^2\gfa{3}^4-4x^2\gfa{3}^3+7x^3\gfa{3}^2-x^3\gfa{3}+x^4$

\end{itemize}
    \item[Discussion] This is the $3\times3$ case of Theorem~\ref{thm:gen33}. The listed formulas also follow from Theorem~\ref{prop:blockul}.
\end{description}

\subsection*{Entry 34}
\begin{description}
    \item[Matrices]
    
    \[\begin{bmatrix}
1 & 0 & 1 \\
1 & 0 & 1 \\
1 & 0 & 0
\end{bmatrix}, \begin{bmatrix}
1 & 0 & 1 \\
1 & 0 & 1 \\
0 & 1 & 0
\end{bmatrix},\begin{bmatrix}
1 & 0 & 1 \\
0 & 1 & 1 \\
1 & 0 & 0
\end{bmatrix}, \begin{bmatrix}
0 & 1 & 1 \\
1 & 0 & 1 \\
1 & 0 & 0
\end{bmatrix}\]
    \item[Sequences]\hspace*{\fill}\vspace{-0.5cm}
    
    \begin{center}
        \[
        \begin{array}{|c||c|c|c|c|c|c|c|c||c|}\hline
           n  & 1 & 2 & 3 & 4 & 5 & 6 & 7 & 8  & \text{Formula}\\\hline\hline
           \text{Total}^{h}  & 3 & 5 & 17 & 72 & 341 & 1729 & 9180 & 50388 & \frac{7n-4}{n(2n-1)}\binom{3n-3}{n-1}\\\hline\hline
           \text{Root Color }1^h & 1 & 2 & 7 & 30 & 143 & 728 & 3876 & 21318 & \frac{1}{n}\binom{3n-2}{n-1}\\\hline
           \text{Root Color }2^h & 1 & 2 & 7 & 30 & 143 & 728 & 3876 & 21318 & \frac{1}{n}\binom{3n-2}{n-1}\\\hline
           \text{Root Color }3^h & 1 & 1 & 3 & 12 & 55 & 273 & 1428 & 7752 & \frac{1}{2n-1}\binom{3n-3}{n-1}\\\hline
        \end{array}
        \]
    \end{center}
    \item[Nontrivial Functional Equations]\hspace*{\fill}

    \begin{itemize}
\item $\gfa{1}^3-2\gfa{1}^2+\gfa{1}-x$

\item $\gfa{2}^3-2\gfa{2}^2+\gfa{2}-x$

\item $\gfa{3}^3-x\gfa{3}+x^2$

\end{itemize}
    \item[Discussion] 
    Swapping colors~$2$ and~$3$ in $A_1$ and then applying Theorem~\ref{prop:blockul} while referring to Entry~7 in Table~\ref{tb:2cols} shows 
    that trees with root color~$1$ are enumerated by $\frac{1}{n}\binom{3n-2}{n-1}=\frac{3n-2}{n(2n-1)}\binom{3n-3}{n-1}$, while trees with root color~$3$ are enumerated by $\frac{1}{2n-1}\binom{3n-3}{n-1}$. Since colors~$1$ and~$2$ are interchangeable, trees with root color~$2$ are also enumerated by $\frac{3n-2}{n(2n-1)}\binom{3n-3}{n-1}$.
    Hence the total number of $n$-vertex trees colored according to any of the four matrices is given by $\frac{2}{n}\binom{3n-3}{n-1}+\frac{3n-2}{n(2n-1)}\binom{3n-3}{n-1}=\frac{7n-4}{n(2n-1)}\binom{3n-3}{n-1}$, as required.
\end{description}

\subsection*{Entry 35}
\begin{description}
    \item[Matrices]
    
    \[\begin{bmatrix}
1 & 1 & 1 \\
0 & 1 & 1 \\
0 & 1 & 0
\end{bmatrix}\]
    \item[Sequences]\hspace*{\fill}\vspace{-0.5cm}
    \begin{center}
        \[
        \begin{array}{|c||c|c|c|c|c|c|c|c||c|}\hline
           n  & 1 & 2 & 3 & 4 & 5 & 6 & 7 & 8  & \text{Formula}\\\hline\hline
           \text{Total}  & 3 & 6 & 25 & 130 & 757 & 4728 & 30988 & 210426 & \frac{2}{n}\binom{3n-3}{n-1}+A394122(n)\\\hline\hline
           \text{Root Color }1 & 1 & 3 & 15 & 88 & 559 & 3727 & 25684 & 181356 & A394122(n) \\\hline
           \text{Root Color }2^{h} & 1 & 2 & 7 & 30 & 143 & 728 & 3876 & 21318 & \frac{1}{n}\binom{3n-2}{n-1}\\\hline
           \text{Root Color }3^{h} & 1 & 1 & 3 & 12 & 55 & 273 & 1428 & 7752 & \frac{1}{2n-1}\binom{3n-3}{n-1}\\\hline
        \end{array}
        \]
    \end{center}
    
    \item[Nontrivial Functional Equations]\hspace*{\fill}

    \begin{itemize}
\item $\gfa{1}^6-\gfa{1}^5+5x\gfa{1}^4+(-x^2 - 2x)\gfa{1}^3+5x^2\gfa{1}^2-x^2\gfa{1}+x^3$

\item $\gfa{2}^3-2\gfa{2}^2+\gfa{2}-x$

\item $\gfa{3}^3-x\gfa{3}+x^2$

\end{itemize}
    \item[Description] Permuting colors so that color~$1$ becomes color~$3$, color~$2$ becomes color~$1$, and color~$3$ becomes color~$2$, and then applying Theorem~\ref{prop:blockul} with Entry~7 in Table~\ref{tb:2cols} 
    in mind yields the listed formulas.
    
\end{description}

\subsection*{Entry 36}
\begin{description}
    \item[Matrices]
    
    \[\begin{bmatrix}
1 & 1 & 1 \\
0 & 1 & 1 \\
0 & 1 & 1
\end{bmatrix}\]
    \item[Sequences]\hspace*{\fill}\vspace{-0.5cm}
    \begin{center}
        \[
        \begin{array}{|c||c|c|c|c|c|c|c|c||c|}\hline
           n  & 1 & 2 & 3 & 4 & 5 & 6 & 7 & 8  & \text{Formula}\\\hline\hline
           \text{Total}  & 3 & 7 & 32 & 181 & 1140 & 7666 & 53888 & 391197 & A394123(n)+2^nC_{n-1}\\\hline\hline
           \text{Root Color }1 & 1 & 3 & 16 & 101 & 692 & 4978 & 36992 & 281373 & A394123(n)\\\hline
           \text{Root Color }2^{h} & 1 & 2 & 8 & 40 & 224 & 1344 & 8448 & 54912 & 2^{n-1}C_{n-1}\\\hline
           \text{Root Color }3^{h} & 1 & 2 & 8 & 40 & 224 & 1344 & 8448 & 54912 & 2^{n-1}C_{n-1}\\\hline
        \end{array}
        \]
    \end{center}
    
    \item[Nontrivial Functional Equations]\hspace*{\fill}

    \begin{itemize}
\item $\gfa{1}^4-\gfa{1}^3+4x\gfa{1}^2-x\gfa{1}+x^2$

\end{itemize}
\item[Description] Swapping colors~$1$ and~$3$, and then applying Theorem~\ref{prop:blockul} with Entry~8 in Table~\ref{tb:2cols} in mind yields the listed formulas.
    
\end{description}

\subsection*{Entry 37}
\begin{description}
    \item[Matrices]
    
    \[\begin{bmatrix}
0 & 1 & 0 \\
1 & 1 & 0 \\
0 & 0 & 0
\end{bmatrix}\]
    \item[Sequences]\hspace*{\fill}\vspace{-0.5cm}
    \begin{center}
        \[
        \begin{array}{|c||c|c|c|c|c|c|c|c||c|}\hline
           n  & 1 & 2 & 3 & 4 & 5 & 6 & 7 & 8  & \text{Formula}\\\hline\hline
           \text{Total}^{h}  & 3 & 3 & 10 & 42 & 198 & 1001 & 5304 & 29070 & 
           \frac{2}{n}\binom{3n-3}{n-1}\\\hline\hline
           \text{Root Color }1^h & 1 & 1 & 3 & 12 & 55 & 273 & 1428 & 7752 & \frac{1}{2n-1}\binom{3n-3}{n-1}\\\hline
           \text{Root Color }2^h & 1 & 2 & 7 & 30 & 143 & 728 & 3876 & 21318 & \frac{1}{n}\binom{3n-2}{n-1}\\\hline
           \text{Root Color }3 & 1 & 0 & 0 & 0 & 0 & 0 & 0 & 0 & 0\\\hline
        \end{array}
        \]
    \end{center}
    
    \item[Nontrivial Functional Equations]\hspace*{\fill}

    \begin{itemize}
\item $\gfa{1}^3-x\gfa{1}+x^2$

\item $\gfa{2}^3-2\gfa{2}^2+\gfa{2}-x$

\end{itemize}

\item[Discussion] Apply Theorem~\ref{prop:unusablecolors} to Entry~7 in Table~\ref{tb:2cols}.
    
\end{description}

\subsection*{Entry 38}
\begin{description}
    \item[Matrices]
    
    \[\begin{bmatrix}
0 & 1 & 0 \\
1 & 1 & 0 \\
0 & 0 & 1
\end{bmatrix}\]
    \item[Sequences]\hspace*{\fill}\vspace{-0.5cm}
    \begin{center}
        \[
        \begin{array}{|c||c|c|c|c|c|c|c|c||c|}\hline
           n  & 1 & 2 & 3 & 4 & 5 & 6 & 7 & 8  & \text{Formula}\\\hline\hline
           \text{Total}  & 3 & 4 & 12 & 47 & 212 & 1043 & 5436 & 29499 & \frac{2}{n}\binom{3n-3}{n-1}+C_{n-1}\\\hline\hline
           \text{Root Color }1^{h} & 1 & 1 & 3 & 12 & 55 & 273 & 1428 & 7752 & \frac{1}{2n-1}\binom{3n-3}{n-1}\\\hline
           \text{Root Color }2^{h} & 1 & 2 & 7 & 30 & 143 & 728 & 3876 & 21318 & \frac{1}{n}\binom{3n-2}{n-1}\\\hline
           \text{Root Color }3^{h} & 1 & 1 & 2 & 5 & 14 & 42 & 132 & 429 & C_{n-1}\\\hline
        \end{array}
        \]
    \end{center}
    
    \item[Nontrivial Functional Equations]\hspace*{\fill}

    \begin{itemize}
\item $\gfa{1}^3-x\gfa{1}+x^2$

\item $\gfa{2}^3-2\gfa{2}^2+\gfa{2}-x$

\end{itemize}

\item[Discussion] This formula follows from Theorem~\ref{prop:blockdiag} and Entry~7 in Table~\ref{tb:2cols}.
    
\end{description}

\subsection*{Entry 39}
\begin{description}
    \item[Matrices]
    
    \[\begin{bmatrix}
1 & 1 & 0 \\
1 & 1 & 0 \\
0 & 0 & 0
\end{bmatrix}\]
    \item[Sequences]\hspace*{\fill}\vspace{-0.5cm}
    \begin{center}
        \[
        \begin{array}{|c||c|c|c|c|c|c|c|c||c|}\hline
           n  & 1 & 2 & 3 & 4 & 5 & 6 & 7 & 8  & \text{Formula}\\\hline\hline
           \text{Total}^{h}  & 3 & 4 & 16 & 80 & 448 & 2688 & 16896 & 109824 & 2^{n}C_{n-1}\\\hline\hline
           \text{Root Color }1^h & 1 & 2 & 8 & 40 & 224 & 1344 & 8448 & 54912 & 2^{n-1}C_{n-1}\\\hline
           \text{Root Color }2^h & 1 & 2 & 8 & 40 & 224 & 1344 & 8448 & 54912 & 2^{n-1}C_{n-1}\\\hline
           \text{Root Color }3 & 1 & 0 & 0 & 0 & 0 & 0 & 0 & 0 & 0\\\hline
        \end{array}
        \]
    \end{center}

\item[Discussion] Apply Theorem~\ref{prop:unusablecolors} to Entry~8 in Table~\ref{tb:2cols}.
    
\end{description}

\subsection*{Entry 40}
\begin{description}
    \item[Matrices]
    
    \[\begin{bmatrix}
0 & 1 & 0 \\
1 & 1 & 0 \\
0 & 1 & 0
\end{bmatrix}\]
    \item[Sequences]\hspace*{\fill}\vspace{-0.5cm}

    \begin{center}
        \[
        \begin{array}{|c||c|c|c|c|c|c|c|c||c|}\hline
           n  & 1 & 2 & 3 & 4 & 5 & 6 & 7 & 8  & \text{Formula}\\\hline\hline
           \text{Total}^{h}  & 3 & 4 & 13 & 54 & 253 & 1274 & 6732 & 36822 & \frac{5n-2}{n(2n-1)}\binom{3n-3}{n-1}\\\hline\hline
           \text{Root Color }1^h & 1 & 1 & 3 & 12 & 55 & 273 & 1428 & 7752 & \frac{1}{2n-1}\binom{3n-3}{n-1}\\\hline
           \text{Root Color }2^h & 1 & 2 & 7 & 30 & 143 & 728 & 3876 & 21318 & \frac{1}{n}\binom{3n-2}{n-1}\\\hline
           \text{Root Color }3^h & 1 & 1 & 3 & 12 & 55 & 273 & 1428 & 7752 & \frac{1}{2n-1}\binom{3n-3}{n-1}\\\hline
        \end{array}
        \]
    \end{center}
    
    \item[Nontrivial Functional Equations]\hspace*{\fill}

    \begin{itemize}
\item $\gfa{1}^3-x\gfa{1}+x^2$

\item $\gfa{2}^3-2\gfa{2}^2+\gfa{2}-x$

\item $\gfa{3}^3-x\gfa{3}+x^2$

\end{itemize}
    \item[Discussion] 
    Swapping colors~$1$ and~$2$ in $A_1$ and then applying Theorem~\ref{prop:blockul} while referring to Entry~7 in Table~\ref{tb:2cols} shows
    that trees with root color~$2$ are enumerated by $\frac{1}{n}\binom{3n-2}{n-1}=\frac{3n-2}{n(2n-1)}\binom{3n-3}{n-1}$, while trees with root color~$1$ are enumerated by $\frac{1}{2n-1}\binom{3n-3}{n-1}$. Since colors~$1$ and~$3$ are interchangeable, trees with root color~$3$ are also enumerated by $\frac{1}{2n-1}\binom{3n-3}{n-1}$.
    Hence the total number of $n$-vertex trees colored according to any of the four matrices is given by $\frac{2}{n}\binom{3n-3}{n-1}+\frac{1}{2n-1}\binom{3n-3}{n-1}=\frac{5n-2}{n(2n-1)}\binom{3n-3}{n-1}$, as required.
\end{description}

\subsection*{Entry 41}
\begin{description}
    \item[Matrices]
    
    \[\begin{bmatrix}
0 & 1 & 0 \\
1 & 1 & 0 \\
0 & 1 & 1
\end{bmatrix}\]
    \item[Sequences]\hspace*{\fill}\vspace{-0.5cm}
    \begin{center}
        \[
        \begin{array}{|c||c|c|c|c|c|c|c|c||c|}\hline
           n  & 1 & 2 & 3 & 4 & 5 & 6 & 7 & 8  & \text{Formula}\\\hline\hline
           \text{Total}  & 3 & 5 & 18 & 81 & 407 & 2185 & 12261 & 71019 & \frac{2}{n}\binom{3n-3}{n-1}+A394124(n)\\\hline\hline
           \text{Root Color }1^{h} & 1 & 1 & 3 & 12 & 55 & 273 & 1428 & 7752 & \frac{1}{2n-1}\binom{3n-3}{n-1}\\\hline
           \text{Root Color }2^{h} & 1 & 2 & 7 & 30 & 143 & 728 & 3876 & 21318 & \frac{1}{n}\binom{3n-2}{n-1}\\\hline
           \text{Root Color }3 & 1 & 2 & 8 & 39 & 209 & 1184 & 6957 & 41949 & A394124(n) \\\hline
        \end{array}
        \]
    \end{center}
    
    \item[Nontrivial Functional Equations]\hspace*{\fill}

    \begin{itemize}
\item $\gfa{1}^3-x\gfa{1}+x^2$

\item $\gfa{2}^3-2\gfa{2}^2+\gfa{2}-x$

\item $\gfa{3}^6-\gfa{3}^5+3x\gfa{3}^4-x\gfa{3}^3+3x^2\gfa{3}^2-x^2\gfa{3}+x^3$

\end{itemize}

\item[Description] Swapping colors~$1$ and~$2$ and then applying Theorem~\ref{prop:blockul} with Entry~7 in Table~\ref{tb:2cols} 
in mind yields the listed formulas.
    
\end{description}

\subsection*{Entry 42}
\begin{description}
    \item[Matrices]
    
    \[\begin{bmatrix}
1 & 1 & 0 \\
1 & 0 & 0 \\
0 & 1 & 0
\end{bmatrix}\]
    \item[Sequences]\hspace*{\fill}\vspace{-0.5cm}
    \begin{center}
        \[
        \begin{array}{|c||c|c|c|c|c|c|c|c||c|}\hline
           n  & 1 & 2 & 3 & 4 & 5 & 6 & 7 & 8  & \text{Formula}\\\hline\hline
           \text{Total}  & 3 & 4 & 12 & 48 & 221 & 1103 & 5799 & 31619 & \frac{2}{n}\binom{3n-3}{n-1}+A098746(n-1)\\\hline\hline
           \text{Root Color }1^{h} & 1 & 2 & 7 & 30 & 143 & 728 & 3876 & 21318 & \frac{1}{n}\binom{3n-2}{n-1}\\\hline
           \text{Root Color }2^{h} & 1 & 1 & 3 & 12 & 55 & 273 & 1428 & 7752 & \frac{1}{2n-1}\binom{3n-3}{n-1}\\\hline
           \text{Root Color }3 & 1 & 1 & 2 & 6 & 23 & 102 & 495 & 2549 & A098746(n-1)\\\hline
        \end{array}
        \]
    \end{center}
    
    \item[Nontrivial Functional Equations]\hspace*{\fill}

    \begin{itemize}
\item $\gfa{1}^3-2\gfa{1}^2+\gfa{1}-x$

\item $\gfa{2}^3-x\gfa{2}+x^2$

\item $(x^2 - x + 1)\gfa{3}^3+(x^2 - 3x)\gfa{3}^2+3x^2\gfa{3}-x^3$

\end{itemize}
    \item[Discussion] 
    Applying Theorem~\ref{prop:blockul} with Entry~7 in Table~\ref{tb:2cols} 
    in mind yields the listed formulas for root colors~1 and~2.
    The number of trees of root color~3 is equal to the number of sequences of trees of root color~2 with $n-1$ vertices. Sequence \seqnum{A098746} is stated as the INVERT transform of \seqnum{A001764}, which encodes exactly this, as $A001764(n-1)=\frac{1}{2n-1}\binom{3n-3}{n-1}$. This also follows from Theorem~\ref{prop:rootonlysingle} applied to Entry~7 of Table~\ref{tb:2cols}.

Another description of \seqnum{A098746} is Dyck paths of semilength $2n$ with all descents of even length and no valley vertices at height $1$. By reversing the paths, this is the same as counting Dyck paths of semilength $2n$ with all ascents of even length and no valley vertices at height $1$. The bijection for Entry~7 in Table~\ref{tb:2cols} given in Appendix~\ref{app:ent7bij} handles Dyck paths like this without the valley restriction, and it can be adapted to a bijection between trees with root color~$3$ colored according to Entry~42 and such Dyck paths. Via the glove bijection, Dyck paths with no valley vertices at height $1$ correspond to rooted trees where every vertex at depth $1$ has exactly one child. Via the bijection from Appendix~\ref{app:ent7bij}, $n$-vertex colored trees with root color~$1$ and all vertices at depth $1$ having color $2$ are in bijection with $2n$-vertex trees with the root having exactly one child, with all vertices at depth $2$ having exactly one child, and having all downward paths of even length. These are in bijection with trees on $2n-1$ vertices with all vertices at depth $1$ having exactly one child and having all downward paths of even length by deleting the root. These, in turn, are in bijection with the Dyck paths under consideration here. So, to parlay this into a bijection for Entry~42 trees with root color~$3$:
\begin{itemize}
    \item Given a tree colored according to Entry~42 with root color~$3$ (which forces all vertices at depth $1$ to be of color~$2$), re-color the root to color~$1$, apply the bijection $\tau$ from Appendix~\ref{app:ent7bij}, delete the root of the resulting tree, then convert to a Dyck path via the standard glove bijection.
    \item Given a Dyck path of semilength $2n$ with all ascents of even length and no valley vertices at height $1$, apply the standard glove bijection to convert it to a tree on $2n-1$ vertices. Then, add a new root above the existing root of that tree. From there, apply $\tau^{-1}$ from Appendix~\ref{app:ent7bij} to obtain a bicolored tree with root color~$1$ and all vertices at depth $1$ having color~$2$. Finally, re-color the root to color~$3$.
\end{itemize}
\end{description}

\subsection*{Entry 43}
\begin{description}
    \item[Matrices]
    
    \[\begin{bmatrix}
1 & 1 & 0 \\
1 & 1 & 0 \\
0 & 1 & 0
\end{bmatrix}\]
    \item[Sequences]\hspace*{\fill}\vspace{-0.5cm}
    \begin{center}
        \[
        \begin{array}{|c||c|c|c|c|c|c|c|c||c|}\hline
           n  & 1 & 2 & 3 & 4 & 5 & 6 & 7 & 8  & \text{Formula}\\\hline\hline
           \text{Total}  & 3 & 5 & 19 & 93 & 515 & 3069 & 19203 & 124413 & 2^{n}C_{n-1}+A064062(n-1)\\\hline\hline
           \text{Root Color }1^{h} & 1 & 2 & 8 & 40 & 224 & 1344 & 8448 & 54912 & 2^{n-1}C_{n-1}\\\hline
           \text{Root Color }2^{h} & 1 & 2 & 8 & 40 & 224 & 1344 & 8448 & 54912 & 2^{n-1}C_{n-1}\\\hline
           \text{Root Color }3 & 1 & 1 & 3 & 13 & 67 & 381 & 2307 & 14589 & A064062(n-1)\\\hline
        \end{array}
        \]
    \end{center}
    
    \item[Nontrivial Functional Equations]\hspace*{\fill}

    \begin{itemize}

\item $(x + 1)\gfa{3}^2-3x\gfa{3}+2x^2$

\end{itemize}
    \item[Discussion] 
    The formulas for root colors~1 and~2 follow from applying Theorem~\ref{prop:blockul} with Entry~8 in Table~\ref{tb:2cols} in mind.
    Now, suppose the root has color~3. Then, all children have color~$2$. Using the glove bijection, each edge in the tree corresponds to an up-step and a down-step in the path. We can color the up-steps by the child's color of their corresponding edges. Then, each up-step at the ``ground level'' has color~$2$, as these steps come from the edges emanating from the root. The other up-steps can have any combination of colors~$2$ and~$3$. If we erase the colors on the ``ground level'' up-steps, we have exactly the second interpretation of sequence \seqnum{A064062}. This is clearly a bijection, as given such a colored Dyck path, we can obtain a colored tree.
\end{description}

\subsection*{Entry 44}
\begin{description}
    \item[Matrices]
    
    \[\begin{bmatrix}
0 & 1 & 0 \\
1 & 0 & 1 \\
0 & 0 & 0
\end{bmatrix}\]
    \item[Sequences]\hspace*{\fill}\vspace{-0.5cm}
    \begin{center}
        \[
        \begin{array}{|c||c|c|c|c|c|c|c|c||c|}\hline
           n  & 1 & 2 & 3 & 4 & 5 & 6 & 7 & 8  & \text{Formula}\\\hline\hline
           \text{Total}  & 3 & 3 & 8 & 25 & 87 & 325 & 1274 & 5169 & A002212(n - 1) + A007317(n)\\\hline\hline
           \text{Root Color }1 & 1 & 1 & 3 & 10 & 36 & 137 & 543 & 2219 & A002212(n-1)\\\hline
           \text{Root Color }2 & 1 & 2 & 5 & 15 & 51 & 188 & 731 & 2950 & A007317(n)\\\hline
           \text{Root Color }3 & 1 & 0 & 0 & 0 & 0 & 0 & 0 & 0 & 0\\\hline
        \end{array}
        \]
    \end{center}
    
    \item[Nontrivial Functional Equations]\hspace*{\fill}

    \begin{itemize}
\item $\gfa{1}^2+(x - 1)\gfa{1}-x^2 + x$

\item $(x - 1)\gfa{2}^2+(1 - x)\gfa{2}-x$

\end{itemize}
    \item[Discussion] Recall the symbols U, D, and F related to Sch\"oder paths as used in the discussion for Entry~18. Sequence \seqnum{A002212} enumerates Schr\"{o}der paths with no consecutive UD peaks an odd number of steps from the $x$-axis, and sequence \seqnum{A007317} enumerates Schr\"oder paths with no consecutive UD peaks an even number of steps from the $x$-axis. Denote the set of $n$-vertex trees colored by Entry~43 whose root is color $1$ by $X^{(1)}_n$ and whose root is color $2$ by $X^{(2)}_n$. We now exhibit a bijection between $X^{(1)}_n$ and the first set of restricted Schr\"oder paths. Traverse a tree in $X^{(1)}_n$ using depth-first search and write F, when we see a leaf of color 2 or 3, U when we go away from the root and D when we get closer to the root. We obtain a Schr\"{o}der path and the correspondence is one-to-one, as we obviously have the reverse map. Furthermore, an occurrence of UD corresponds to a leaf of color 1, which can only occur at even distance from the root. This map is invertible, so it is a bijection. We obtain a similar bijection between $X^{(2)}_n$ and the second set of restricted Schr\"oder paths. No nontrivial tree is colored by Entry~43 with root colored 3, so the formula in the table follows.
\end{description}

\subsection*{Entry 45}
\begin{description}
    \item[Matrices]
    
    \[\begin{bmatrix}
0 & 1 & 0 \\
1 & 0 & 1 \\
0 & 0 & 1
\end{bmatrix}\]
    \item[Sequences]\hspace*{\fill}\vspace{-0.5cm}
    
    \begin{center}  
        \[
        \begin{array}{|c||c|c|c|c|c|c|c|c||c|}\hline
           n  & 1 & 2 & 3 & 4 & 5 & 6 & 7 & 8  & \text{Formula}\\\hline\hline
           \text{Total}  & 3 & 4 & 11 & 37 & 138 & 549 & 2284 & 9822 & 
           ***\\\hline\hline
           \text{Root Color }1 & 1 & 1 & 3 & 11 & 44 & 185 & 804 & 3579 & A127632(n-1)\\\hline
           \text{Root Color }2 & 1 & 2 & 6 & 21 & 80 & 322 & 1348 & 5814 & A121988(n)\\\hline
           \text{Root Color }3^{h} & 1 & 1 & 2 & 5 & 14 & 42 & 132 & 429 & C_{n-1}\\\hline
        \end{array}
        \]
        $***$: Formula for Total is $t_A(n)=A127632(n - 1) + A121988(n) + C_{n-1}$.
    \end{center}
    
    \item[Nontrivial Functional Equations]\hspace*{\fill}

    \begin{itemize}
\item $\gfa{1}^4-\gfa{1}^3+2x\gfa{1}^2-2x^2\gfa{1}+x^3$

\item $\gfa{2}^4-2\gfa{2}^3+2\gfa{2}^2-\gfa{2}+x$

\end{itemize}
    \item[Discussion] We first prove that the number of trees with root of color~1 colored according to $A$ is given by OEIS sequence $\text{A127632}(n-1)$.
    OEIS entry \seqnum{A127632} has generating function $g(x)$ satisfying
    \[
    xg(x)^4-g(x)^3+2g(x)^2-2g(x)+1=0.
    \]
    Letting $h(x)=xg(x)$, this is the same as
    \[
    x\left(\frac{h(x)}{x}\right)^4-\left(\frac{h(x)}{x}\right)^3+2\left(\frac{h(x)}{x}\right)^2-2\left(\frac{h(x)}{x}\right)+1=0.
    \]
    Algebraic manipulation yields
    \[
    h(x)^4-h(x)^3+2xh(x)^2-2x^2h(x)+x^3=0.
    \]
    So, $h(x)$ satisfies the same polynomial equation as $\gfa{1}$. This fact, combined with the fact that the first several terms of $\text{A127632}(n-1)$ and $t_A^{(1)}(n)$ agree, shows that the number of $n$-vertex trees colored according to $A$ with root of color~1 is enumerated by $\text{A127632}(n-1)$.

    As for root color~2, 
    OEIS entry \seqnum{A121988} says that that sequence is the series reversion of $x-2x^2+2x^3-x^4$, which agrees with the polynomial for $\gfa{2}$.

    Finally, the formula for root color~3 follows from Theorem~\ref{prop:catrows}.
\end{description}

\subsection*{Entry 46}
\begin{description}
    \item[Matrices]
    
    \[\begin{bmatrix}
0 & 1 & 0 \\
1 & 1 & 1 \\
0 & 0 & 0
\end{bmatrix}\]
    \item[Sequences]\hspace*{\fill}\vspace{-0.5cm}
    \begin{center}
        \[
        \begin{array}{|c||c|c|c|c|c|c|c|c||c|}\hline
           n  & 1 & 2 & 3 & 4 & 5 & 6 & 7 & 8  & \text{Formula}\\\hline\hline
           \text{Total}  & 3 & 4 & 17 & 88 & 508 & 3138 & 20296 & 135708 & ***\\\hline\hline
           \text{Root Color }1 & 1 & 1 & 4 & 20 & 113 & 688 & 4404 & 29219 & A108447(n-1)\\\hline
           \text{Root Color }2 & 1 & 3 & 13 & 68 & 395 & 2450 & 15892 & 106489 & A200757(n)\\\hline
           \text{Root Color }3 & 1 & 0 & 0 & 0 & 0 & 0 & 0 & 0 & 0\\\hline
        \end{array}
        \]
        $***$: Formula for Total is $t_A(n)=A108447(n-1)+A200757(n)$.
    \end{center}

    \item[Nontrivial Functional Equations]\hspace*{\fill}

    \begin{itemize}
\item $\gfa{1}^3+x\gfa{1}^2+(-x^2 - x)\gfa{1}+x^2$

\item $\gfa{2}^3+(x - 2)\gfa{2}^2+(1 - x)\gfa{2}-x$

\end{itemize}
    \item[Discussion] We first prove that the number of trees with root of color~1 colored according to $A$ is given by OEIS sequence $\text{A108447}(n-1)$.
    OEIS entry \seqnum{A108447} has generating function $g(x)$ satisfying
    \[
    g(x) = 1 + xg(x)(g(x)^2+g(x)-1).
    \]
    Letting $h(x)=xg(x)$, this is the same as
    \[
    \frac{h(x)}{x} = 1 + x\left(\frac{h(x)}{x}\right)\left(\left(\frac{h(x)}{x}\right)^2+\left(\frac{h(x)}{x}\right)-1\right).
    \]
    Algebraic manipulation yields
    \[
    h(x)^3+xh(x)^2+(-x^2-x)h(x)+x^2=0.
    \]
    So, $h(x)$ satisfies the same polynomial equation as $\gfa{1}$. This fact, combined with the fact that the first several terms of $\text{A108447}(n-1)$ and $t_A^{(1)}(n)$ agree, shows that the number of $n$-vertex trees colored according to $A$ with root of color~1 is enumerated by $\text{A108447}(n-1)$.

    We now prove that the number of trees with root of color~2 colored according to $A$ is given by OEIS sequence $\text{A200757}(n)$.
    OEIS entry \seqnum{A200757} has generating function $g(x)$ satisfying
    \[
    g(x) = x + \frac{(x + g(x))^2}{1-x-g(x)}-\frac{g(x)^2}{1-g(x)}.
    \]
    Algebraic manipulation yields
    \[
    g(x)^3+(x-2)g(x)^2+(1-x)g(x)-x=0.
    \]
    So, $g(x)$ satisfies the same polynomial equation as $\gfa{2}$. This fact, combined with the fact that the first several terms of $\text{A200757}(n-1)$ and $t_A^{(2)}(n)$ agree, shows that the number of $n$-vertex trees colored according to $A$ with root of color~2 is enumerated by $\text{A200757}(n-1)$.

    Finally, the formula for root color~3 follows from Theorem~\ref{prop:zerorows}.
\end{description}

\subsection*{Entry 47}
\begin{description}
    \item[Matrices]
    
    \[\begin{bmatrix}
0 & 1 & 0 \\
1 & 1 & 1 \\
0 & 0 & 1
\end{bmatrix}\]
    \item[Sequences]\hspace*{\fill}\vspace{-0.5cm}
    \begin{center}
        \[
        \begin{array}{|c||c|c|c|c|c|c|c|c||c|}\hline
           n  & 1 & 2 & 3 & 4 & 5 & 6 & 7 & 8  & \text{Formula}\\\hline\hline
           \text{Total}  & 3 & 5 & 20 & 103 & 602 & 3794 & 25128 & 172361 & ***\\\hline\hline
           \text{Root Color }1 & 1 & 1 & 4 & 21 & 124 & 784 & 5190 & 35525 & A394125(n)\\\hline
           \text{Root Color }2 & 1 & 3 & 14 & 77 & 464 & 2968 & 19806 & 136407 & A394126(n)\\\hline
           \text{Root Color }3^{h} & 1 & 1 & 2 & 5 & 14 & 42 & 132 & 429 & C_{n-1}\\\hline
        \end{array}
        \]
        $***$: Formula for Total is $t_A(n)=A394125(n)+A394126(n)+C_{n-1}$.
    \end{center}
    
    \item[Nontrivial Functional Equations]\hspace*{\fill}

    \begin{itemize}
\item $\gfa{1}^6+\gfa{1}^5-2x\gfa{1}^4-x\gfa{1}^3+(x^3 + 3x^2)\gfa{1}^2-3x^3\gfa{1}+x^4$

\item $\gfa{2}^6-3\gfa{2}^5+(x + 3)\gfa{2}^4+(-4x - 1)\gfa{2}^3+4x\gfa{2}^2-x\gfa{2}+x^2$

\end{itemize}
    
\end{description}

\subsection*{Entry 48}
\begin{description}
    \item[Matrices]
    
    \[\begin{bmatrix}
1 & 1 & 0 \\
1 & 0 & 1 \\
0 & 0 & 0
\end{bmatrix}\]
    \item[Sequences]\hspace*{\fill}\vspace{-0.5cm}
    \begin{center}
        \[
        \begin{array}{|c||c|c|c|c|c|c|c|c||c|}\hline
           n  & 1 & 2 & 3 & 4 & 5 & 6 & 7 & 8  & \text{Formula}\\\hline\hline
           \text{Total}  & 3 & 4 & 14 & 62 & 312 & 1694 & 9666 & 57124 & ***\\\hline\hline
           \text{Root Color }1 & 1 & 2 & 8 & 38 & 198 & 1096 & 6330 & 37722 & A394127(n) \\\hline
           \text{Root Color }2 & 1 & 2 & 6 & 24 & 114 & 598 & 3336 & 19402 & A228907(n-1)\\\hline
           \text{Root Color }3 & 1 & 0 & 0 & 0 & 0 & 0 & 0 & 0 & 0\\\hline
        \end{array}
        \]
        $***$: Formula for Total is $t_A(n)=A394127(n)+A228907(n-1)$.
    \end{center}
    
    \item[Nontrivial Functional Equations]\hspace*{\fill}

    \begin{itemize}
\item $\gfa{1}^3+(x - 2)\gfa{1}^2+(1 - x)\gfa{1}+x^2 - x$

\item $(1 - x)\gfa{2}^3+(x^2 - x)\gfa{2}^2+(2x^2 - x)\gfa{2}+x^2$

\end{itemize}
    \item[Discussion] We prove that the number of trees with root of color~2 colored according to $A$ is given by OEIS sequence $\text{A228907}(n-1)$.
    OEIS entry \seqnum{A228907} has generating function $g(x)$ satisfying
    \[
    g(x) = 1 + xg(x)(2 - g(x) + g(x)^2) + x^2g(x)^2(1 - g(x)).
    \]
    Letting $h(x)=xg(x)$, this is the same as
    \[
    \frac{h(x)}{x} = 1 + x\left(\frac{h(x)}{x}\right)\left(2 - \left(\frac{h(x)}{x}\right) + \left(\frac{h(x)}{x}\right)^2\right) + x^2\left(\frac{h(x)}{x}\right)^2\left(1 - \left(\frac{h(x)}{x}\right)\right).
    \]
    Algebraic manipulation yields
    \[
    (1 - x)h(x)^3+(x^2 - x)h(x)^2+(2x^2 - x)h(x)+x^2=0.
    \]
    So, $h(x)$ satisfies the same polynomial equation as $\gfa{2}$. This fact, combined with the fact that the first several terms of $\text{A228907}(n-1)$ and $t_A^{(2)}(n)$ agree, shows that the number of $n$-vertex trees colored according to $A$ with root of color~2 is enumerated by $\text{A228907}(n-1)$.

    Also, the formula for root color~$3$ follows from Theorem~\ref{prop:zerorows}.
\end{description}

\subsection*{Entry 49}
\begin{description}
    \item[Matrices]
    
    \[\begin{bmatrix}
1 & 1 & 0 \\
1 & 0 & 1 \\
0 & 0 & 1
\end{bmatrix}\]
    \item[Sequences]\hspace*{\fill}\vspace{-0.5cm}
    \begin{center}
        \[
        \begin{array}{|c||c|c|c|c|c|c|c|c||c|}\hline
           n  & 1 & 2 & 3 & 4 & 5 & 6 & 7 & 8  & \text{Formula}\\\hline\hline
           \text{Total}  & 3 & 5 & 17 & 74 & 368 & 1989 & 11371 & 67644 & ***\\\hline\hline
           \text{Root Color }1 & 1 & 2 & 8 & 39 & 209 & 1186 & 7000 & 42535 & A394128(n) \\\hline
           \text{Root Color }2 & 1 & 2 & 7 & 30 & 145 & 761 & 4239 & 24680 & A394129(n) \\\hline
           \text{Root Color }3^{h} & 1 & 1 & 2 & 5 & 14 & 42 & 132 & 429 & C_{n-1}\\\hline
        \end{array}
        \]
        $***$: Formula for Total is $t_A(n)=A394128(n)+A394129(n)+C_{n-1}$.
    \end{center}
    
    \item[Nontrivial Functional Equations]\hspace*{\fill}

    \begin{itemize}
\item $\gfa{1}^6-3\gfa{1}^5+(x + 3)\gfa{1}^4+(-3x - 1)\gfa{1}^3+(2x^2 + 2x)\gfa{1}^2-2x^2\gfa{1}+x^3$

\item $\gfa{2}^5-\gfa{2}^4+(1 - 3x)\gfa{2}^3+x\gfa{2}^2+(2x^2 - x)\gfa{2}+x^3$

\end{itemize}
    
\end{description}

\subsection*{Entry 50}
\begin{description}
    \item[Matrices]
    
    \[\begin{bmatrix}
1 & 1 & 0 \\
1 & 1 & 1 \\
0 & 0 & 0
\end{bmatrix}\]
    \item[Sequences]\hspace*{\fill}\vspace{-0.5cm}
    \begin{center}
        \[
        \begin{array}{|c||c|c|c|c|c|c|c|c||c|}\hline
           n  & 1 & 2 & 3 & 4 & 5 & 6 & 7 & 8  & \text{Formula}\\\hline\hline
           \text{Total}  & 3 & 5 & 23 & 131 & 834 & 5685 & 40585 & 299563 & *** \\\hline\hline
           \text{Root Color }1 & 1 & 2 & 9 & 51 & 324 & 2206 & 15737 & 116098 & A378465(n-1) \\\hline
           \text{Root Color }2 & 1 & 3 & 14 & 80 & 510 & 3479 & 24848 & 183465 & A121873(n)\\\hline
           \text{Root Color }3 & 1 & 0 & 0 & 0 & 0 & 0 & 0 & 0 & 0\\\hline
        \end{array}
        \]
        $***$: Formula for Total is $t_A(n)=A378465(n-1) + A121873(n)$.
    \end{center}
    
    \item[Nontrivial Functional Equations]\hspace*{\fill}

    \begin{itemize}
\item $\gfa{1}^3-3\gfa{1}^2+(x + 1)\gfa{1}-x$

\item $\gfa{2}^3+(x + 1)\gfa{2}^2+(2x - 1)\gfa{2}+x$

\end{itemize}
    \item[Discussion] 
    We first prove that the number of trees with root of color~1 colored according to $A$ is given by OEIS sequence $\text{A378465}(n-1)$. OEIS entry \seqnum{A378465} is defined as the expansion of $\frac{1}{x}$ times the series reversion of $x\left(1-x-\frac{x}{1-x}\right)$. So, the generating function $g(x)$ for $\text{A378465}(n-1)$ satisfies
    \[
    g(x)\left(1-g(x)-\frac{g(x)}{1-g(x)}\right)=x.
    \]
    Algebraic manipulation yields
    \[
    g(x)^3-3xg(x)^2+(x+1)g(x)-x=0.
    \]
    So, $g(x)$ satisfies the same polynomial equation as $\gfa{1}$. This fact, combined with the fact that the first several terms of $\text{A378465}(n-1)$ and $t_A^{(1)}(n)$ agree, shows that the number of $n$-vertex trees colored according to $A$ with root of color~1 is enumerated by $\text{A378465}(n-1)$.

    Next, we prove that the number of trees with root of color~2 colored according to $A$ is given by OEIS sequence $\text{A121873}(n-1)$.
    OEIS entry \seqnum{A121873} has generating function $g(x)$ satisfying
    \[
    x = \frac{g(x) - g(x)^2 - g(x)^3}{(1+g(x))^2}.
    \]
    Algebraic manipulation yields
    \[
    g(x)^3+(x + 1)g(x)^2+(2x - 1)g(x)+x=0.
    \]
    So, $g(x)$ satisfies the same polynomial equation as $\gfa{2}$. This fact, combined with the fact that the first several terms of $\text{A121873}(n)$ and $t_A^{(2)}(n)$ agree, shows that the number of $n$-vertex trees colored according to $A$ with root of color~2 is enumerated by $\text{A121873}(n)$.

    Finally, the formula for root color~$3$ follows from Theorem~\ref{prop:zerorows}.
\end{description}

\subsection*{Entry 51}
\begin{description}
    \item[Matrices]
    
    \[\begin{bmatrix}
1 & 1 & 0 \\
1 & 1 & 1 \\
0 & 0 & 1
\end{bmatrix}\]
    \item[Sequences]\hspace*{\fill}\vspace{-0.5cm}
    \begin{center}
        \[
        \begin{array}{|c||c|c|c|c|c|c|c|c||c|}\hline
           n  & 1 & 2 & 3 & 4 & 5 & 6 & 7 & 8  & \text{Formula}\\\hline\hline
           \text{Total}  & 3 & 6 & 26 & 146 & 933 & 6433 & 46610 & 349693 & *** \\\hline\hline
           \text{Root Color }1 & 1 & 2 & 9 & 52 & 338 & 2353 & 17142 & 129024 & A394130(n)\\\hline
           \text{Root Color }2 & 1 & 3 & 15 & 89 & 581 & 4038 & 29336 & 220240 & A394131(n)\\\hline
           \text{Root Color }3^{h} & 1 & 1 & 2 & 5 & 14 & 42 & 132 & 429 & C_{n-1}\\\hline
        \end{array}
        \]
        $***$: Formula for Total is $t_A(n)=A394130(n)+A394131(n)+C_{n-1}$.
    \end{center}
    
    \item[Nontrivial Functional Equations]\hspace*{\fill}

    \begin{itemize}
\item $\gfa{1}^6-4\gfa{1}^5+(6x + 4)\gfa{1}^4+(-9x - 1)\gfa{1}^3+(5x^2 + 3x)\gfa{1}^2-3x^2\gfa{1}+x^3$

\item $\gfa{2}^5+\gfa{2}^4+(5x - 1)\gfa{2}^3+x\gfa{2}^2+(6x^2 - x)\gfa{2}+x^3$

\end{itemize}
    
\end{description}

\subsection*{Entry 52}
\begin{description}
    \item[Matrices]
    
    \[\begin{bmatrix}
0 & 1 & 1 \\
1 & 0 & 0 \\
1 & 0 & 1
\end{bmatrix}\]
    \item[Sequences]\hspace*{\fill}\vspace{-0.5cm}
    \begin{center}
        \[
        \begin{array}{|c||c|c|c|c|c|c|c|c||c|}\hline
           n  & 1 & 2 & 3 & 4 & 5 & 6 & 7 & 8  & \text{Formula}\\\hline\hline
           \text{Total}  & 3 & 5 & 18 & 82 & 422 & 2340 & 13644 & 82477 & ***\\\hline\hline
           \text{Root Color }1 & 1 & 2 & 7 & 31 & 156 & 850 & 4888 & 29221 & A394132(n)\\\hline
           \text{Root Color }2 & 1 & 1 & 3 & 12 & 56 & 288 & 1583 & 9129 & A394133(n)\\\hline
           \text{Root Color }3 & 1 & 2 & 8 & 39 & 210 & 1202 & 7173 & 44127 & A394134(n)\\\hline
        \end{array}
        \]
        $***$: Formula for Total is $t_A(n)=A394132(n)+A394133(n)+A394134(n)$.
    \end{center}
    
    \item[Nontrivial Functional Equations]\hspace*{\fill}

    \begin{itemize}
\item $\gfa{1}^5+(x - 2)\gfa{1}^4+(1 - x)\gfa{1}^3+x\gfa{1}^2-x\gfa{1}+x^2$

\item $\gfa{2}^5-2\gfa{2}^4+(4x + 1)\gfa{2}^3+(-3x^2 - 3x)\gfa{2}^2+(x^3 + 3x^2)\gfa{2}-x^3$

\item $\gfa{3}^5-2\gfa{3}^4+(2x + 1)\gfa{3}^3-4x\gfa{3}^2+(x^2 + x)\gfa{3}-x^2$

\end{itemize}
    
\end{description}

\subsection*{Entry 53}
\begin{description}
    \item[Matrices]
    
    \[\begin{bmatrix}
1 & 1 & 1 \\
1 & 0 & 0 \\
1 & 0 & 0
\end{bmatrix}\]
    \item[Sequences]\hspace*{\fill}\vspace{-0.5cm}
    \begin{center}
        \[
        \begin{array}{|c||c|c|c|c|c|c|c|c||c|}\hline
           n  & 1 & 2 & 3 & 4 & 5 & 6 & 7 & 8  & \text{Formula}\\\hline\hline
           \text{Total}  & 3 & 5 & 22 & 121 & 746 & 4930 & 34140 & 244513 & ***\\\hline\hline
           \text{Root Color }1 & 1 & 3 & 14 & 79 & 494 & 3294 & 22952 & 165127 & A003169(n)\\\hline
           \text{Root Color }2 & 1 & 1 & 4 & 21 & 126 & 818 & 5594 & 39693 & A003168(n-1)\\\hline
           \text{Root Color }3 & 1 & 1 & 4 & 21 & 126 & 818 & 5594 & 39693 & A003168(n-1)\\\hline
        \end{array}
        \]
        $***$: Formula for Total is $t_A(n)=A003169(n)+2\cdot A003168(n-1)$.
    \end{center}
    
    \item[Nontrivial Functional Equations]\hspace*{\fill}

    \begin{itemize}
\item $\gfa{1}^3-2\gfa{1}^2+(1 - x)\gfa{1}-x$

\item $2\gfa{2}^3-x\gfa{2}^2-x\gfa{2}+x^2$

\item $2\gfa{3}^3-x\gfa{3}^2-x\gfa{3}+x^2$

\end{itemize}
    \item[Discussion] 
This is a special case of Family 2 (see Section~\ref{sec:fam3}), where $\ell = 1$ and $m=2$. The generating function identities follow from Theorem~\ref{thm:removesquare-gfs}. Moreover, Theorems~\ref{thm:removesquare},~\ref{thm:removesquare-alternative} and~\ref{thm:removesquare-alternative2} yield formulas for $t_A(n)$, $t_A^{(1)}(n)$ and
$t_A^{(2)}(n) = t_A^{(3)}(n)$: for all $n \geq 2$, we have
\begin{align*}
  t_{A}(n) &= \frac{2}{n} \sum_{k=0}^{n-1} 2^k  \binom{n}{k} \binom{2n-3}{n-k-1} \\
  &= \sum_{k=1}^{n} \frac{3n-2k-1}{n(n-1)} \binom{n}{k} \binom{2n+k-2}{k-1} \\ 
  &= \frac{2^n}{n} \sum_{k=0}^{n-1} \Big({-}\frac12\Big)^k \binom{2n-3}{k} \binom{3n-k-3}{n-k-1}, \\
  t_{A}^{(1)}(n) &= \frac{1}{n} \sum_{k=0}^{n-1} 2^k \binom{n}{k} \binom{2n-2}{n-k-1} \\
  &= \frac{1}{n} \sum_{k=1}^{n} \binom{n}{k} \binom{2n+k-2}{k-1} \\
  &= \frac{2^{n-1}}{n} \sum_{k=0}^{n-1} \Big({-}\frac12\Big)^k \binom{2n-2}{k} \binom{3n-k-2}{n-k-1}, \\
  t_{A}^{(2)}(n) &= \frac{1}{n-1} \sum_{k=1}^{n-1} 2^{k-1} \binom{n-1}{k-1} \binom{2n-2}{n-k-1}  \\
  &= \frac{1}{n-1} \sum_{k=1}^{n-1} \binom{n-1}{k} \binom{2n+k-2}{k-1} \\
  &= \frac{2^{n-2}}{n-1} \sum_{k=0}^{n-2} \Big({-}\frac12\Big)^k \binom{2n-2}{k} \binom{3n-k-3}{n-k-2}.
\end{align*}

Some of these sum formulas appear already in the OEIS.

\end{description}

\subsection*{Entry 54}
\begin{description}
    \item[Matrices]
    
    \[\begin{bmatrix}
1 & 1 & 1 \\
1 & 1 & 0 \\
1 & 0 & 0
\end{bmatrix}\]
    \item[Sequences]\hspace*{\fill}\vspace{-0.5cm}
    \begin{center}
        \[
        \begin{array}{|c||c|c|c|c|c|c|c|c||c|}\hline
           n  & 1 & 2 & 3 & 4 & 5 & 6 & 7 & 8  & \text{Formula}\\\hline\hline
           \text{Total}^{h}  & 3 & 6 & 28 & 165 & 1092 & 7752 & 57684 & 444015 & 
           \frac{3}{n}\binom{4n-4}{n-1}\\\hline\hline
           \text{Root Color }1^h & 1 & 3 & 15 & 91 & 612 & 4389 & 32890 & 254475 & 
           \frac{3}{4n-1}\binom{4n-1}{n-1}\\\hline
           \text{Root Color }2^h & 1 & 2 & 9 & 52 & 340 & 2394 & 17710 & 135720 & 
           \frac{2}{3n-1}\binom{4n-3}{n-1}\\\hline
           \text{Root Color }3^h & 1 & 1 & 4 & 22 & 140 & 969 & 7084 & 53820 & 
           \frac{1}{n}\binom{4n-4}{n-2}\\\hline
        \end{array}
        \]
    \end{center}
    
    \item[Nontrivial Functional Equations]\hspace*{\fill}

    \begin{itemize}
\item $\gfa{1}^4-3\gfa{1}^3+3\gfa{1}^2-\gfa{1}+x$

\item $\gfa{2}^4+2x\gfa{2}^2-x\gfa{2}+x^2$

\item $\gfa{3}^4-x^2\gfa{3}+x^3$

\end{itemize}
    \item[Discussion] This follows from Theorem~\ref{thm:k-plane}.

    The OEIS sequences corresponding to this entry are:
    \begin{itemize}
        \item $t_A(n)=A007228(n-1)$
        \item $t_A^{(1)}(n)=A006632(n)$
        \item $t_A^{(2)}(n)=A069271(n-1)$
        \item $t_A^{(3)}(n)=A002293(n-1)$.
    \end{itemize}
\end{description}

\subsection*{Entry 55}
\begin{description}
    \item[Matrices]
    
    \[\begin{bmatrix}
1 & 1 & 1 \\
1 & 1 & 0 \\
1 & 0 & 1
\end{bmatrix}, \begin{bmatrix}
1 & 1 & 1 \\
1 & 0 & 1 \\
1 & 0 & 1
\end{bmatrix}, \begin{bmatrix}
1 & 1 & 1 \\
1 & 0 & 1 \\
1 & 1 & 0
\end{bmatrix}\]
    \item[Sequences]\hspace*{\fill}\vspace{-0.5cm}
    \begin{center}
        \[
        \begin{array}{|c||c|c|c|c|c|c|c|c||c|}\hline
           n  & 1 & 2 & 3 & 4 & 5 & 6 & 7 & 8  & \text{Formula}\\\hline\hline
           \text{Total}  & 3 & 7 & 34 & 209 & 1446 & 10744 & 83736 & 675367 & ***\\\hline\hline
           \text{Root Color }1 & 1 & 3 & 16 & 103 & 732 & 5534 & 43654 & 355219 & A394135(n) \\\hline
           \text{Root Color }2 & 1 & 2 & 9 & 53 & 357 & 2605 & 20041 & 160074 & A379209(n-1)\\\hline
           \text{Root Color }3 & 1 & 2 & 9 & 53 & 357 & 2605 & 20041 & 160074 & A379209(n-1)\\\hline
        \end{array}
        \]
        $***$: Formula for Total is $t_A(n)=A394135(n)+2\cdot A379209(n-1)$.
    \end{center}
    
    \item[Nontrivial Functional Equations]\hspace*{\fill}

    \begin{itemize}
\item $\gfa{1}^4-2\gfa{1}^3+(1 - 4x)\gfa{1}^2-x^2$

\item $\gfa{2}^4-\gfa{2}^3-x\gfa{2}^2+x\gfa{2}-x^2$

\item $\gfa{3}^4-\gfa{3}^3-x\gfa{3}^2+x\gfa{3}-x^2$

\end{itemize}
    \item[Discussion]
    We prove that the number of trees with root of color~2 colored according to $A_1$ is given by OEIS sequence $\text{A379209}(n-1)$.
    OEIS entry \seqnum{A379209} has generating function $g(x)$ satisfying
    \[
    g(x) = \frac{1}{\left(1-x\cdot g(x)^2\right)\left(1-x\cdot g(x)\right)}.
    \]
    Letting $h(x)=xg(x)$, this is the same as
    \[
    \frac{h(x)}{x} = \frac{1}{\left(1-\frac{h(x)^2}{x}\right)\left(1-h(x)\right)}.
    \]
    Algebraic manipulation yields
    \[
    h(x)^4-h(x)^3-xh(x)^2+xh(x)-x^2=0.
    \]
    So, $h(x)$ satisfies the same polynomial equation as $\gff{2}{A_1}$. This fact, combined with the fact that the first several terms of $\text{A379209}(n-1)$ and $t_{A_1}^{(2)}(n)$ agree, shows that the number of $n$-vertex trees colored according to $A_1$ with root of color~2 is enumerated by $\text{A379209}(n-1)$. Since colors~2 and~3 are interchangeable, this proves that the number of $n$-vertex trees colored according to $A_1$ with root of color~3 is also enumerated by $\text{A379209}(n-1)$.
\end{description}

\subsection*{Entry 56}
\begin{description}
    \item[Matrices]
    
    \[\begin{bmatrix}
0 & 1 & 0 \\
1 & 0 & 1 \\
1 & 0 & 0
\end{bmatrix}\]
    \item[Sequences]\hspace*{\fill}\vspace{-0.5cm}
    \begin{center}
        \[
        \begin{array}{|c||c|c|c|c|c|c|c|c||c|}\hline
           n  & 1 & 2 & 3 & 4 & 5 & 6 & 7 & 8  & \text{Formula}\\\hline\hline
           \text{Total}  & 3 & 4 & 11 & 38 & 147 & 609 & 2643 & 11863 & ***\\\hline\hline
           \text{Root Color }1 & 1 & 1 & 3 & 11 & 44 & 186 & 818 & 3706 & A394136(n)\\\hline
           \text{Root Color }2 & 1 & 2 & 6 & 21 & 81 & 334 & 1444 & 6463 & A394137(n)\\\hline
           \text{Root Color }3 & 1 & 1 & 2 & 6 & 22 & 89 & 381 & 1694 & A200753(n-1)\\\hline
        \end{array}
        \]
        $***$: Formula for Total is $t_A(n)=A394136(n)+A394137(n)+A200753(n-1)$.
    \end{center}
    
    \item[Nontrivial Functional Equations]\hspace*{\fill}

    \begin{itemize}
\item $\gfa{1}^3-2\gfa{1}^2+\gfa{1}+x^2 - x$

\item $(1 - x)\gfa{2}^3+(3x - 2)\gfa{2}^2+(1 - x)\gfa{2}+x^2 - x$

\item $(1 - x)\gfa{3}^3-x\gfa{3}+x^2$

\end{itemize}
    \item[Discussion] We prove that the number of trees with root of color~3 colored according to $A$ is given by OEIS sequence $\text{A200753}(n-1)$.
    OEIS entry \seqnum{A200753} has generating function $g(x)$ satisfying
    \[
    g(x) = 1 + (x - x^2)g(x)^3.
    \]
    Letting $h(x)=xg(x)$, this is the same as
    \[
    \frac{h(x)}{x} = 1 + (x - x^2)\left(\frac{h(x)}{x}\right)^3.
    \]
    Algebraic manipulation yields
    \[
    (1 - x)h(x)^3-xh(x)+x^2=0.
    \]
    So, $h(x)$ satisfies the same polynomial equation as $\gfa{3}$. This fact, combined with the fact that the first several terms of $\text{A200753}(n-1)$ and $t_A^{(3)}(n)$ agree, shows that the number of $n$-vertex trees colored according to $A$ with root of color~3 is enumerated by $\text{A200753}(n-1)$.
\end{description}

\subsection*{Entry 57}
\begin{description}
    \item[Matrices]
    
    \[\begin{bmatrix}
0 & 1 & 0 \\
1 & 1 & 1 \\
1 & 0 & 0
\end{bmatrix}\]
    \item[Sequences]\hspace*{\fill}\vspace{-0.5cm}
    \begin{center}
        \[
        \begin{array}{|c||c|c|c|c|c|c|c|c||c|}\hline
           n  & 1 & 2 & 3 & 4 & 5 & 6 & 7 & 8  & \text{Formula}\\\hline\hline
           \text{Total}  & 3 & 5 & 20 & 105 & 624 & 3982 & 26636 & 184285 & *** \\\hline\hline
           \text{Root Color }1 & 1 & 1 & 4 & 21 & 124 & 786 & 5228 & 36005 & A244062(n-1)\\\hline
           \text{Root Color }2 & 1 & 3 & 14 & 77 & 466 & 3002 & 20200 & 140333 & A394138(n) \\\hline
           \text{Root Color }3 & 1 & 1 & 2 & 7 & 34 & 194 & 1208 & 7947 & A394139(n) \\\hline
        \end{array}
        \]
        $***$: Formula for Total is $t_A(n)=A244062(n-1)+A394138(n)+A394139(n)$.
    \end{center}
    
    \item[Nontrivial Functional Equations]\hspace*{\fill}

    \begin{itemize}
\item $\gfa{1}^4-\gfa{1}^3-2x\gfa{1}^2+(2x^2 + x)\gfa{1}-x^2$

\item $\gfa{2}^4+(2x - 3)\gfa{2}^3+(3 - 4x)\gfa{2}^2+(x - 1)\gfa{2}-x^2 + x$

\item $(x - 1)\gfa{3}^4+(-2x^2 + 3x - 1)\gfa{3}^3+(-2x^2 + 3x)\gfa{3}^2-3x^2\gfa{3}+x^3$

\end{itemize}
    \item[Discussion] We prove that the number of trees with root of color~1 colored according to $A$ is given by OEIS sequence $\text{A244062}(n-1)$.
    OEIS entry \seqnum{A244062} has generating function $g(x)$ satisfying
    \[
    g(x) = \frac{2}{1 - xg(x)^2} - \frac{1}{1 - xg(x)}
    \]
    Letting $h(x)=xg(x)$, this is the same as
    \[
    \frac{h(x)}{x} = \frac{2}{1 - x\left(\frac{h(x)}{x}\right)^2} - \frac{1}{1 - h(x)}.
    \]
    Algebraic manipulation yields
    \[
    h(x)^4-h(x)^3-2xh(x)^2+(2x^2 + x)h(x)-x^2=0.
    \]
    So, $h(x)$ satisfies the same polynomial equation as $\gfa{1}$. This fact, combined with the fact that the first several terms of $\text{A244062}(n-1)$ and $t_A^{(1)}(n)$ agree, shows that the number of $n$-vertex trees colored according to $A$ with root of color~1 is enumerated by $\text{A244062}(n-1)$.
\end{description}

\subsection*{Entry 58}
\begin{description}
    \item[Matrices]
    
    \[\begin{bmatrix}
0 & 1 & 0 \\
1 & 1 & 1 \\
1 & 0 & 1
\end{bmatrix}\]
    \item[Sequences]\hspace*{\fill}\vspace{-0.5cm}
    \begin{center}
        \[
        \begin{array}{|c||c|c|c|c|c|c|c|c||c|}\hline
           n  & 1 & 2 & 3 & 4 & 5 & 6 & 7 & 8  & \text{Formula}\\\hline\hline
           \text{Total}  & 3 & 6 & 26 & 142 & 873 & 5772 & 40092 & 288574 & ***\\\hline\hline
           \text{Root Color }1 & 1 & 1 & 4 & 22 & 138 & 930 & 6561 & 47800 & A394140(n)\\\hline
           \text{Root Color }2 & 1 & 3 & 15 & 89 & 577 & 3954 & 28159 & 206344 & A394141(n)\\\hline
           \text{Root Color }3 & 1 & 2 & 7 & 31 & 158 & 888 & 5372 & 34430 & A394142(n)\\\hline
        \end{array}
        \]
        $***$: Formula for Total is $t_A(n)=A394140(n)+A394141(n)+A394142(n)$.
    \end{center}
    
    \item[Nontrivial Functional Equations]\hspace*{\fill}

    \begin{itemize}
\item $(x + 1)\gfa{1}^5-x\gfa{1}^4+(-2x^2 - x)\gfa{1}^3+(2x^3 + 3x^2)\gfa{1}^2-3x^3\gfa{1}+x^4$

\item $\gfa{2}^5-2\gfa{2}^4+(2x + 1)\gfa{2}^3-4x\gfa{2}^2+x\gfa{2}-x^2$

\item $(x + 1)\gfa{3}^5+(-x - 2)\gfa{3}^4+(2x^2 + 3x + 1)\gfa{3}^3+(-x^2 - 3x)\gfa{3}^2+3x^2\gfa{3}-x^3$

\end{itemize}
    
\end{description}

\subsection*{Entry 59}
\begin{description}
    \item[Matrices]
    
    \[\begin{bmatrix}
1 & 1 & 0 \\
1 & 0 & 1 \\
1 & 0 & 0
\end{bmatrix}\]
    \item[Sequences]\hspace*{\fill}\vspace{-0.5cm}
    \begin{center}
        \[
        \begin{array}{|c||c|c|c|c|c|c|c|c||c|}\hline
           n  & 1 & 2 & 3 & 4 & 5 & 6 & 7 & 8  & \text{Formula}\\\hline\hline
           \text{Total}  & 3 & 5 & 18 & 83 & 433 & 2432 & 14347 & 87653 & *** \\\hline\hline
           \text{Root Color }1 & 1 & 2 & 8 & 39 & 210 & 1203 & 7192 & 44362 & A394143(n)\\\hline
           \text{Root Color }2 & 1 & 2 & 7 & 31 & 157 & 864 & 5024 & 30370 & A394144(n)\\\hline
           \text{Root Color }3 & 1 & 1 & 3 & 13 & 66 & 365 & 2131 & 12921 & A394145(n) \\\hline
        \end{array}
        \]
        $***$: Formula for Total is $t_A(n)=A394143(n)+A394144(n)+A394145(n)$.
    \end{center}
    
    \item[Nontrivial Functional Equations]\hspace*{\fill}

    \begin{itemize}
\item $\gfa{1}^4-3\gfa{1}^3+(3 - x)\gfa{1}^2-\gfa{1}-x^2 + x$

\item $(1 - x)\gfa{2}^4+x\gfa{2}^3+(4x^2 - 2x)\gfa{2}^2+x^2\gfa{2}+x^3$

\item $\gfa{3}^4-2\gfa{3}^3+x\gfa{3}^2+x\gfa{3}-x^2$

\end{itemize}
    
\end{description}

\subsection*{Entry 60}
\begin{description}
    \item[Matrices]
    
    \[\begin{bmatrix}
1 & 1 & 0 \\
1 & 1 & 1 \\
1 & 0 & 0
\end{bmatrix}\]
    \item[Sequences]\hspace*{\fill}\vspace{-0.5cm}
    \begin{center}
        \[
        \begin{array}{|c||c|c|c|c|c|c|c|c||c|}\hline
           n  & 1 & 2 & 3 & 4 & 5 & 6 & 7 & 8  & \text{Formula}\\\hline\hline
           \text{Total}  & 3 & 6 & 27 & 156 & 1017 & 7122 & 52308 & 397500 & *** \\\hline\hline
           \text{Root Color }1 & 1 & 2 & 9 & 52 & 339 & 2374 & 17436 & 132500 & A394146(n)\\\hline
           \text{Root Color }2 & 1 & 3 & 15 & 90 & 597 & 4221 & 31188 & 237978 & A394147(n)\\\hline
           \text{Root Color }3 & 1 & 1 & 3 & 14 & 81 & 527 & 3684 & 27022 & A256331(n-1)\\\hline
        \end{array}
        \]
        $***$: Formula for Total is $t_A(n)=A394146(n)+A394147(n)+A256331(n-1)$.
    \end{center}
    
    \item[Nontrivial Functional Equations]\hspace*{\fill}

    \begin{itemize}
\item $3\gfa{1}^3-4\gfa{1}^2+(2x + 1)\gfa{1}-x$

\item $3\gfa{2}^3+(6x - 2)\gfa{2}^2+x\gfa{2}+x^2$

\item $\gfa{3}^3+(-2x - 2)\gfa{3}^2+5x\gfa{3}-3x^2$

\end{itemize}
    \item[Discussion] We prove that the number of trees with root of color~3 colored according to $A$ is given by OEIS sequence $\text{A256331}(n-1)$.
    OEIS entry \seqnum{A256331} has generating function $g(x)$ satisfying
    \[
    xg(x)^3 - (2x+2)g(x)^2 + 5g(x) - 3=0.
    \]
    Letting $h(x)=xg(x)$, this is the same as
    \[
    x\left(\frac{h(x)}{x}\right)^3 - (2x+2)\left(\frac{h(x)}{x}\right)^2 + 5\left(\frac{h(x)}{x}\right) - 3=0.
    \]
    Algebraic manipulation yields
    \[
    h(x)^3+(-2x - 2)h(x)^2+5xh(x)-3x^2=0.
    \]
    So, $h(x)$ satisfies the same polynomial equation as $\gfa{3}$. This fact, combined with the fact that the first several terms of $\text{A256331}(n-1)$ and $t_A^{(3)}(n)$ agree, shows that the number of $n$-vertex trees colored according to $A$ with root of color~3 is enumerated by $\text{A256331}(n-1)$.
\end{description}

\subsection*{Entry 61}
\begin{description}
    \item[Matrices]
    
    \[\begin{bmatrix}
1 & 1 & 0 \\
1 & 1 & 1 \\
1 & 0 & 1
\end{bmatrix}\]
    \item[Sequences]\hspace*{\fill}\vspace{-0.5cm}
    \begin{center}
        \[
        \begin{array}{|c||c|c|c|c|c|c|c|c||c|}\hline
           n  & 1 & 2 & 3 & 4 & 5 & 6 & 7 & 8  & \text{Formula}\\\hline\hline
           \text{Total}  & 3 & 7 & 33 & 196 & 1311 & 9431 & 71262 & 557853 & *** \\\hline\hline
           \text{Root Color }1 & 1 & 2 & 9 & 53 & 356 & 2581 & 19661 & 155043 & A394148(n) \\\hline
           \text{Root Color }2 & 1 & 3 & 16 & 102 & 713 & 5280 & 40697 & 323103 & A394149(n) \\\hline
           \text{Root Color }3 & 1 & 2 & 8 & 41 & 242 & 1570 & 10904 & 79707 & A394150(n) \\\hline
        \end{array}
        \]
        $***$: Formula for Total is $t_A(n)=A394148(n)+A394149(n)+A394150(n)$.
    \end{center}
    
    \item[Nontrivial Functional Equations]\hspace*{\fill}

    \begin{itemize}
\item $3\gfa{1}^6-7\gfa{1}^5+(9x + 5)\gfa{1}^4+(-11x - 1)\gfa{1}^3+(6x^2 + 3x)\gfa{1}^2-3x^2\gfa{1}+x^3$

\item $3\gfa{2}^6-2\gfa{2}^5+9x\gfa{2}^4-3x\gfa{2}^3+6x^2\gfa{2}^2-x^2\gfa{2}+x^3$

\item $\gfa{3}^6-2\gfa{3}^5+(6x + 1)\gfa{3}^4-8x\gfa{3}^3+(9x^2 + 2x)\gfa{3}^2-5x^2\gfa{3}+3x^3$

\end{itemize}
    
\end{description}

\subsection*{Entry 62}
\begin{description}
    \item[Matrices]
    
    \[\begin{bmatrix}
1 & 1 & 1 \\
1 & 0 & 1 \\
0 & 0 & 0
\end{bmatrix}\]
    \item[Sequences]\hspace*{\fill}\vspace{-0.5cm}
    \begin{center}
        \[
        \begin{array}{|c||c|c|c|c|c|c|c|c||c|}\hline
           n  & 1 & 2 & 3 & 4 & 5 & 6 & 7 & 8  & \text{Formula}\\\hline\hline
           \text{Total}  & 3 & 5 & 21 & 112 & 674 & 4356 & 29518 & 206924 & ***\\\hline\hline
           \text{Root Color }1 & 1 & 3 & 14 & 78 & 479 & 3131 & 21372 & 150588 & (-1)^{n-1}\cdot A366326(n)\\\hline
           \text{Root Color }2 & 1 & 2 & 7 & 34 & 195 & 1225 & 8146 & 56336 & A199475(n-1)\\\hline
           \text{Root Color }3 & 1 & 0 & 0 & 0 & 0 & 0 & 0 & 0 & 0\\\hline
        \end{array}
        \]
        $***$: Formula for Total is $t_A(n)=(-1)^{n-1}\cdot A366326(n) + A199475(n-1)$.
    \end{center}
    
    \item[Nontrivial Functional Equations]\hspace*{\fill}

    \begin{itemize}
\item $\gfa{1}^3+(2x - 2)\gfa{1}^2+(x^2 - 2x + 1)\gfa{1}+x^2 - x$

\item $(1 - x)\gfa{2}^3+(x^2 - x)\gfa{2}+x^2$

\end{itemize}
    \item[Discussion] 
    First, we prove that the number of trees with root of color~1 colored according to $A$ is given by $(-1)^{n-1}\cdot \text{A366326}(n)$ for $n\geq2$. OEIS entry \seqnum{A366326} has generating function $g(x)$ defined by
    \[
    g(x)=\left(1+x\right)\left(1+\frac{x}{g(x)^2}\right).
    \]
    Substituting $-x$ for $x$ gives the equation
    \[
    g(-x)=\left(1-x\right)\left(1-\frac{x}{g(-x)^2}\right),
    \]
    where $g(-x)$ is the generating function for $(-1)^{n}\cdot \text{A366326}(n)$.
    Letting $h(x)=1-x-g(-x)$, we have $g(-x)=1-x-h(x)$, and, hence,
    \[
    1-x-h(x)=\left(1-x\right)\left(1-\frac{x}{\left(1-x-h(x)\right)^2}\right).
    \]
    Algebraic manipulation yields
    \[
    h(x)^3+(2x - 2)h(x)^2+(x^2 - 2x + 1)h(x)+x^2 - x.
    \]
    So, $h(x)$ satisfies the same polynomial equation as $\gfa{1}$. This fact, combined with the fact that the several terms of $(-1)^{n-1}\cdot \text{A366326}(n)$ and $t_A^{(1)}(n)$ agree starting from $n=2$, shows that the number of $n$-vertex trees colored according to $A$ with root of color~1 is enumerated by $(-1)^{n-1}\cdot \text{A366326}(n)$ for $n\geq2$.
    
    Next, we prove that the number of trees with root of color~2 colored according to $A$ is given by OEIS sequence $\text{A199475}(n-1)$.
    OEIS entry \seqnum{A199475} has generating function $g(x)$ defined by
    \[
    g(x) = \frac{1}{(1-x)(1 - xg(x)^2)}.
    \]
    Letting $h(x)=xg(x)$, this is the same as
    \[
    \frac{h(x)}{x}=\frac{1}{(1-x)\left(1 - \frac{h(x)^2}{x}\right)}.
    \]
    Algebraic manipulation yields
    \[
    (1-x)h(x)^3+(x^2-x)h(x)+x^2=0.
    \]
    So, $h(x)$ satisfies the same polynomial equation as $\gfa{2}$. This fact, combined with the fact that the first several terms of $\text{A199475}(n-1)$ and $t_A^{(2)}(n)$ agree, shows that the number of $n$-vertex trees colored according to $A$ with root of color~2 is enumerated by $\text{A199475}(n-1)$.
\end{description}

\subsection*{Entry 63}
\begin{description}
    \item[Matrices]
    
    \[\begin{bmatrix}
1 & 1 & 1 \\
1 & 0 & 1 \\
0 & 0 & 1
\end{bmatrix}\]
    \item[Sequences]\hspace*{\fill}\vspace{-0.5cm}
    \begin{center}
        \[
        \begin{array}{|c||c|c|c|c|c|c|c|c||c|}\hline
           n  & 1 & 2 & 3 & 4 & 5 & 6 & 7 & 8  & \text{Formula}\\\hline\hline
           \text{Total}  & 3 & 6 & 25 & 134 & 818 & 5399 & 37512 & 270207 & *** \\\hline\hline
           \text{Root Color }1 & 1 & 3 & 15 & 88 & 563 & 3812 & 26877 & 195349 & A394151(n)\\\hline
           \text{Root Color }2 & 1 & 2 & 8 & 41 & 241 & 1545 & 10503 & 74429 & A381826(n-1)\\\hline
           \text{Root Color }3^{h} & 1 & 1 & 2 & 5 & 14 & 42 & 132 & 429 & C_{n-1}\\\hline
        \end{array}
        \]
        $***$: Formula for Total is $t_A(n)=A394151(n)+A381826(n-1)+C_{n-1}$.
    \end{center}
    
    \item[Nontrivial Functional Equations]\hspace*{\fill}

    \begin{itemize}
\item $\gfa{1}^6-2\gfa{1}^5+(2x + 1)\gfa{1}^4-3x\gfa{1}^3+5x^2\gfa{1}^2-x^2\gfa{1}+x^3$

\item $\gfa{2}^6-2x\gfa{2}^4+x\gfa{2}^3+x^2\gfa{2}^2-x^2\gfa{2}+x^3$

\end{itemize}
    \item[Discussion] 
    We prove that the number of trees with root of color~2 colored according to $A$ is given by OEIS sequence $\text{A381826}(n-1)$.
    OEIS entry \seqnum{A381826} has generating function $g(x)$ satisfying
    \[
    g(x) = \frac{C(x)}{1-xg(x)^2},
    \]
    where $C(x)$ is the generating function for the Catalan numbers. Solving for $C(x)$, we have
    \[
    C(x)=g(x)\left(1-xg(x)^2\right).
    \]
    Since $xC(x)^2-C(x)+1=0$, we have
    \[
    x\left(g(x)\left(1-xg(x)^2\right)\right)^2-\left(g(x)\left(1-xg(x)^2\right)\right)+1=0.
    \]
    Algebraic manipulation yields
    \[
    x^3g(x)^6-2x^2g(x)^4+xg(x)^3+xg(x)^2-g(x)+1=0.
    \]
    Letting $h(x)=xg(x)$ and doing more algebraic manipulation, this is the same as
    \[
    h(x)^6-2xh(x)^4+xh(x)^3+x^2h(x)^2-x^2h(x)+x^3=0.
    \]
    So, $h(x)$ satisfies the same polynomial equation as $\gfa{2}$. This fact, combined with the fact that the first several terms of $\text{A381826}(n-1)$ and $t_A^{(2)}(n)$ agree, shows that the number of $n$-vertex trees colored according to $A$ with root of color~2 is enumerated by $\text{A381826}(n-1)$.
\end{description}

\subsection*{Entry 64}
\begin{description}
    \item[Matrices]
    
    \[\begin{bmatrix}
1 & 1 & 1 \\
1 & 1 & 1 \\
0 & 0 & 0
\end{bmatrix}\]
    \item[Sequences]\hspace*{\fill}\vspace{-0.5cm}
    \begin{center}
        \[
        \begin{array}{|c||c|c|c|c|c|c|c|c||c|}\hline
           n  & 1 & 2 & 3 & 4 & 5 & 6 & 7 & 8  & \text{Formula}\\\hline\hline
           \text{Total}  & 3 & 6 & 30 & 186 & 1290 & 9582 & 74550 & 599730 & A239488(n-1)
           \\\hline\hline
           \text{Root Color }1 & 1 & 3 & 15 & 93 & 645 & 4791 & 37275 & 299865 & A103210(n-1)\\\hline
           \text{Root Color }2 & 1 & 3 & 15 & 93 & 645 & 4791 & 37275 & 299865 & A103210(n-1)\\\hline
           \text{Root Color }3 & 1 & 0 & 0 & 0 & 0 & 0 & 0 & 0 & 0\\\hline
        \end{array}
        \]
    \end{center}
    
    \item[Nontrivial Functional Equations]\hspace*{\fill}

    \begin{itemize}
\item $2\gfa{1}^2+(x - 1)\gfa{1}+x$

\item $2\gfa{2}^2+(x - 1)\gfa{2}+x$

\end{itemize}
    \item[Discussion] 
This is the special case $\ell=2$, $m=1$ of Theorem~\ref{thm:gennarayana}: trees with root color 3 exist only for $n=1$, while the number of trees with root color 1 or 2 colored according to $A$ is
$$\sum_{k=1}^{n-1} N_{n-1,k} 2^{n-k-1}3^k,$$
and the total number is
$$\sum_{k=1}^{n-1} N_{n-1,k} 2^{n-k}3^k$$
for $n > 1$. These are OEIS sequences \seqnum{A103210} and \seqnum{A239488}, respectively.

\end{description}

\subsection*{Entry 65}
\begin{description}
    \item[Matrices]
    
    \[\begin{bmatrix}
1 & 1 & 1 \\
1 & 1 & 1 \\
0 & 0 & 1
\end{bmatrix}\]
    \item[Sequences]\hspace*{\fill}\vspace{-0.5cm}
    \begin{center}
        \[
        \begin{array}{|c||c|c|c|c|c|c|c|c||c|}\hline
           n  & 1 & 2 & 3 & 4 & 5 & 6 & 7 & 8  & \text{Formula}\\\hline\hline
           \text{Total}  & 3 & 7 & 34 & 211 & 1482 & 11206 & 88960 & 731019 & 2\cdot A394152(n)+C_{n-1}\\\hline\hline
           \text{Root Color }1 & 1 & 3 & 16 & 103 & 734 & 5582 & 44414 & 365295 & A394152(n) \\\hline
           \text{Root Color }2 & 1 & 3 & 16 & 103 & 734 & 5582 & 44414 & 365295 & A394152(n) \\\hline
           \text{Root Color }3^{h} & 1 & 1 & 2 & 5 & 14 & 42 & 132 & 429 & C_{n-1}\\\hline
        \end{array}
        \]
    \end{center}
    
    \item[Nontrivial Functional Equations]\hspace*{\fill}

    \begin{itemize}
\item $4\gfa{1}^4-2\gfa{1}^3+5x\gfa{1}^2-x\gfa{1}+x^2$

\item $4\gfa{2}^4-2\gfa{2}^3+5x\gfa{2}^2-x\gfa{2}+x^2$

\end{itemize}
    
\end{description}

\subsection*{Entry 66}
\begin{description}
    \item[Matrices]
    
    \[\begin{bmatrix}
1 & 1 & 1 \\
1 & 0 & 1 \\
0 & 1 & 0
\end{bmatrix}\]
    \item[Sequences]\hspace*{\fill}\vspace{-0.5cm}
    \begin{center}
        \[
        \begin{array}{|c||c|c|c|c|c|c|c|c||c|}\hline
           n  & 1 & 2 & 3 & 4 & 5 & 6 & 7 & 8  & \text{Formula}\\\hline\hline
           \text{Total}  & 3 & 6 & 26 & 144 & 902 & 6082 & 43068 & 315810 & *** \\\hline\hline
           \text{Root Color }1 & 1 & 3 & 15 & 89 & 579 & 3995 & 28721 & 212847 & A394153(n) \\\hline
           \text{Root Color }2 & 1 & 2 & 8 & 42 & 254 & 1672 & 11638 & 84254 & A394154(n) \\\hline
           \text{Root Color }3 & 1 & 1 & 3 & 13 & 69 & 415 & 2709 & 18709 & A394155(n) \\\hline
        \end{array}
        \]
        $***$: Formula for Total is $t_A(n)=A394153(n)+A394154(n)+A394155(n)$.
    \end{center}
    
    \item[Nontrivial Functional Equations]\hspace*{\fill}

    \begin{itemize}
\item $(1 - x)\gfa{1}^4+(5x - 2)\gfa{1}^3+(1 - 5x)\gfa{1}^2+x^2\gfa{1}-x^2$

\item $\gfa{2}^4+(x - 1)\gfa{2}^3-x\gfa{2}^2+x\gfa{2}-x^2$

\item $(x - 1)\gfa{3}^4+(2x + 1)\gfa{3}^3+(-2x^2 - 3x)\gfa{3}^2+(x^3 + 3x^2)\gfa{3}-x^3$

\end{itemize}
    
\end{description}

\subsection*{Entry 67}
\begin{description}
    \item[Matrices]
    
    \[\begin{bmatrix}
1 & 1 & 1 \\
1 & 0 & 1 \\
0 & 1 & 1
\end{bmatrix}\]
    \item[Sequences]\hspace*{\fill}\vspace{-0.5cm}
    \begin{center}
        \[
        \begin{array}{|c||c|c|c|c|c|c|c|c||c|}\hline
           n  & 1 & 2 & 3 & 4 & 5 & 6 & 7 & 8  & \text{Formula}\\\hline\hline
           \text{Total}  & 3 & 7 & 33 & 195 & 1293 & 9200 & 68662 & 530440 & *** \\\hline\hline
           \text{Root Color }1 & 1 & 3 & 16 & 102 & 712 & 5256 & 40317 & 318079 & A394156(n) \\\hline
           \text{Root Color }2 & 1 & 2 & 9 & 52 & 340 & 2395 & 17736 & 136156 & A394157(n) \\\hline
           \text{Root Color }3 & 1 & 2 & 8 & 41 & 241 & 1549 & 10609 & 76205 & A394158(n) \\\hline
        \end{array}
        \]
        $***$: Formula for Total is $t_A(n)=A394156(n)+A394157(n)+A394158(n)$.
    \end{center}
    
    \item[Nontrivial Functional Equations]\hspace*{\fill}

    \begin{itemize}
\item $\gfa{1}^7-2\gfa{1}^6+(4x + 1)\gfa{1}^5-7x\gfa{1}^4+(5x^2 + 2x)\gfa{1}^3-6x^2\gfa{1}^2+(x^3 + x^2)\gfa{1}-x^3$

\item $\gfa{2}^7+x\gfa{2}^6-2x\gfa{2}^5-x^2\gfa{2}^4+2x^2\gfa{2}^3-x^3\gfa{2}+x^4$

\item $\gfa{3}^7-4\gfa{3}^6+(2x + 6)\gfa{3}^5+(-7x - 4)\gfa{3}^4+(3x^2 + 8x + 1)\gfa{3}^3+(-5x^2 - 3x)\gfa{3}^2+(x^3 + 3x^2)\gfa{3}-x^3$

\end{itemize}
    
\end{description}

\subsection*{Entry 68}
\begin{description}
    \item[Matrices]
    
    \[\begin{bmatrix}
1 & 1 & 1 \\
1 & 0 & 1 \\
1 & 0 & 0
\end{bmatrix}\]
    \item[Sequences]\hspace*{\fill}\vspace{-0.5cm}
    \begin{center}
        \[
        \begin{array}{|c||c|c|c|c|c|c|c|c||c|}\hline
           n  & 1 & 2 & 3 & 4 & 5 & 6 & 7 & 8  & \text{Formula}\\\hline\hline
           \text{Total}  & 3 & 6 & 27 & 155 & 1003 & 6971 & 50814 & 383261 & *** \\\hline\hline
           \text{Root Color }1 & 1 & 3 & 15 & 90 & 596 & 4201 & 30914 & 234762 & A394159(n) \\\hline
           \text{Root Color }2 & 1 & 2 & 8 & 43 & 268 & 1819 & 13045 & 97227 & A394160(n) \\\hline
           \text{Root Color }3 & 1 & 1 & 4 & 22 & 139 & 951 & 6855 & 51272 & A394161(n) \\\hline
        \end{array}
        \]
        $***$: Formula for Total is $t_A(n)=A394159(n)+A394160(n)+A394161(n)$.
    \end{center}
    
    \item[Nontrivial Functional Equations]\hspace*{\fill}

    \begin{itemize}
\item $\gfa{1}^5-4\gfa{1}^4+(6 - 2x)\gfa{1}^3+(3x - 4)\gfa{1}^2+\gfa{1}+x^2 - x$

\item $(x - 1)\gfa{2}^5-x\gfa{2}^4+(-2x^2 + 2x)\gfa{2}^3-x^2\gfa{2}+x^3$

\item $\gfa{3}^5-3x\gfa{3}^3+2x^2\gfa{3}^2+x^2\gfa{3}-x^3$

\end{itemize}
    
\end{description}

\subsection*{Entry 69}
\begin{description}
    \item[Matrices]
    
    \[\begin{bmatrix}
1 & 1 & 1 \\
1 & 1 & 1 \\
1 & 0 & 0
\end{bmatrix}\]
    \item[Sequences]\hspace*{\fill}\vspace{-0.5cm}
    \begin{center}
        \[
        \begin{array}{|c||c|c|c|c|c|c|c|c||c|}\hline
           n  & 1 & 2 & 3 & 4 & 5 & 6 & 7 & 8  & \text{Formula}\\\hline\hline
           \text{Total}  & 3 & 7 & 36 & 233 & 1692 & 13174 & 107496 & 907221 & ***\\\hline\hline
           \text{Root Color }1 & 1 & 3 & 16 & 105 & 768 & 6006 & 49152 & 415701 & A085614(n)\\\hline
           \text{Root Color }2 & 1 & 3 & 16 & 105 & 768 & 6006 & 49152 & 415701 & A085614(n)\\\hline
           \text{Root Color }3 & 1 & 1 & 4 & 23 & 156 & 1162 & 9192 & 75819 & A007297(n)\\\hline
        \end{array}
        \]
        $***$: Formula for Total is $t_A(n)=2\cdot A085614(n) + A007297(n)$.
    \end{center}
    
    \item[Nontrivial Functional Equations]\hspace*{\fill}

    \begin{itemize}
\item $2\gfa{1}^3-3\gfa{1}^2+\gfa{1}-x$

\item $2\gfa{2}^3-3\gfa{2}^2+\gfa{2}-x$

\item $\gfa{3}^3+\gfa{3}^2-3x\gfa{3}+2x^2$

\end{itemize}
    \item[Discussion]
    We first prove that the number of trees with root of color~2 colored according to $A$ is given by OEIS sequence $\text{A085614}(n)$.
    The first formula for OEIS sequence \seqnum{A085614} states that that sequence's ordinary generating function is the series reversion of $x-3x^2+2x^3$, which is exactly what we have here by looking at the functional equation for $\gfa{1}$. 
    Since colors~1 and~2 are interchangeable, this proves that $t_A^{(1)}$ and $t_A^{(2)}$ are both enumerated by \seqnum{A085614}.

    We now prove that the number of trees with root of color~3 colored according to $A$ is given by OEIS sequence $\text{A007297}(n)$.
    Let $g(x)=\frac{\gfa{3}}{x}-1$, so that $\gfa{3}=x\left(g(x)+1\right)$. We have
    \begin{align*}
        &\phantom{\Rightarrow} x^3\left(g(x)+1\right)^3+x^2\left(g(x)+1\right)^2-3x^2\left(g(x)+1\right)+2x^2=0\\
        &\Rightarrow x^3\left(g(x)^3+3g(x)^2+3g(x)+1\right)+x^2\left(g(x)^2+2g(x)+1\right)-3x^2\left(g(x)+1\right)+2x^2=0\\
        &\Rightarrow x^3g(x)^3+\left(3x^3+x^2\right)g(x)^2+\left(3x^3-x^2\right)g(x)+x^3=0\\
        &\Rightarrow xg(x)^3+\left(3x+1\right)g(x)^2+\left(3x-1\right)g(x)+x=0\\
        &\Rightarrow x(g(x)+1)^3-g(x)+g(x)^2+x=0\\
        &\Rightarrow x=\frac{g(x)-g(x)^2}{(1+g(x))^3}.
    \end{align*}
    This shows that $g(x)$ is the generating function for OEIS sequence \seqnum{A263843}, which, in turn, implies that $t_A^{(3)}$ is enumerated by \seqnum{A007297}.
\end{description}

\subsection*{Entry 70}
\begin{description}
    \item[Matrices]
    
    \[\begin{bmatrix}
1 & 1 & 1 \\
1 & 1 & 1 \\
1 & 0 & 1
\end{bmatrix}\]
    \item[Sequences]\hspace*{\fill}\vspace{-0.5cm}
    \begin{center}
        \[
        \begin{array}{|c||c|c|c|c|c|c|c|c||c|}\hline
           n  & 1 & 2 & 3 & 4 & 5 & 6 & 7 & 8  & \text{Formula}\\\hline\hline
           \text{Total}  & 3 & 8 & 43 & 290 & 2195 & 17822 & 151708 & 1336086 & *** \\\hline\hline
           \text{Root Color }1 & 1 & 3 & 17 & 118 & 909 & 7467 & 64084 & 567784 & A394162(n) \\\hline
           \text{Root Color }2 & 1 & 3 & 17 & 118 & 909 & 7467 & 64084 & 567784 & A394162(n) \\\hline
           \text{Root Color }3 & 1 & 2 & 9 & 54 & 377 & 2888 & 23540 & 200518 & A394163(n) \\\hline
        \end{array}
        \]
        $***$: Formula for Total is $t_A(n)=2\cdot A394162(n)+A394163(n)$.
    \end{center}
    
    \item[Nontrivial Functional Equations]\hspace*{\fill}

    \begin{itemize}
\item $2\gfa{1}^4-\gfa{1}^3+4x\gfa{1}^2-x\gfa{1}+x^2$

\item $2\gfa{2}^4-\gfa{2}^3+4x\gfa{2}^2-x\gfa{2}+x^2$

\item $\gfa{3}^4-2\gfa{3}^3+(4x + 1)\gfa{3}^2-3x\gfa{3}+2x^2$

\end{itemize}
    
\end{description}

\subsection*{Entry 71}
\begin{description}
    \item[Matrices]
    
    \[\begin{bmatrix}
1 & 1 & 1 \\
1 & 1 & 1 \\
1 & 1 & 0
\end{bmatrix}\]
    \item[Sequences]\hspace*{\fill}\vspace{-0.5cm}
    \begin{center}
        \[
        \begin{array}{|c||c|c|c|c|c|c|c|c||c|}\hline
           n  & 1 & 2 & 3 & 4 & 5 & 6 & 7 & 8  & \text{Formula}\\\hline\hline
           \text{Total}  & 3 & 8 & 44 & 304 & 2356 & 19576 & 170460 & 1535200 & ***\\\hline\hline
           \text{Root Color }1 & 1 & 3 & 17 & 119 & 929 & 7755 & 67745 & 611567 & A344553(n)\\\hline
           \text{Root Color }2 & 1 & 3 & 17 & 119 & 929 & 7755 & 67745 & 611567 & A344553(n)\\\hline
           \text{Root Color }3 & 1 & 2 & 10 & 66 & 498 & 4066 & 34970 & 312066 & A027307(n-1)\\\hline
        \end{array}
        \]
        $***$: Formula for Total is $t_A(n)=2\cdot A344553(n) + A027307(n-1)$.
    \end{center}
    
    \item[Nontrivial Functional Equations]\hspace*{\fill}

    \begin{itemize}
\item $4\gfa{1}^3-4\gfa{1}^2+(x + 1)\gfa{1}-x$

\item $4\gfa{2}^3-4\gfa{2}^2+(x + 1)\gfa{2}-x$

\item $\gfa{3}^3+x\gfa{3}^2-x\gfa{3}+x^2$

\end{itemize}
    \item[Discussion]
This is a special case of Family 2 (see Section~\ref{sec:fam3}), where $\ell = 2$ and $m=1$. The generating function identities follow from Theorem~\ref{thm:removesquare-gfs}. Moreover, Theorems~\ref{thm:removesquare},~\ref{thm:removesquare-alternative} and~\ref{thm:removesquare-alternative2} yield formulas for $t_A(n)$, $t_A^{(1)}(n) = t_A^{(2)}(n)$ and $t_A^{(3)}(n)$: for all $n \geq 2$, we have
\begin{align*}
  t_{A}(n) &= \frac{2}{n} \sum_{k=0}^{n-1} 2^{n-k}  \binom{n}{k} \binom{2n-3}{n-k-1} \\
  &= \frac{(-1)^n}{n(n-1)} \sum_{k=1}^{n} (-2)^k (k-1) \binom{n}{k} \binom{2n+k-2}{k-1} \\ 
  &= \frac{4}{n} \sum_{k=0}^{n-1} \binom{2n-3}{k} \binom{3n-k-3}{n-k-1}, \\
  t_{A}^{(1)}(n) &= \frac{1}{n} \sum_{k=0}^{n-1} 2^{n-k-1} \binom{n}{k} \binom{2n-2}{n-k-1} \\
  &= \frac{(-1)^{n-1}}{n} \sum_{k=1}^{n} (-2)^{k-1} \binom{n}{k} \binom{2n+k-2}{k-1} \\
  &= \frac{1}{n} \sum_{k=0}^{n-1} \binom{2n-2}{k} \binom{3n-k-2}{n-k-1}, \\
  t_{A}^{(3)}(n) &= \frac{1}{n-1} \sum_{k=1}^{n-1} 2^{n-k} \binom{n-1}{k-1} \binom{2n-2}{n-k-1}  \\
  &= \frac{(-1)^{n-1}}{n-1} \sum_{k=1}^{n-1} (-2)^k \binom{n-1}{k} \binom{2n+k-2}{k-1} \\
  &= \frac{2}{n-1} \sum_{k=0}^{n-2} \binom{2n-2}{k} \binom{3n-k-3}{n-k-2}.
\end{align*}

Some of these sum formulas appear already in the OEIS.

\end{description}

\subsection*{Entry 72}
\begin{description}
    \item[Matrices]
    
    \[\begin{bmatrix}
1 & 1 & 1 \\
1 & 1 & 1 \\
1 & 1 & 1
\end{bmatrix}\]
    \item[Sequences]\hspace*{\fill}\vspace{-0.5cm}
    \begin{center}
        \[
        \begin{array}{|c||c|c|c|c|c|c|c|c||c|}\hline
           n  & 1 & 2 & 3 & 4 & 5 & 6 & 7 & 8  & \text{Formula}\\\hline\hline
           \text{Total}^{h}  & 3 & 9 & 54 & 405 & 3402 & 30618 & 288684 & 2814669 & 
           3^{n}C_{n-1}\\\hline\hline
           \text{Root Color }1^h & 1 & 3 & 18 & 135 & 1134 & 10206 & 96228 & 938223 & 3^{n-1}C_{n-1}\\\hline
           \text{Root Color }2^h & 1 & 3 & 18 & 135 & 1134 & 10206 & 96228 & 938223 & 3^{n-1}C_{n-1}\\\hline
           \text{Root Color }3^h & 1 & 3 & 18 & 135 & 1134 & 10206 & 96228 & 938223 & 3^{n-1}C_{n-1}\\\hline
        \end{array}
        \]
    \end{center}

\item[Discussion] This follows from Theorem~\ref{prop:row-sums}, since all row sums are equal to $3$.
    
\end{description}

\end{document}